\documentclass[11pt]{amsart}

\usepackage{amsmath,amssymb,latexsym}
\usepackage{graphicx}
\usepackage{dsfont}
\usepackage[T1]{fontenc}
\usepackage{enumerate}
\usepackage{eqnarray}
\usepackage{mathtools}
\usepackage{url}
\usepackage{xcolor}
\usepackage[all,cmtip]{xy}
\usepackage{pb-diagram,pb-xy}
\usepackage[a4paper,top=3cm, bottom=3cm, inner=3.5cm,outer=2.7cm,nofoot,headsep=1.5cm]{geometry}
\providecommand{\mapsfrom}{\mathrel{\reflectbox{\ensuremath{\mapsto}}}}

\def\Ind#1#2{#1\setbox0=\hbox{$#1x$}\kern\wd0\hbox to 0pt{\hss$#1\mid$\hss}
\lower.9\ht0\hbox to 0pt{\hss$#1\smile$\hss}\kern\wd0}
\def\Notind#1#2{#1\setbox0=\hbox{$#1x$}\kern\wd0\hbox to 0pt{\mathchardef
\nn="3236\hss$#1\nn$\kern1.4\wd0\hss}\hbox to 0pt{\hss$#1\mid$\hss}\lower.9\ht0
\hbox to 0pt{\hss$#1\smile$\hss}\kern\wd0}
\def\indi{\mathop{\mathpalette\Ind{}}}

\theoremstyle{plain}
\newtheorem{theorem}{Theorem}[section]
\newtheorem{prop}[theorem]{Proposition}
\newtheorem{fact}[theorem]{Fact}
\newtheorem{lemma}[theorem]{Lemma}
\newtheorem{cor}[theorem]{Corollary}

\theoremstyle{definition}
\newtheorem{defn}[theorem]{Definition}
\newtheorem{remark}[theorem]{Remark}

\newtheorem{problem}[theorem]{Problem}
\newtheorem{expl}[theorem]{Example}
\newtheorem{conj}[theorem]{Conjecture}

\newtheorem*{claim*}{Claim}
\newtheorem{claim}[theorem]{Claim}

\newcommand{\tp}{\operatorname{tp}}
\newcommand{\M}{\mathcal{M}}
\newcommand{\C}{\mathbb{M}}

\newcommand{\Span}{\operatorname{Span}}
\newcommand{\ded}{\operatorname{ded}}
\newcommand{\cof}{\operatorname{cof}}
\newcommand{\qftp}{\operatorname{qftp}}
\newcommand{\acl}{\operatorname{acl}}
\newcommand{\dcl}{\operatorname{dcl}}
\newcommand{\RCF}{\operatorname{RCF}}
\newcommand{\Aut}{\operatorname{Aut}}
\newcommand{\ACVF}{\operatorname{ACVF}}

\newcommand{\pCF}{\operatorname{pCF}}
\newcommand{\val}{\operatorname{val}}
\newcommand{\opg}{\operatorname{opg}}
\newcommand{\op}{\operatorname{op}}
\newcommand{\alg}{\operatorname{alg}}
\newcommand{\Th}{\operatorname{Th}}
\newcommand{\Sk}{\operatorname{Sk}}

\newcommand{\Sym}{\operatorname{Sym}}
\newcommand{\ind}{\operatorname{\indi}}

\newcommand{\sign}{\operatorname{sign}}
\newcommand{\cha}{\operatorname{char}}
\newcommand{\K}{\mathbb{K}}
\newcommand{\F}{\mathbb{F}}

\renewcommand{\a}{\overline{a}}
\renewcommand{\b}{\overline{b}}
\renewcommand{\c}{\overline{c}}
\newcommand{\x}{\overline{x}}
\newcommand{\G}{\mathbb{G}}
\newcommand{\trd}{\operatorname{trd}}

\newcommand{\tuple}{\overline}
\newcommand{\lambdatup}{\tuple{\lambda}}
\newcommand{\alphatup}{\tuple{\alpha}}

\newcommand{\as}{\wp}
\font\myfont=cmr12 at 7pt

\title{On $n$-dependent groups and fields II}
\author[Artem Chernikov and Nadja Hempel]{Artem Chernikov and Nadja Hempel \\ \\ ({\myfont with an appendix by} Martin Bays)}

\begin{document}

\maketitle

\begin{abstract}
  We continue the study of $n$-dependent groups, fields and related structures, largely motivated by the conjecture that every $n$-dependent field is dependent. We provide evidence towards this conjecture by showing that every infinite $n$-dependent valued field of positive characteristic is henselian, obtaining a variant of Shelah's Henselianity Conjecture in this case and generalizing a recent result of Johnson for dependent fields. Additionally, we prove a result on intersections of type-definable connected components over generic sets of parameters in $n$-dependent groups, generalizing Shelah's absoluteness of $G^{00}$ in dependent theories and relative absoluteness of $G^{00}$ in $2$-dependent theories.  In an effort to clarify the scope of this conjecture, we provide new examples of strictly $2$-dependent fields with additional structure, showing that Granger's examples of non-degenerate bilinear forms over dependent fields are $2$-dependent.
   Along the way, we obtain some purely model-theoretic results of independent interest: we show that $n$-dependence is witnessed by formulas with all but one variable singletons; provide a type-counting criterion for $2$-dependence and use it to deduce $2$-dependence for compositions of dependent relations with arbitrary binary functions (the Composition Lemma); and show that an expansion of a geometric theory $T$ by a generic predicate is dependent if and only if it is $n$-dependent for some $n$, if and only if the algebraic closure in $T$ is disintegrated.
\end{abstract}

\section{Introduction}
A classical line of research in model theory, both pure and applied, aims to determine properties of algebraic structures, such as groups and fields, that satisfy certain model-theoretic tameness assumptions. This is analogous to the study of algebraic or Lie groups in algebraic or
differential geometry, but instead of considering groups definable in a specific structure like $\mathbb{C}$ or $\mathbb{R}$ one typically considers groups definable in a class of first-order structures with some restrictions on the complexity of their definable subsets. Some of the most striking applications of model theory are based
on a detailed understanding of definable groups in certain specific contexts of this kind (e.g.~Hrushovski's proof of the Mordell-Lang conjecture for function fields \cite{hrushovski1996mordell}  is based on the theory of stable  groups,  applied to groups definable in differentially closed and in separably closed fields). But even if one
is only interested in abstract classification of first-order structures, the study of definable groups unavoidably enters the picture (e.g.~through Zilber's work on totally categorical structures \cite{zilber1993uncountably} or  Hrushovski's theorem on unidimensional theories \cite{hrushovski1990unidimensional}). In the case of model-theoretically tame fields one often expects not only to deduce some of their general properties, but in fact to obtain an explicit algebraic classification.
Probably the first result of this type is Macintyre's proof that all $\aleph_0$-stable fields (roughly speaking, fields admitting a Zariski-like notion of dimension on their definable subsets) are algebraically closed \cite{macintyre1971omega}, generalized by Cherlin-Shelah to the larger class of superstable fields \cite{cherlin1980superstable}. Some of the longest standing conjectures in model theory that played a fundamental role in the development of the area are asking for characterizations of this type, e.g.~Podewski's conjecture whether all minimal fields are algebraically closed \cite{podewski1973minimale}, or the stable field conjecture asking if all stable fields are separably closed (see e.g.~\cite{johnson2020etale}), and we discuss some further examples below.

 In this article we continue the study of groups, fields and related structures satisfying a model-theoretic tameness condition called \emph{$n$-dependence}, for $n\in \mathbb{N}$, initiated in \cite{hempel2016n} and continued in \cite{chernikov2019mekler}. The class of $n$-dependent theories was introduced by Shelah in \cite{shelah2014strongly}, with the $1$-dependent (or just \emph{dependent}) case corresponding to the class of NIP theories that has attracted a lot of attention recently (see e.g.~\cite{simon2015guide} for an introduction to the area). Basic properties of $n$-dependent theories are investigated in \cite{chernikov2014n}. Roughly speaking, $n$-dependence of a theory guarantees that the edge relation of an infinite generic $(n+1)$-hypergraph is not definable in its models (see Definition \ref{def: k-dependence}). For $n \geq 2$, we say that a theory is \emph{strictly $n$-dependent} if it is $n$-dependent, but not $(n-1)$-dependent.

This paper is largely, but not exclusively, motivated by the following conjecture.
\begin{conj}\label{conj: main conj}
	There are no strictly $n$-dependent fields for $n \geq 2$ (in the pure ring language).
\end{conj}
We expect that the same should hold for fields expanded with some  natural operators such as for example derivations or valuations. We also expect a generalization to type-definable fields: every field type-definable in an $n$-dependent structure is isomorphic to a field type-definable in a dependent structure.
Some initial evidence towards Conjecture \ref{conj: main conj} is given by the results in \cite{hempel2016n}: every infinite $n$-dependent field is Artin-Schreier closed (generalizing \cite{kaplan2011artin} for $n=1$); every non-separable PAC (i.e.~Pseudo-Algebraically Closed) field is not $n$-dependent for any $n$ (generalizing \cite{duret1980corps}). In particular, for fields with (super-)simple theories, our Conjecture \ref{conj: main conj} follows from the well-known conjecture that all such fields are (bounded) PAC (see e.g.~\cite{MR1653312}). On the other hand, combined with Shelah's conjectures on dependent fields discussed below, Conjecture \ref{conj: main conj} leads to a complete classification of $n$-dependent fields.

In this paper we obtain some new results about $n$-dependent groups and fields, in part providing further evidence for Conjecture \ref{conj: main conj} and clarifying its scope, and in part generalizing the known results about dependent or $2$-dependent structures.
First, we prove Shelah's Henselianity Conjecture for $n$-dependent valued fields of positive characteristic in Section \ref{sec: valued fields} (this generalizes Johnson \cite{johnson2021dpA} for $n=1$ and, towards   Conjecture \ref{conj: main conj}, demonstrates that a known property of dependent fields also holds for all $n$-dependent fields). Additionally, we establish a result on intersections of type-definable connected components over generic sets of parameters in $n$-dependent groups in Section \ref{sec: n-dep conn comp}, generalizing Shelah's theorems on absoluteness of $G^{00}$ in dependent theories and relative absoluteness of $G^{00}$ for $2$-dependent theories. While we do not have any direct application  of this result towards Conjecture \ref{conj: main conj} at the moment, the $n=1$ case is a fundamental property of dependent groups and is used extensively in Johnson's classification of dependent fields of finite dp-rank \cite{johnson2021dpA, johnson2021dpB}, so we expect it to be useful in the future study of Conjecture \ref{conj: main conj} and its aforementioned generalization to type-definable fields in $n$-dependent structures.
 Second, we provide new examples of strictly  $2$-dependent fields with additional structure by showing that Granger's examples of non-degenerate bilinear forms over dependent fields are strictly $2$-dependent in Section \ref{sec: Granger} (demonstrating in particular the necessity of the pure ring language assumption in Conjecture \ref{conj: main conj}).
Our proof of this relies on establishing first some general results on $n$-dependent theories, possibly of independent interest: a reduction of the $n$-dependence of a theory to formulas with all but one of its variables singletons (Section \ref{sec: Prelims}), a type-counting criterion for $2$-dependence and the Composition Lemma showing $2$-dependence of compositions of dependent relations with binary functions (Section \ref{sec: functions preserve 2-dep}). 
And third, we show in Section \ref{sec: adding rand pred} that an expansion of a geometric theory $T$ by a generic predicate is dependent if and only if it is $n$-dependent for some $n$, if and only if the algebraic closure in $T$ is disintegrated (generalizing the $n=1$  case from  \cite{chatzidakis1998generic}). While not directly related to Conjecture \ref{conj: main conj}, this gives an example of a class of structures (expansions of geometric structures by a generic predicate) for which $n$-dependence is equivalent to dependence, the behavior predicted for fields by Conjecture \ref{conj: main conj}, and the authors hope that some of the arguments might be useful in the future for the original question.

In the rest of the introduction, we discuss these results in further detail and overview the structure of the paper.

One of the results in \cite{chernikov2014n} gives a characterization of $n$-dependence in terms of generalized indiscernibles (indexed by ordered random partite $n$-hypergraphs) and demonstrates, using this characterization, that in order to verify $n$-dependence of a theory it is enough to check that every formula $\varphi(x; y_1, \ldots,y_n)$ with at least one of the tuples of variables $x,y_1, \ldots, y_n$ singleton is $n$-dependent (generalizing the well-known theorem of Shelah for dependent theories). In Section \ref{sec: Prelims} we refine and generalize some of these results allowing indexing structures of larger cardinalities and obtaining a better reduction to singletons: a theory $T$ is $n$-dependent if and only if every formula $\varphi(x,y_1, \ldots, y_n)$ such that all but at most one of the tuples $x, y_1, \ldots, y_n$ are singletons is $n$-dependent (Theorem \ref{thm: better reduction to singletons}).

In Section \ref{sec: valued fields} (which is self-contained except for the results in Appendix \ref{sec: app Bays}) we obtain further evidence towards Conjecture \ref{conj: main conj} in the case of valued fields. The question of classifying dependent (valued) fields is currently an
active area of research motivated by various versions of Shelah's Conjecture, which in particular predicts that every infinite dependent valued field is henselian.
A recent result of Johnson \cite{johnson2021dpA, johnson2021dpB} confirms this for valued fields of positive characteristic. In Theorem \ref{thm: main hens} we generalize this by showing that every $n$-dependent valued field of positive characteristic is henselian, for arbitrary $n$.
As in Johnson's proof, the theorem is deduced by showing that any two valuations on an infinite $n$-dependent field of positive characteristic must be comparable. In the case of $n=1$, this lemma can be quickly obtained using Artin-Schreier closedness of dependent fields and absoluteness of the connected component $G^{00}$ (in fact, an application of Baldwin-Saxl is sufficient). However, replacing the absolute connected component with a weaker condition for intersections of uniformly definable families of subgroups available in $n$-dependent theories (Proposition \ref{prop: chain cond for mult fam of subgrps}) requires a detailed analysis of the effect that the isomorphism for special linear groups from Kaplan-Scanlon-Wagner \cite{kaplan2011artin} has on multiple valuations. We are able to carry it out, relying in particular on the explicit description of this isomorphism given by Bays in Appendix \ref{sec: app Bays}. Concerning the (open) case of characteristic $0$, in Section \ref{sec: gen multi-ordered} we observe that the model completions of multi-ordered and multi-valued fields with at least two orders (respectively, valuations) as studied in \cite{LouThesis, johnson2016fun} are not $n$-dependent for any $n$.

Given a definable group $G$ and a small set of parameters $A$ (see Section \ref{sec: basic model theory defs} for the definitions of ``small'', ``saturated'', etc.), we denote by $G^{00}_A$ the intersection of all subgroups of $G$ of bounded index type-definable over $A$ (see Section \ref{sec: Connected components} for more details). A crucial fact about definable groups in dependent theories, due to Shelah,   is that for every small set $A$ one has $G^{00}_A = G^{00}_{\emptyset}$ \cite{shelah2008minimal}. This can be viewed as an infinitary analog of the Baldwin-Saxl condition on intersections of uniformly definable families of subgroups in dependent theories \cite{baldwin1976logical}. In \cite{MR3666349}, Shelah established the following result for groups definable in $2$-dependent theories: let $\M$ be a sufficiently saturated model, and let $b$ be a finite tuple in $\C$ (and not contained in $\M$ in the case of interest), then $G^{00}_{\M\cup b} = G^{00}_{\M} \cap G^{00}_{Cb}$ for some small set $C \subseteq \M$.
In Section \ref{sec: n-dep conn comp} (which is self-contained), we generalize this result from $2$-dependent groups to $n$-dependent groups. Specifically, we show that if $T$ is $n$-dependent and $G = G(\mathbb{M})$ is a type-definable group (over $\emptyset)$, then for any small model $M$ and finite tuples $b_1, \ldots, b_{n-1}$ sufficiently independent over $M$ in an appropriate sense, we have that 
$$G^{00}_{\M \cup b_1 \cup \dots \cup  b_{n-1}} = \bigcap_{i=1, \dots, n-1} G_{\M \cup  b_1 \cup \ldots \cup b_{i-1} \cup b_{i+1} \cup \ldots \cup  b_{n-1}}^{00} \cap G^{00}_{C \cup  b_1\cup \dots\cup  b_{n-1}}$$ 
for some $C \subseteq \M$ of absolutely bounded size (Theorem \ref{thm: connected comp type-def main} and Corollary \ref{cor: G00 2-dep}). In other words, the intersection of all subgroups of $G$ of bounded index type-definable over $\M \cup b_1 \cup \dots \cup  b_{n-1}$ is already given by the intersection of a (potentially) smaller collection of subgroups containing only boundedly many groups whose definitions involve all $n-1$ of the parameters $b_1, \ldots, b_{n-1}$ at the same time.
Our independence assumption on the parameters holds trivially in the cases $n=1,2$ giving the aforementioned results for dependent and $2$-dependent groups, and in general can be achieved assuming that the $b_i$'s appear as the vertices of an amalgamation diagram with respect to the independence relation of being a $\kappa$-coheir (see Definition \ref{def: generic position} for the precise definition of our independence assumption). While this result has no direct applications to Conjecture \ref{conj: main conj} at the moment (in our proof of Theorem \ref{thm: main hens} we only needed a chain condition for uniformly definable families of subgroups from Proposition \ref{prop: chain cond for mult fam of subgrps}), we expect that it will be useful in the  future, in particular  for the aforementioned variant of Conjecture \ref{conj: main conj} for type-definable fields.

Next, we consider the limitations of Conjecture \ref{conj: main conj} (in terms of the additional structure allowed on the field) and try to place it in a more general model-theoretic setting. In \cite{hempel2016n} it was observed that the theory of a bilinear form on an infinite dimensional vector space over a \emph{finite} field is strictly $2$-dependent. In fact, all of the previously known ``algebraic'' examples of strictly $n$-dependent theories with $n \geq 2$ tend to look like multi-linear forms over finite fields. E.g., smoothly approximable structures are $2$-dependent and coordinatizable via bilinear forms over finite fields \cite{cherlin2003finite}; and the strictly $n$-dependent pure groups constructed in \cite{chernikov2019mekler} using Mekler's construction are essentially of this form as well, using Baudisch's interpretation of Mekler's construction in alternating bilinear maps \cite{baudisch2002mekler}.
In Section \ref{sec: Granger} we show that one can replace finite fields by arbitrary dependent fields in these examples. Namely, we investigate $n$-dependence for theories of bilinear forms on vector spaces with a separate sort for the field, in the sense of Granger \cite{granger1999stability}. We show that all such theories are $2$-dependent assuming that the field is dependent, and that the assumption of dependence is necessary (see Theorem \ref{thm: Granger}). Combined with the fact that the intersection conditions on the connected components discussed above resemble modular behavior in the $2$-dependent case, this leads one to speculate that 
$n$-dependence of a theory might imply some form of ``linearity relative to the dependent part''. While formulating this precisely appears difficult at the moment, we view Conjecture \ref{conj: main conj} as a specific instance of this general principle.
Our proof of Theorem \ref{thm: Granger} relies on the criterion for $n$-dependence  in terms of generalized indiscernibles from Section \ref{sec: Prelims}, and on some additional purely model-theoretic results contained in Section \ref{sec: functions preserve 2-dep} that we now describe.

In \cite{chernikov2014n} a generalization of the Sauer-Shelah lemma to $n$-dependent formulas is given, in particular demonstrating that a formula $\varphi(x;y_1, \ldots, y_n)$ is $n$-dependent if and only if the number of $\varphi$-types over an arbitrary large finite set $A$ of parameters is bounded by $2^{|A|^{n-\varepsilon}}$ for some $\varepsilon = \varepsilon(\varphi) \in \mathbb{R}_{>0}$.
Concerning the number of types over infinite sets of parameters, a well-known theorem of Shelah \cite[Theorem II.4.11]{shelah1990classification} shows that if $\varphi(x,y)$ is dependent, then the number of $\varphi$-types over an infinite set of parameters of size $\kappa$ is at most $\ded(\kappa)$, where $\ded(\kappa)$ is the supremum over the number of Dedekind cuts in a linear order of cardinality $\kappa$. In Section \ref{sec: counting types}, we show that a theory is 2-dependent if and only if the following type counting criterion is satisfied.  Let $c$ be a finite tuple and $I$ an indiscernible sequence of size $\kappa$. Then the number of types over $Ic$ that are realized cofinally in a sequence mutually indiscernible to $I$ is bounded by $\ded(\kappa)$ (see Proposition \ref{prop: few types on a tail} for details).
In Section \ref{sec: Comp Lemma} this criterion is combined with set-theoretic absoluteness to obtain a more general version of the following finitary combinatorial statement of independent interest, the ``Composition Lemma''. Let $R \subseteq M^3$ be a ternary relation definable in a dependent structure, and let  $f: M^2 \to M$ be an arbitrary (not necessarily definable) function. Then the ternary relation $R'(x,y,z) = R(f(x,y), f(x,z), f(y,z))$ is $2$-dependent (Theorem \ref{thm: functions into stable are 2-dep }). It is interesting to compare this to a line of results around Hilbert's 13th problem demonstrating that a function of arbitrary arity can be expressed as a finite composition of binary functions (in the category of all functions, or of continuous functions on $\mathbb{R}$ --- a celebrated theorem of Kolmogorov and Arnold \cite{arnol1957functions}).
Our result can be viewed as saying that in such presentations, the outer relation is necessarily ``fractal-like''. It is worth mentioning that some other connections of $n$-dependence to finitary hypergraph combinatorics are considered in \cite{terry} (in connection to hypergraph growth) and in \cite{ChernikovTowsner} (which establishes a strong regularity lemma for $n$-dependent hypergraphs demonstrating that every $n$-dependent relation of arbitrarily high arity can be approximated by relations of arity $n$ up to measure $0$).
The Composition Lemma is applied in the proof of Theorem \ref{thm: Granger} to conclude $2$-dependence of certain basic atomic formulas involving the ``generic'' binary function given by the bilinear form.

Finally, in Section \ref{sec: adding rand pred} (which only depends on Section \ref{sec: Prelims}) we consider $n$-dependence for expansions of geometric theories by generic predicates and relations of higher arity. In particular, we show that an expansion of a geometric theory $T$ by a generic predicate is dependent if and only if it is $n$-dependent for some $n$, if and only if the algebraic closure in $T$ is disintegrated (Corollary \ref{cor: n-dep iff acl disint}). This generalizes the corresponding result for dependence in \cite{chatzidakis1998generic}. In Remark \ref{rem: counterex to ChPil} we give an example showing that geometricity of $T$ (or some other additional assumption) is necessary even for $n=1$. Our proof for relations of higher arity relies on an infinitary generalization of Hrushovski's observation \cite{hrushovski1991pseudo} that the random $n$-ary hypergraph is not a finite Boolean combination of relations of arity $n-1$.
\subsection*{Acknowledgements}
We are grateful to the referee for many valuable suggestions on improving the paper.
We thank Martin Bays for agreeing to include his result as an appendix.  We thank Itay Kaplan for helpful discussions concerning Section \ref{sec: type-def connected comp};  Nick Ramsey for suggesting the question considered in Section \ref{sec: Granger} and for a helpful discussion concerning Section \ref{sec: functions preserve 2-dep}; Kota Takeuchi for a discussion concerning Section \ref{sec: counting types}; Erik Walsberg for a discussion concerning Section \ref{sec: adding rand pred}.

Chernikov and Hempel were partially supported by the NSF Research Grant DMS-1600796 and by the NSF CAREER grant DMS-1651321. Bays was supported in part by the Deutsche Forschungsgemeinschaft (DFG,
German Research Foundation) under Germany's Excellence Strategy EXC 2044
–390685587, Mathematics Münster: Dynamics–Geometry–Structure.
\section{Preliminaries and some general lemmas on $n$-dependence}\label{sec: Prelims}
\subsection{Notation}\label{sec: basic model theory defs}
We will be following standard model-theoretic notation, and refer to e.g.~\cite{tent2012course, marker2006model} for an introduction to model theory. Usually $T$ will denote a complete first-order theory in a language $\mathcal{L}$ (possibly multisorted), and $\mathcal{M}, \mathcal{N}$ will be first-order $\mathcal{L}$-structures.  We recall that, given an infinite cardinal $\kappa$, a first-order structure $\mathcal{M}$ is $\kappa$-saturated when for every set of formulas with parameters from a subset of $\mathcal{M}$ of size $<\kappa$ in a fixed finite tuple of variables for which every finite subset of these formulas is satisfied by a tuple in $\mathcal{M}$, the whole set is satisfied by a tuple in $\mathcal{M}$. Given a tuple $a$ of elements in $\mathcal{M}$ and a subset $B$ of $\mathcal{M}$, we denote by $\tp(a/B)$ the complete type of $a$ over $b$ (i.e.~the collection of all formulas with parameters in $B$ satisfied by $a$). We write $a \equiv_B a'$ when $\tp(a/B) = \tp(a'/B)$. A structure $\mathcal{M}$ is $\kappa$-homogeneous if for any two finite tuples $a,a'$ and a set of parameters $B$ in $\mathcal{M}$ with $|B|<\kappa$ for which $a \equiv_B a'$, there exists an automorphism $\sigma \in \Aut(\mathcal{M}/B)$ of $\mathcal{M}$ fixing $B$ pointwise and sending $a$ to $a'$.
We let $\mathbb{M} \models T$ be a \emph{monster model} of $T$, i.e.~a $\kappa$-saturated and $\kappa$-homogeneous model of $T$ for some sufficiently large strongly inaccessible cardinal $\kappa = \kappa(\mathbb{M})$.  
Once $\mathbb{M}$ is fixed, as usual we will say that a subset of $\mathbb{M}$ (or just an arbitrary set/structure)  is \emph{small} if it has cardinality smaller than $\kappa(\mathbb{M})$.  Given an $\mathcal{L}$-structure $\mathcal{M} \models T$ and a single variable (of a prescribed sort of the language $\mathcal{L}$), we write $\mathcal{M}_x$ to denote the corresponding sort of $\mathcal{M}$. If $x = (x_1, \ldots, x_n)$ is a finite tuple of variables, we let $\mathcal{M}_x := \prod_{i=1}^n \mathcal{M}_{x_i}$.

\subsection{$N$-dependent formulas and their basic properties}
We begin with the definition of $n$-dependent theories and some of the basic properties of $n$-dependent formulas and theories.

\begin{defn}\rm \label{def: k-dependence}
A partitioned formula $\varphi\left(x;y_{1},\ldots,y_{n}\right)$ has the \emph{$n$-independence property} (with respect to a theory $T$), if in some model of $T$ there is a sequence of tuples $\left(a_{1,i},\ldots,a_{n,i}\right)_{i\in\omega}$
such that for every $s\subseteq\omega^{n}$ there is a tuple $b_{s}$ with the following property:
\[
\models\varphi\left(b_{s};a_{1,i_{1}},\ldots,a_{n,i_{n}}\right)\Leftrightarrow\left(i_{1},\ldots,i_{n}\right)\in s\mbox{.}
\]
Otherwise we say that $\varphi\left(x,y_{1},\ldots,y_{n}\right)$ is \emph{$n$-dependent}.
A theory is \emph{$n$-dependent} if it implies that every formula is
$n$-dependent. 

We simply say that a theory is \emph{dependent} if it is $1$-dependent. A structure $\mathcal{M}$ is $n$-dependent if $\Th(\mathcal{M})$ is $n$-dependent.
\end{defn}

\begin{fact} \label{fac: props of n-dependent formulas} \cite[Proposition 6.5]{chernikov2014n}
\begin{enumerate}
	\item Let $\varphi(x,y_1, \ldots, y_n)$ and $\psi(x,y_1, \ldots, y_n)$ be $n$-dependent formulas. Then $\neg \varphi$, $\varphi \land \psi$ and $\varphi \lor \psi$ are $n$-dependent.
	\item Let $\varphi(x,y_1, \ldots, y_n)$ be a formula. Suppose that $(w, z_1, \ldots, z_n)$ is any permutation of the tuple $(x,y_1, \ldots, y_n)$. Then $\psi(w, z_1, \ldots, z_n) := \varphi(x,y_1, \ldots, y_n)$ is $n$-dependent if and only if $\varphi(x,y_1, \ldots, y_n)$ is $n$-dependent.
	\item A theory $T$ is $n$-dependent if and only if every formula $\varphi(x,y_1, \ldots, y_n)$ with $|x|=1$ is $n$-dependent (see also Section \ref{sec: red to singletons}).
\end{enumerate}
	\end{fact}

\subsection{Generalized indiscernibles}

We will often use a characterizations of $n$-dependence from \cite{chernikov2014n} in terms of generalized indiscernibles.

\begin{defn}
Fix a language $\mathcal{L}_{\opg}^n=\{R_n(x_1,\ldots ,x_{n}),<, P_1(x),\ldots , P_{n}(x)\}$. 
An \emph{ordered $n$-partite hypergraph} is an $\mathcal{L}^{n}_{\opg}$-structure $ \mathcal{A} = \left(A; <, R_n, P_1, \ldots , P_{n} \right)$ such that:
\begin{enumerate}
\item
$A$ is the disjoint union $P^{\mathcal{A}}_1 \sqcup\ldots \sqcup P^{\mathcal{A}}_{n}$,
\item $R_n^{\mathcal{A}}$ is a symmetric relation such that if $(a_1,\ldots ,a_{n})\in R_n^{\mathcal{A}}$ then $P^{\mathcal{A}}_i\cap \{a_1, \ldots, a_{n}\}$ is a singleton for every $1 \leq i \leq n$,
\item
$<^{\mathcal{A}}$ is a linear ordering on $A$ with $P^{\mathcal{A}}_1 <\ldots <P^{\mathcal{A}}_{n}$.
\end{enumerate}

\end{defn}

\begin{fact} \label{fac: Gnp} 
\begin{enumerate}
	\item \cite[Proposition A.5]{chernikov2014n} Let $\mathcal{K}$ be the class of all finite ordered $n$-partite hypergraphs. Then $\mathcal{K}$ is a Fra\"{i}ss\'e class, and its limit is called the \emph{generic ordered $n$-partite hypergraph}, denoted by $G_{n,p}$.
	\item \cite[Remark 4.5]{chernikov2014n}  An ordered $n$-partite hypergraph $\mathcal{A}$ is a model of $\Th(G_{n,p})$ if and only if:
\begin{itemize}
\item
$(P^{\mathcal{A}}_i, <)$ is a dense linear order without endpoints for each $1 \leq i \leq n$,
\item
for every $1 \leq j \leq n$, finite disjoint sets 
$A_0,A_1\subset {\prod_{1 \leq i \leq n, i\neq j}P^{\mathcal{A}}_i}$
 and
$b_0<b_1\in P^{\mathcal{A}}_j$, there is some $b \in P^{\mathcal{A}}_j$ such that $b_0<b<b_1$ and: $R_n(b,\bar{a})$ holds for every $\bar{a} \in A_0$ and $\neg R_n(b, \bar{a})$ holds for every $\bar{a} \in A_1$.
\end{itemize}
\end{enumerate}
 \end{fact}
\noindent We denote by $O_{n,p}$ the reduct of $G_{n,p}$ to the language $\mathcal{L}^{n}_{\op} = \{<, P_1(x),\ldots , P_{n}(x)\}$.

\begin{remark}\label{rem : induced copy}
	It is easy to see from the axiomatization in Fact \ref{fac: Gnp}(2) that given $G_{n,p}$ and any non-empty intervals $I_t \subseteq P_t$ for $t= 1, \ldots, n$, the set $I_1 \times \ldots \times I_n$ contains an induced copy of $G_{n,p}$.
\end{remark}

\begin{defn}
Let $T$ be a theory in a language $\mathcal{L}$, and let $\C$ be a monster
model of $T$.
\begin{enumerate}
\item Let $I$ be a structure in the language $\mathcal{L}_0$.
We say that $\bar{a}=\left(a_{i}\right)_{i\in I}$, 
with $a_{i}$ a tuple in $\C$, is \emph{$I$-indiscernible} over a set of parameters $C \subseteq \C$ if for
all $n\in\omega$ and all $i_{0},\ldots,i_{n}$ and $j_{0},\ldots,j_{n}$
from $I$ we have:
$$
\qftp_{\mathcal{L}_0}\left(i_{0}, \ldots, i_{n}\right)=\qftp_{\mathcal{L}_0}\left(j_{0},\ldots,j_{n}\right)
\Rightarrow $$
$$\tp_{\mathcal{L}}\left(a_{i_{0}}, \ldots, a_{i_{n}} / C\right)=\tp_{\mathcal{L}}\left(a_{j_{0}}, \ldots, a_{j_{n}} / C \right)
.$$
%For any $L_1\subseteq L_0$, $(a_i)_{i\in I}$ is said to be $L_1$-indiscernible if it is $(I\restriction L_1)$-indiscernible.
\item For $\mathcal{L}_0$-structures $I$ and $J$, we say that $\left(b_{i}\right)_{i\in J}$
is \emph{based on} $\left(a_{i}\right)_{i\in I}$ over a set of parameters $C \subseteq \C$ if for any finite
set $\Delta$ of $\mathcal{L}(C)$-formulas,  and for any finite tuple $\left(j_{0},\ldots,j_{n}\right)$
from $J$ there is a tuple $\left(i_{0},\ldots,i_{n}\right)$ from
$I$ such that:

\begin{itemize}
%change -- L' by L_0
\item $\qftp_{\mathcal{L}_0}\left(j_{0},\ldots,j_{n}\right)=\qftp_{\mathcal{L}_0}\left(i_{0},\ldots,i_{n}\right)$ and
\item $\tp_{\Delta}\left(b_{j_{0}},\ldots,b_{j_{n}}\right)=\tp_{\Delta}\left(a_{i_{0}},\ldots,a_{i_{n}}\right)$.
\end{itemize}
\end{enumerate}
\end{defn}

The following general fact is used to find $G_{n,p}$-indiscernibles.

\begin{fact}\cite[Corollary 4.8]{chernikov2014n}\label{fac: random hypergraph indiscernibles exist}
Let $C \subseteq \mathbb{M}$ be a small set of parameters.
\begin{enumerate}
\item For any $n \in \omega$ and $\bar{a}=\left(a_{g}\right)_{g\in O_{n,p}}$, there is some $\left(b_{g}\right)_{g\in O_{n,p}}$ which is $O_{n,p}$-indiscernible over $C$ and is based on $\bar{a}$ over $C$.
\item For any $n \in \omega$ and $\bar{a}=\left(a_{g}\right)_{g\in G_{n,p}}$, there is some $\left(b_{g}\right)_{g\in G_{n,p}}$ which is $G_{n,p}$-indiscernible over $C$ and is based on $\bar{a}$ over $C$.
\end{enumerate}
\end{fact}
%
%\begin{fact} \cite[Corollary 4.8]{chernikov2014n} \label{fac: extracting indisc}
%	Let $C$ be a small set of parameters. For any tuples $(a_i, b_j : i,j \in \mathbb{Q})$ there is some $(a'_i, b'_j : i,j \in \mathbb{Q})$ which is $G_{2,p}$-indiscernible (respectively, $O_{2,p}$-indiscernible) over $C$ and is \emph{based} on $(a_i, b_j)$ over $C$ (that is, for any $n \in \omega$, any finite set $\Delta$ of $\mathcal{L}(C)$-formulas and for any finite tuples $\bar{i} = (i_0, \ldots, i_n), \bar{j}=(j_0, \ldots, j_n)$ from $\mathbb{Q}$ there are tuples $\bar{i}',\bar{j}'$ in $\mathbb{Q}$ of the same length such that $\qftp_{G_{2,p}}(\bar{i}, \bar{j}) = \qftp_{G_{2,p}}(\bar{i}', \bar{j}')$ (respectively, $\qftp_{O_{2,p}}(\bar{i}, \bar{j}) = \qftp_{O_{2,p}}(\bar{i}', \bar{j}')$) and $\tp_{\Delta}((a'_{i_\alpha}, b'_{i_\alpha} : 1 \leq \alpha \leq n)) = \tp_{\Delta}((a_{i'_\alpha}, b_{i'_\alpha} : 1 \leq \alpha \leq n))$).
%\end{fact}

%\begin{fact}\label{fac: 2-dep criterion}\cite[Lemma 6.2 + Theorem 6.4]{chernikov2014n} The following are equivalent:
%\begin{enumerate}
%	\item $T$ is not $2$-dependent;
%	\item there is a \emph{singleton} $c$, finite tuples $(a_i, b_j : i,j \in \mathbb{Q})$ and a formula $\varphi(x;\bar{y}, \bar{z})$ such that $(a_i,b_j : i,j \in \mathbb{Q})$ is $O_{2,p}$-indiscernible and $\varphi(c; a_i, b_j)$ holds if and only if $ G_{2,p} \models R(i,j)$ for all $i,j \in \mathbb{Q}$;
%	\item same as (2), but with $c$ a \emph{tuple} of arbitrary finite length.
%\end{enumerate}
% 
%\end{fact}

Using this, we can characterize $n$-dependence of a formula as follows.

\begin{prop}\label{prop: indisc witness to IPn}The following are equivalent, in any theory $T$.
\begin{enumerate}
	\item $\varphi(x;y_1, \ldots, y_n)$ is not $n$-dependent.
	\item There are tuples $b$ and $(a_g)_{g \in G_{n,p}}$ such that
	\begin{enumerate}
	\item $(a_g)_{g \in G_{n,p}}$ is $O_{n,p}$-indiscernible over $\emptyset$ and $G_{n,p}$-indiscernible over $b$;
	\item $\models \varphi(b;a_{g_1}, \ldots, a_{g_{n}}) \iff G_{n,p}\models R_n(g_1, \ldots, g_n)$, for all $g_i 
	\in P_i$.
	\end{enumerate}
	\item (2) holds for any small $G'_{n,p} \equiv G_{n,p}$ in the place of $G_{n,p}$.
	\end{enumerate}	
\end{prop}
\begin{proof}
The equivalence of (1) and (2) is the equivalence of (a) and (c) in \cite[Lemma 6.2]{chernikov2014n}, and (3) implies (2) is obvious. Given a witness to (2), we can find a witness to (3) by compactness as every finite substructure of $G'_{n,p}$ appears as a finite substructure of $G_{n,p}$.	
\end{proof}

Additionally, we have the following ``formula-free'' characterization of $n$-dependence of a theory.
\begin{prop}\label{prop: char of NIP_k by preserving indisc}
Let $T$ be a complete theory and let $\mathbb{M} \models T$ be a monster model.
Then for any $n \in \mathbb{N}$, the following are equivalent:
\begin{enumerate}
\item $T$ is $n$-dependent.
\item For any $\left(a_{g}\right)_{g\in G_{n,p}}$ and $b$ with $a_g, b$ finite tuples in $\mathbb{M}$,
if $\left(a_{g}\right)_{g\in G_{n,p}}$ is $G_{n,p}$-indiscernible over
$b$ and $O_{n,p}$-indiscernible (over $\emptyset$), then it is $O_{n,p}$-indiscernible over $b$.
\item For any small $G'_{n,p} \equiv G_{n,p}$, $\left(a_{g}\right)_{g\in G'_{n,p}}$ and $b$,
if $\left(a_{g}\right)_{g\in G'_{n,p}}$ is $G'_{n,p}$-indiscernible over
$b$ and $O'_{n,p}$-indiscernible, then it is $O'_{n,p}$-indiscernible over $b$ (where $O'_{n,p}$ is the $\mathcal{L}^n_{\textrm{op}}$-reduct of $G'_{n,p}$).
\end{enumerate}
\end{prop}
\begin{proof}
	The equivalence of (1) and (2) is \cite[Proposition 6.3]{chernikov2014n}, (3) implies (2) is obvious, and we show that (2) implies (3).
	Assume that (3) fails, i.e.~there exist some $G'_{n,p} \equiv G_{n,p}$, $\left(a_{g}\right)_{g\in G'_{n,p}}$ and $b$ small tuples such that $\left(a_{g}\right)_{g\in G'_{n,p}}$ is $O'_{n,p}$-indiscernible, $G'_{n,p}$-indiscernible over $b$, but not $O'_{n,p}$-indiscernible over $b$. By definition, this is witnessed by some finite set of formulas and some finite set of indices from $G'_{n,p}$. Restricting all of the $a_g$'s and $b$ to the corresponding subtuples appearing in those formulas, we may assume that (3) fails with all of $b$ and $a_g$ finite. Moreover, we can choose a countable elementary submodel of $G'_{n,p}$ containing all of the indices witnessing failure of indiscernibility. It is isomorphic to $G_{n,p}$ by $\aleph_0$-categoricity of $\Th(G_{n,p})$, hence restricting $(a_g)_{g \in G'_{n,p}}$ to the corresponding set of indices we get a failure of (2).
\end{proof}

\subsection{Improved reduction to singletons}\label{sec: red to singletons}
In this section we will improve Fact \ref{fac: props of n-dependent formulas}(3) by showing that $T$ is $n$-dependent if and only if every formula in which  all but at most one of its variables are singletons is $n$-dependent (as opposed to ``at least one of the variables  is a singleton'' as in Fact \ref{fac: props of n-dependent formulas}(3)).

To do so, we need some auxiliary results. First we refine the equivalence of (1) and (2) in Proposition  \ref{prop: char of NIP_k by preserving indisc} making explicit the correspondence between the variables of a formula that is not $n$-dependent and the sorts of the tuples in a generalized indiscernible witnessing it.
In the proof below we are following the proof in \cite[Proposition 6.3]{chernikov2014n} with some modifications.
\begin{prop}\label{prop: controlling all but one}
	Fix $n \geq 1$ and let $x, y_1, \ldots, y_{n-1}$ be some fixed finite tuples of variables. The following are equivalent:
	\begin{enumerate}
		\item There exists some $(a_g)_{g \in G_{n,p}}$ with $a_g \in \mathbb{M}_{y_i}$ for all $g \in P_i$, $1 \leq i \leq n-1$ and $a_g \in \mathbb{M}_{y'_n}$ for some finite tuple of variables $y'_n$ and all $g \in P_n$, and $b \in \mathbb{M}_{x}$ such that $(a_g)_{g \in G_{n,p}}$ is $G_{n,p}$-indiscernible over $b$ and $O_{n,p}$-indiscernible over $\emptyset$, but is not $O_{n,p}$-indiscernible over $b$.
		\item There exists some formula $\varphi(x, y_1, \ldots, y_{n-1}, y''_n)$ which is not $n$-dependent, with $y''_n$ some finite tuple of variables.
	\end{enumerate}
\end{prop}

\begin{proof}
	We obtain immediately that  (2) implies (1) by the implication $(1) \Rightarrow (2)$ in Proposition \ref{prop: indisc witness to IPn}, with $y'_n = y''_n$.
	
	To prove  (1) implies (2), let $(a_g)_{g \in G_{n,p}}$ and $b$ be as given by (1). We define $a'_g := a_g$ for all $g \in \bigcup_{1 \leq i \leq n-1} P_i$ and $a'_g := a_g b$ for all $g \in P_n$. We have that $(a'_g)_{g \in G_{n,p}}$ is not $O_{n,p}$-indiscernible, but is $G_{n,p}$-indiscernible (over $\emptyset$), from the corresponding properties of $(a_g)_{g \in G_{n,p}}$ over $b$. Namely, by assumption there are some finite subsets $V, W \subseteq G_{n,p}$ with the $\mathcal{L}_{\textrm{op}}$-isomorphic induced structures such that $(a_g)_{g \in V} \not \equiv_b (a_g)_{g \in W}$. Then taking some $h \in P_n$ above all of the elements of $V \cup W$ with respect to the order on $P_n$, we have that $Vh\cong^{\mathcal{L}_{\textrm{op}}} Wh$. However, since the tuple $a'_h$ contains $b$, we have $(a'_g)_{g \in Vh} \not\equiv (a'_g)_{g \in Wh}$.
	
	 Then by \cite[Proposition 5.8]{chernikov2014n} there is an $\mathcal{L}_{\textrm{opg}}$-substructure $G' \subseteq G_{n,p}$, a finite set $V \subset G_{n,p}$ and a formula $\psi(y_1, \ldots, y_{n-1}, \tilde{y}_n, z) \in \mathcal{L}$ such that $\tilde{y}_n = y'_nx$, $z$ is a finite tuple of variables corresponding to a fixed enumeration $(a_g)_{g \in V}$ of $V$,  and
	\begin{enumerate}
		\item $G' \cong^{\mathcal{L}_{\textrm{opg}}}G_{n,p}$,
		\item $G_{n,p} \models R_n(g_1, \ldots, g_n)$ if and only if $\models \psi(a'_{g_1}, \ldots, a'_{g_n}, (a'_g)_{g \in V})$, for every $g_i \in P_i(G')$,
		\item for every finite $W,W' \subseteq G'$ we have $WV \cong^{\mathcal{L}_{\textrm{op}}} W'V$ whenever $W\cong^{\mathcal{L}_{\textrm{op}}} W'$.
	\end{enumerate}
	Let $\varphi(x,y_1, \ldots, y_{n-1}, y''_n)$ be the formula $\psi(y_1, \ldots, y_{n-1}, y'_nx, z)$ with $y''_n := y'_n z$, and let $a''_g := a'_g = a_g$ for $g \in P_i(G'), 1 \leq i \leq n-1$ and let $a''_g := a_g (a_h)_{h \in V}$ for $g \in P_n(G')$. Then $(a''_g)_{g \in G'}$ is $\mathcal{L}_{\textrm{op}}$-indiscernible (by $\mathcal{L}_{\textrm{op}}$-indiscernibility of $(a'_g)_{g \in G_{n,p}}$ and the choice of $V$) and $\varphi(b, a''_{g_1}, \ldots, a''_{g_n})$ holds if and only if $R_n(g_1, \ldots, g_n)$ does, for all $g_i \in P_i(G')$. Then $\varphi$ is not $n$-dependent by (2)$\Rightarrow$(1) in Proposition \ref{prop: indisc witness to IPn}. \end{proof}

\begin{lemma}\label{lem: red 1 to one var}
Let $y_1, \ldots, y_{n-1}$ be some fixed finite tuples of variables. If the condition (1) in Proposition \ref{prop: controlling all but one} holds for some finite tuple of variables $x$, then it already holds with $x$ a single variable.
\end{lemma}
\begin{proof}
	We assume that (1) fails for $|x|=1$, and prove that then it fails for any tuple of variables $x$ by induction on $|x|$.
	So let $b \in \mathbb{M}_x$ with $|b| > 1$ be given, say $b = b_1 b_2$ for some tuples $1 \leq |b_1|, |b_2| < n$. And assume that $(a_g)_{g \in G_{n,p}}$ with $a_g \in \mathbb{M}_{y_i}$ for $g \in P_i, 1 \leq i \leq n-1$ is such that $(a_g)_{g \in G_{n,p}}$ is $G_{n,p}$-indiscernible over $b$ and $O_{n,p}$-indiscernible over $\emptyset$. We need to show that  $(a_g)_{g \in G_{n,p}}$ is $O_{n,p}$-indiscernible over $b$. 
	
	In particular $(a_g)_{g \in G_{n,p}}$ is $G_{n,p}$-indiscernible over $b_2$, hence it is  $O_{n,p}$-indiscernible over $b_2$ by the inductive assumption. Let $a'_g := a_g$ for $g \in P_i, 1 \leq i \leq n-1$ and let $a'_g := a_g b_2$ for $g \in P_n$. Note that $(a'_g)_{g \in G_{n,p}}$ is $G_{n,p}$-indiscernible over $b_1$, and is $O_{n,p}$-indiscernible over $\emptyset$ by the previous sentence. Applying the inductive assumption again, we conclude that $(a'_g)_{g \in G_{n,p}}$ is $O_{n,p}$-indiscernible over $b_1$, hence $(a_g)_{g \in G_{n,p}}$ is $O_{n,p}$-indiscernible over $b=b_1 b_2$.
	\end{proof}

Using this, we can finally strengthen Fact \ref{fac: props of n-dependent formulas}(3).
\begin{theorem}\label{thm: better reduction to singletons}
\begin{enumerate}
\item 	Assume that the formula $\varphi(x, y_1, \ldots, y_n)$ is not $n$-dependent. Then there exists some formula $\varphi'(x', y_1, \ldots, y_{n-1}, y'_n)$ which is not $n$-dependent, and such that $x'$ is a single variable and $y'_n$ is some finite tuple of variables extending $y_n$.

	\item A theory $T$ is $n$-dependent if and only if every formula $\varphi(x,y_1, \ldots, y_n)$ such that all but at most one of the tuples $x, y_1, \ldots, y_n$ are singletons is $n$-dependent.
\end{enumerate}
	\end{theorem}
\begin{proof}
(1) By Lemma \ref{lem: red 1 to one var} and the equivalence of (1) and (2) in Proposition \ref{prop: controlling all but one}.

(2)	Assume that some formula $\varphi(x, y_1, \ldots, y_n)$ is not $n$-dependent. Applying (1), we find some formula $\varphi'(x',y_1, \ldots,y_{n-1}, y^1_n)$ which is not $n$-dependent, $x'$ is a singleton and $y^1_n$ is a tuple of variables extending $y_n$. Exchanging the roles of $x'$ and $y_1$ by Fact \ref{fac: props of n-dependent formulas} we thus obtain a formula $\varphi_1(y_1, y'_1 , y_2, \ldots, y_{n-1}, y^1_n)$ which is not $n$-dependent and $|y'_1| = 1$. Repeating the same procedure recursively with $y_i$ in the role of $y_1$, for $1 \leq i \leq n-1$ we find formulas $\varphi_i(y_i, y'_1, \ldots, y'_{i},y_{i+1}, \ldots, y_{n-1}, y^i_n)$ which are not $n$-dependent, $|y'_j|=1$ for $1 \leq j \leq i$ and $y_n^{j+1}$ extending $y_n^j$. Finally, taking $\varphi_{n-1}(y_{n-1}, y'_1, \ldots, y'_{n-1}, y_n^{n-1})$ and applying (1) one more time, we obtain the desired formula with all but the last variable singletons.
\end{proof}

\section{$N$-dependent valued fields}\label{sec: valued fields}

The main result of this section is the following theorem generalizing a recent result of Johnson \cite{johnson2021dpA} from $n=1$ to all $n \in \mathbb{N}$.
\begin{theorem}\label{thm: main hens}
	If $(K,\mathcal{O})$ is an infinite valued field of positive characteristic and $\Th(K)$ is $n$-dependent for some $n \in \mathbb{N}$, then $K$ is henselian.
\end{theorem}

From now on, let $K$ be an infinite field of characteristic $p>0$ and $\mathcal O_i$ a valuation ring on $K$ for $i = 1,2$. %Throughout the section, we will assume that the structure $(K, \mathcal O_1, \mathcal O_2)$ is $k$-dependent and that $K$ has positive characteristic.  
We additionally fix the following notation.
 \begin{itemize}
	 \item  For $i = 1, 2$, let $\mathfrak m_i$ be the maximal ideal of $\mathcal O_i$;
	 \item let $J := \mathfrak m_1 \cap \mathfrak m_2$.
	 \end{itemize}

\begin{fact}\cite[Remark 2.1]{johnson2021dpA}\label{fac: incomp vals are indep}
	Assume $\mathcal O_1$ and $\mathcal O_2$ are incomparable (i.e.~none of them is contained in the other). Then
	$$ (a + \mathfrak m_1) \cap (b + \mathfrak m_2)\neq \emptyset$$
	for any $a\in \mathcal O_1$ and $b\in \mathcal O_2$.
\end{fact}

\begin{defn}
	We say that $b \in K$ is an \emph{Artin-Schreier root} of $a \in K$ if $a = b^p-b$. We call $K$ \emph{Artin-Schreier closed} if every element of $K$ has an Artin-Schreier root in $K$.
\end{defn}

 Recall the following.
\begin{fact}(\cite{kaplan2011artin} for $n=1$, \cite{hempel2016n} for arbitrary $n \in \mathbb{N}$)\label{fac: k-dep AS closed}
	Let $K$ be an infinite field of positive characteristic, such that $\Th(K)$ is $n$-dependent. Then $K$ is Artin-Schreier closed.
\end{fact}

Our main contribution is the following result.
\begin{prop}\label{prop_mainhensilanity}
	Suppose that $(K, \mathcal O_1, \mathcal O_2)$ is $n$-dependent and  $\cha (K) = p >0$. Then every element in $J$ has an Artin-Schreier root in $J$.  
\end{prop}

Being able to find an Artin-Schreier root in both maximal ideals \emph{simultaneously} forces the corresponding valuations to be comparable:
	
\begin{cor}\label{cor: main hens}
 If the structure $(K, \mathcal O_1, \mathcal O_2)$ is $n$-dependent and $\cha (K) = p >0$, then $\mathcal O_1$ and $\mathcal O_2$ are comparable.
\end{cor}
\begin{proof}
Assume not, then by Fact \ref{fac: incomp vals are indep} with $a=0$ and $b=1$ there exists some $w\in  \mathfrak m_1 \cap (1 + \mathfrak m_2)$. Let $y := w^p -w$. Now, as $\val_1(w)>0$, we have that $$\val_1(y) = \val_1(w) > 0.$$ Secondly, let $z \in \mathfrak m_2$, i.e.~$\val_2(z) > 0$, be such that $w = 1 +z$. Then  $$\val_2(y) = \val_2(w^p - w) = \val_2\left( (1+z)^p  - (1+z)\right) = \val_2(z^p - z) = \val_2(z) > 0.$$ Thus  $y \in J$. However, the Artin-Schreier roots of $y$ are exactly $w, w+1, \dots, w+p-1$, none of which can lie in $\mathfrak m_1 \cap \mathfrak m_2 =J$. This contradicts Proposition \ref{prop_mainhensilanity}.
\end{proof}

Then Theorem \ref{thm: main hens} follows from Corollary \ref{cor: main hens} exactly as in the proof of \cite[Theorem 2.8]{johnson2021dpA} using that $n$-dependence is preserved under interpretations.
Our proof of Proposition  \ref{prop_mainhensilanity} is given in Section \ref{subsec: proof of Hens}, but before presenting it we have to develop the  following three main ingredients: 
\begin{itemize}
\item a chain condition for intersections in uniformly definable families of subgroups in $n$-dependent theories, discussed in Section \ref{sec: chain conditions};
\item an explicit version of the isomorphism for special linear groups from Kaplan-Scanlon-Wagner \cite{kaplan2011artin} (see Section \ref{subsection_iso} for a discussion and Appendix \ref{sec: app Bays} by Martin Bays for the proofs);
	\item a detailed analysis of what happens to the valuations of certain elements in the field  when this special isomorphism is applied (carried out in Section \ref{sec: f on val}).
	\end{itemize}

%-------------------------------------------------
%-------------------------------------------------
%-------------------------------------------------
%--------- Moore matrix --------------------------
%-------------------------------------------------
%-------------------------------------------------
%-------------------------------------------------

\subsection{A ``chain condition'' for intersections of definable subgroups in $n$-dependent theories}\label{sec: chain conditions}

Recall the ``chain condition'' for definable families of subgroups in $n$-dependent theories.

\begin{fact}\cite[Proposition 4.1]{hempel2016n}\label{fac: BaldwinSaxl}
Let $G$ be a definable group, and let $\psi(x; y_0, \ldots, y_{n-1})$ be an $n$-dependent formula such that $\psi(G; b_0, \ldots, b_{n-1})$ is a subgroup of $G$ for any parameters $b_0, \ldots, b_{n-1}$. Then there exists some $m_\psi \in \omega$ such that for any $d \geq m_\psi$ and any array of parameters $(b_{i,j} : i < n, j \leq d)$ there is some $\nu \in d^n$ such that
$$\bigcap_{\eta \in d^n} H_{\eta} = \bigcap_{\eta \in d^n, \eta \neq \nu} H_\eta,$$
where $H_\eta := \psi(G; b_{0,i_0}, \ldots, b_{n-1, i_{n-1}})$ for $\eta = (i_0, \ldots, i_{n-1})$.
\end{fact}

We generalize it to simultaneous intersections of several definable families of subgroups. Before we do so, let us recall the partite version of Ramsey's theorem.
\begin{fact}\label{fac: partite Ramsey}
\begin{enumerate}
\item (Infinitary version) For every $m,n \in \omega$ and any function $f: \omega^n \to m$ there exist some infinite sets $s_0, \ldots, s_{n-1} \subseteq \omega$ such that  $f \restriction_{s_0 \times \ldots \times s_{n-1}}$ is constant.
\item (Finitary version) For every $l, m,n \in \omega$ there is some $R = R(l, m,n) \in \omega$ such that:
	for any function $f: R^n \to m$ there are some sets $s_0, \ldots, s_{n-1} \subseteq R$ with $|s_0|, \ldots, |s_{n-1}| \geq l$ and such that $f \restriction_{s_0 \times \ldots \times s_{n-1}}$ is constant.
\end{enumerate}
\end{fact}

\begin{prop}\label{prop: chain cond for mult fam of subgrps}
	Let $G$ be a definable group, and for $t< k$ let $\psi_t(x; y_0, \ldots, y_{n-1})$ be an $n$-dependent formula such that $\psi_t(G; b_0, \ldots, b_{n-1})$ is a subgroup of $G$ for any $t<k$ and any parameters $b_0, \ldots, b_{n-1}$. Then there exists some $m = m(\psi_0, \ldots, \psi_{k-1})\in \omega$ such that: for any $d \geq m$ and any array of parameters $(b_{i,j} : i < n, j \leq d)$ there is a single $\nu \in d^n$ such that for all $t<k$,
$$\bigcap_{\eta \in d^n} H^t_{\eta} = \bigcap_{\eta \in d^n, \eta \neq \nu} H^t_\eta,$$
where $H^t_\eta := \psi_t(G; b_{0,i_0}, \ldots, b_{n-1, i_{n-1}})$ for $\eta = (i_0, \ldots, i_{n-1})$.

\end{prop}
\begin{proof}

We argue by induction on $k$, the base case $k=1$ given by Fact \ref{fac: BaldwinSaxl}. Let $m_1 := m(\psi_0)$ and $m_2 := m(\psi_1, \ldots, \psi_{k-1})$ be given by the inductive hypothesis. Let $R := R(m_2, m_1^n, n)$ be given by Fact \ref{fac: partite Ramsey}(2).
We take $m = m(\psi_0, \ldots, \psi_{k-1}) := R m_1$.

Let an array $B = (b_{i,j} : i<n, j \leq m)$ be given. For each $ \gamma = (\gamma_0, \ldots, \gamma_{n-1}) \in R^n$, consider the subarray 
$$B_{\gamma} = \left( b_{0, \gamma_0 m_1 +\eta_0}, \ldots, b_{n-1, \gamma_{n-1}m_1 + \eta_{n-1}} : \eta = (\eta_0, \ldots, \eta_{n-1}) \in m_1^n \right).$$

By the choice of $m_1$, for each $\gamma \in R^n$ there is some $\nu_\gamma \in m_1^n$ such that 
$$ \bigcap_{\eta \in m_1^n} H^0_{(\gamma_0 m_1 +\eta_0, \ldots, \gamma_{n-1}m_1 + \eta_{n-1})} = \bigcap_{\eta \in m_1^n, \eta \neq \nu_\gamma} H^0_{(\gamma_0 m_1 +\eta_0, \ldots, \gamma_{n-1}m_1+ \eta_{n-1})}. \ \ \ \ \ (*)$$

By the choice of $R$, there are some sets $s_0, \ldots, s_{n-1} \subseteq R$ with $|s_0|= \ldots = |s_{n-1}| = m_2$ such that $\nu_\gamma$ is equal to some fixed $\nu' = (\nu_0', \dots, \nu_{n-1}') \in m_1^n$, for all $\gamma \in s_0 \times \ldots \times s_{n-1}$. Consider the array 
$$B' = (b_{0, \gamma_0 m_1+ \nu'_0}, \ldots, b_{{n-1, \gamma_{n-1} m_1+ \nu'_{n-1}}} : \gamma \in s_1 \times \ldots \times s_{n-1}).$$
By the choice of $m_2$, there is some $\gamma' \in s_1 \times \ldots \times s_{n-1}$ such that
$$ \bigcap_{\gamma \in s_0 \times \ldots \times s_{n-1}} H^t_{(\gamma_0 m_1 + \nu_0', \ldots, \gamma_{n-1}m_1 + \nu'_{n-1})} = $$
$$\bigcap_{\gamma \in s_0 \times \ldots \times s_{n-1}, \gamma \neq \gamma'} H^t_{(\gamma_0 m_1 +\nu_0', \ldots, \gamma_{n-1}m_1 + \nu_{n-1}')} \textrm{ for all } 1 \leq t <k.  \ \ \ \ \  (**)$$

Let $\nu := (\gamma_0'm_1 + \nu_0', \ldots, \gamma_{n-1}'m_1 + \nu_{n-1}')$. By $(*)$ and $(**)$ we have
$$\bigcap_{\eta \in m^n} H^t_{\eta} = \bigcap_{\eta \in m^n, \eta \neq \nu} H^t_\eta \textrm{ for all } 0 \leq t<k,$$
as desired.
\end{proof}

%-------------------------------------------------
%-------------------------------------------------
%-------------------------------------------------
%--------- special vector group ------------------
%-------------------------------------------------
%-------------------------------------------------
%-------------------------------------------------

\subsection{Special vector groups and their explicit isomorphisms}\label{subsection_iso}
Let $K$ be a field of characteristic $p>0$. We let $\K$ be the algebraic closure of $K$, $\mathcal{K}$ a perfect subfield of $K$, and let $\wp(x) $ be the additive homomorphism $x \mapsto x^p -x$ on $\K$. We consider the following algebraic subgroups of  $(\K, +)^{n}$:
\begin{defn}\label{def_Ga}
For a singleton $a$ in $\K$, we let $G_{a}$ be equal to $(\K, +)$, and for a tuple $\bar{a}=(a_0, \dots , a_{n-1}) \in \K^n$ with $n > 1$ we define:
$$G_{\bar{a}} = \left\{ (x_0, \dots , x_{n-1}) \in \K^{n} :\ a_0 \cdot \wp(x_0)= a_i \cdot \wp(x_i)\mbox{ for }0 \leq i < n \right\}.$$
\end{defn}
Recall that for an algebraic group $G$, we denote by $G^0$ the connected component of the unit element of $G$ (in the Zariski topology). Note that if $G$ is definable over some parameter set $A$, its connected component $G^0$ coincides with the smallest $A$-definable subgroup of $G$ of finite index (in $\mathbb{K}$). We have the following sufficient condition for connectedness of $G_{\bar{a}}$.

\begin{fact}\label{fact_GaCon} \cite[Lemma 5.3]{hempel2016n} Let $\bar{a} = (a_0, \dots, a_{n-1} )$ be a tuple in $\K^{\times}$ for which the set 
$\left\{ \frac{1}{a_0}, \dots,   \frac{1}{a_{n-1}} \right\}$
is linearly $\F_p$-independent. Then $G_{\bar{a}}$ is connected.
\end{fact}

Moreover, under the same assumption on $\bar{a}$ these groups are isomorphic to the additive group of the field:

\begin{fact}\label{fact_cor_ConComIso}\cite[Corollary 5.4]{hempel2016n} \label{fact_iso}
Let $\mathcal{K}$ be a perfect subfield of an algebraically closed field $\K$, and  let $\bar{a} \in \mathcal{K}^n$ be such that the set 
$\left\{ \frac{1}{a_0}, \dots,   \frac{1}{a_{n-1}} \right\}$
is linearly $\F_p$-independent. Then $G_{\bar{a}}$ is (algebraically) isomorphic to $(\K, +)$ over $\mathcal{K}$. In particular, for any field $K$ with $\mathcal{K} \leq K \leq \K$, the group $G_{\bar{a}}(K)$ is isomorphic to $(K, +)$.
\end{fact}

%\begin{fact}\label{fact_n-dep}\cite{Hem_n-dep}[Corollary ] Let $\{\alpha _{i,j} : i \in n, j \in m\} $ be a set of algebraically independent elements in $\K$. Then the tuple $( a_{(i_0, \dots, i_{n-1})}: (i_0, \dots, i_{n-1}) \in m^n)$ with $a_{(i_0, \dots, i_{n-1})}=  \prod_{l=0}^{n-1}\alpha _{l, i_l}$ and ordered lexicographically satisfies the condition of Lemma \ref{lem_GaCon}.
%\end{fact}

%-------------------------------------------------
%-------------------------------------------------
%-------------------------------------------------
%--------- rest ------------------
%-------------------------------------------------
%-------------------------------------------------
%-------------------------------------------------

In Appendix \ref{sec: app Bays}, Bays provides an explicit description of such an isomorphism that we now describe to set up the notation.

Given arbitrary $m\in \mathbb{N}$ and $x_1, \ldots, x_m \in K$, the corresponding \emph{Moore matrix} is the $m \times m$ matrix
	$$M(x_1, \ldots, x_m) = \begin{pmatrix}
x_1 & \ldots &x_m\\
x_1^p & \ldots & x_m^p\\
\vdots &    & \vdots \\
x_1^{p^{m-1}} & \ldots & x_m^{p^{m-1}}
\end{pmatrix},$$
and the \emph{Moore determinant} is $\Delta(x_1, \ldots, x_m) := \det M(x_1, \ldots, x_m)$. By Fact \ref{f:moore}, the set $\{ x_1, \ldots, x_m\}$ is linearly independent over $\mathbb{F}_p$ if and only if $\Delta(x_1, \ldots, x_m) \neq 0$. 

Now, fix $\bar a  = (a_0, \dots, a_{m}) \in \mathcal{K}^{m+1}$ such that the set $\left\{ \frac{1}{a_0}, \dots,   \frac{1}{a_m} \right\}$
is $\F_p$-linearly independent, and let 
$$A := M \left(a^{-\frac{1}{p^m}}_0, \ldots, a^{-\frac{1}{p^m}}_m \right).$$
 Note that $\left\{ a^{-\frac{1}{p^m}}_0, \ldots, a^{-\frac{1}{p^m}}_m \right\}$ is still an $\mathbb{F}_p$-linearly independent subset of $\mathcal{K}$, hence $A$ is invertible by Fact \ref{f:moore}. We define  
$$\bar{\alpha} = \begin{pmatrix}
\alpha_0\\
\vdots \\
\alpha_{m-1}\\
\alpha_m
\end{pmatrix} 
:= 
A^{-1}\begin{pmatrix}
0\\
\vdots \\
0\\
1
\end{pmatrix},$$
that is $\alpha_i = (A^{-1})_{i, m} \in \mathcal{K}$ for $0 \leq i \leq m$. One still has that $(\alpha_0, \ldots, \alpha_m)$ are linearly $\mathbb{F}_p$-independent (see Claim \ref{c:alphaIndep}), hence $M(\alpha_0, \dots, \alpha_{m})$ is invertible. Let $\beta_{i,j} \in \mathcal{K}$ be the entries of the inverse matrix of $M(\alpha_0, \dots, \alpha_{m})$. Then we have:
\begin{fact}\label{fac: expl iso fa}
The map $f_{\bar{a}}: G_{\bar a}(\mathbb{K}) \rightarrow (\mathbb{K},+)$ given by
$$ f_{\bar{a}}(x_0, \ldots, x_m) := \sum_{j=0}^{m} \alpha_j x_j$$
is a group isomorphism, and  $f^{-1}_{\bar{a}}: (\mathbb{K},+) \rightarrow  G_{\bar a}(\mathbb{K}) $ given by 
$$  f^{-1}_{\bar{a}}(t) := \left(\sum_{j=0}^{m} \beta_{i,j} t^{p^j} : 0 \leq i \leq m\right)$$
is its inverse. \end{fact}

\subsection{The effect of the isomorphism $f_{\bar{a}}$ on the valuation}\label{sec: f on val}
\emph{For the rest of this subsection we assume that $\mathcal{O}$ is a valuation ring on $K$, $\mathfrak{m}$ is its maximal ideal, $\val$ is the corresponding valuation, and we fix some $a_0, \ldots, a_m \in \mathcal{K}$ such that $\val(a_i) \neq \val(a_j)$ for all $0 \leq i \neq j \leq m$.}

This implies in particular that $\left\{ \frac{1}{a_0}, \ldots, \frac{1}{a_m} \right\}$ are $\mathbb{F}_p$-linearly independent (as the valuation of any $\mathbb{F}_p$-linear combination is $\neq \infty$). Throughout this section, we let $f:=f_{\bar{a}} = \sum_{j=0}^{m} \alpha_j x_j$ be the isomorphism $G_{\bar a}(\mathbb{K}) \rightarrow (\mathbb{K},+)$ given by Fact \ref{fac: expl iso fa}. Let $\alpha_0, \ldots, \alpha_m \in \mathcal{K}$ be as defined in Subsection \ref{subsection_iso}. We will prove several technical lemmas that allow us to control $\val(f(x_1, \ldots, x_m))$ in terms of the tuple $(\val(x_1), \ldots, \val(x_m))$, and vice versa. 
To motivate this analysis, the reader might prefer to check how it is used in the proof of Proposition \ref{prop_mainhensilanity} in the next section before going into the details of the calculations here.

\begin{remark}\label{rem: OFE}
Assume $(x_0, \dots, x_m) \in G_{\bar a}$, then 
$$\val(a_i) +\val(x_i^p-x_i)= \val(a_j)+\val(x_j^p-x_j) \ \ \ \ $$
for all $0 \leq i,j\leq m$.

 Additionally note that 
		$$
		\val ( x_i^p-x_i) = \left\{ \begin{aligned}
		                \val(x_i)  \ \ \ \  &\mbox{if } \val(x_i)>0 \\
		                p \val(x_i)\  \ \ \  &\mbox{if}  \val(x_i)<0
		                \end{aligned} \right.
		$$
and 
		$$\val ( x_i^p-x_i)	  \geq 0  \mbox{ if } \val(x_i)=0.$$
\end{remark}

%\vspace{12pt}
%\noindent {\Large \underline{Compute $\val(\alpha_i)$}}
%\vspace{12pt}

\begin{lemma}\label{lem_valalpha} Suppose that $0< \val(a_0)< \dots < \val(a_m)$.  Then the sequence $(\val(\alpha_i): i \in \{0, \dots, m\})$ is strictly increasing. 

In fact, for any $0 \leq i \leq m$ we have
	$$\val(\alpha_i) =  \frac{1}{p^{m-i}} \val(a_i)+ \sum_{j=i}^{m-1} \frac{p-1}{p^{m-j}} \val(a_{j+1})>0 .$$ 
	\end{lemma}

\begin{proof}
 Recall that, by linear algebra, for each $0 \leq i \leq m$ we have:
$$ \alpha_i =\left(A^{-1}\right)_{i,m} = \frac{1}{\det(A)} C_{m,i},$$
where $C_{m,i}$ is the corresponding cofactor of $A$. That is,
$$ \alpha_i = \frac{(-1)^{m+i}\Delta\left(a^{-\frac{1}{p^m}}_0, \ldots, a^{-\frac{1}{p^m}}_{i-1}, a^{-\frac{1}{p^m}}_{i+1}, \ldots, a^{-\frac{1}{p^m}}_m \right)}{\Delta\left(a^{-\frac{1}{p^m}}_0, \ldots, a^{-\frac{1}{p^m}}_m \right)}. $$
Now, we compute the valuation of the numerator and denominator separately. First,
$$\Delta\left(a^{-\frac{1}{p^m}}_0, \ldots, a^{-\frac{1}{p^m}}_m\right)= \sum_{\pi \in \Sym \left(\{0, \ldots, m\} \right)} \sign(\pi) \left(a^{-\frac{1}{p^m}}_{\pi(0)} \right)^{p^0} \cdot \ldots  \cdot  \left(a^{-\frac{1}{p^m}}_{\pi(m)} \right)^{p^m}. $$
Let $i<j$, then 
$$\val \left({a_i}^{-\frac{1}{p^m}} \right) = -\frac{1}{p^m} \val(a_i) >   -\frac{1}{p^m} \val(a_j)= \val \left({a_j}^{-\frac{1}{p^m}} \right).$$
Thus $\left(\val \left({a_i}^{-\frac{1}{p^m}} \right) : 0 \leq i \leq m \right)$ is strictly decreasing.
Using this, we  see that 
$$\val \left( \left(a^{-\frac{1}{p^m}}_{0} \right)^{p^0} \cdot \ldots  \cdot  \left(a^{-\frac{1}{p^m}}_{m}\right)^{p^m} \right) < \val \left( \left(a^{-\frac{1}{p^m}}_{\pi(0)} \right)^{p^0} \cdot \ldots  \cdot  \left(a^{-\frac{1}{p^m}}_{\pi(m)}\right)^{p^m} \right)$$
for every \emph{non-identity} permutation $\pi \in \Sym \left(\{0, \ldots, m\} \right)$. Thus 
\begin{align*}
	\val \left(\Delta\left(a^{-\frac{1}{p^m}}_0, \ldots, a^{-\frac{1}{p^m}}_m\right)\right)& = \val \left(\left(a^{-\frac{1}{p^m}}_{0}\right)^{p^0} \cdot \ldots  \cdot  \left(a^{-\frac{1}{p^m}}_{m}\right)^{p^m}\right)\\
	& = - \sum_{j=0}^m  {\frac{1}{p^{m-j}}} \val(a_j).
	\end{align*}
Now we turn to the numerator: 
\begin{align*}
&\Delta\left(a^{-\frac{1}{p^m}}_0, \ldots, a^{-\frac{1}{p^m}}_{i-1}, a^{-\frac{1}{p^m}}_{i+1}, \ldots, a^{-\frac{1}{p^m}}_m \right)=\\
&\sum_{\pi \in \Sym\left(\{ 0, \ldots, i-1, i+1, \ldots, m \} \right) } \sign(\pi) \prod_{0 \leq j \leq i-1} \left( a_{\pi(j)}^{-\frac{1}{p^m}} \right)^{p^j} \cdot \prod_{i+1 \leq j \leq m} \left( a_{\pi(j)}^{-\frac{1}{p^m}} \right)^{p^{j-1}}.
\end{align*}
Again, 
\begin{align*}
&\val \left( \prod_{0 \leq j \leq i-1} \left( a_{j}^{-\frac{1}{p^m}} \right)^{p^j} \cdot \prod_{i+1 \leq j \leq m} \left( a_{j}^{-\frac{1}{p^m}} \right)^{p^{j-1}} \right) \\<\ \ & \val \left( \prod_{0 \leq j \leq i-1} \left( a_{\pi(j)}^{-\frac{1}{p^m}} \right)^{p^j} \cdot \prod_{i+1 \leq j \leq m} \left( a_{\pi(j)}^{-\frac{1}{p^m}} \right)^{p^{j-1}} \right)
\end{align*}
 for every \emph{non-identity} permutation $\pi \in \Sym \left(\{0, \ldots, m\} \right)$. Thus
\begin{align*}
	&\val \left(\Delta\left(a^{-\frac{1}{p^m}}_0, \ldots, a^{-\frac{1}{p^m}}_{i-1}, a^{-\frac{1}{p^m}}_{i+1}, \ldots, a^{-\frac{1}{p^m}}_m \right)\right) =\\
	&\val \left( \prod_{0 \leq j \leq i-1} \left(a^{-\frac{1}{p^m}}_{j} \right)^{p^j} \cdot \prod_{i+1 \leq j \leq m} \left(a^{-\frac{1}{p^m}}_{j} \right)^{p^{j-1}} \right) =\\
	& - \sum_{j=0}^{i-1} \frac{1}{p^{m-j}} \val(a_j) - \sum_{j=i+1}^{m}\frac{1}{p^{m-j+1}} \val(a_j).
	\end{align*}
Combining, we get
\begin{align*}
 \val(\alpha_i) &=  \val \left(\Delta\left(a^{-\frac{1}{p^m}}_0, \ldots, a^{-\frac{1}{p^m}}_{i-1}, a^{-\frac{1}{p^m}}_{i+1}, \ldots, a^{-\frac{1}{p^m}}_m \right)\right) - \val \left(\Delta\left(a^{-\frac{1}{p^m}}_0, \ldots, a^{-\frac{1}{p^m}}_m\right)\right)\\
	&= - \sum_{j=0}^{i-1}  {\frac{1}{p^{m-j}}} \val(a_j)- \sum_{j=i+1}^{m} {\frac{1}{p^{m-j+1}}} \val(a_{j})  -\left(- \sum_{j=0}^m  {\frac{1}{p^{m-j}}} \val(a_j) \right) \\
	&= - \sum_{j=i+1}^{m} {\frac{1}{p^{m-j+1}}} \val(a_{j}) +\frac{1}{p^{m-i}} \val(a_i) + \sum_{j=i+1}^{m}\frac{1}{p^{m-j}} \val(a_j) \\
	&= \frac{1}{p^{m-i}} \val(a_i) +\sum_{j=i+1}^{m} \left(\frac{1}{p^{m-j}}-\frac{1}{p^{m-j+1}}\right) \val(a_{j}) \\
	&= \frac{1}{p^{m-i}} \val(a_i) +\sum_{j=i+1}^{m}\frac{p-1}{p^{m-j+1}}\val(a_{j})\\
	& = \frac{1}{p^{m-i}} \val(a_i)+ \sum_{j=i}^{m-1} \frac{p-1}{p^{m-j}} \val(a_{j+1}).
\end{align*}
\end{proof}
%$$ = \frac{(-1)^{m+i}\Delta\left(a^{-\frac{1}{p^m}}_0, \ldots, a^{-\frac{1}{p^m}}_{i-1}, a^{-\frac{1}{p^m}}_{i+1}, \ldots, a^{-\frac{1}{p^m}}_m \right)}{(-1)^{m-i}\Delta\left(a^{-\frac{1}{p^m}}_0, \ldots, a^{-\frac{1}{p^m}}_{i-1}, a^{-\frac{1}{p^m}}_{i+1}, \ldots, a^{-\frac{1}{p^m}}_m,  a^{-\frac{1}{p^m}}_i\right)}$$
%$$ = \frac{1}{\prod_{c_1, \ldots,c_m \in \mathbb{F}_p} (c_1 a^{-\frac{1}{p^m}}_0 + \ldots + c_i a^{-\frac{1}{p^m}}_{i-1} + c_{i+1}a^{-\frac{1}{p^m}}_{i+1}, \ldots + c_m a^{-\frac{1}{p^m}}_{m} + a^{-\frac{1}{p^m}}_i) } $$
%$$ = \frac{1}{\prod_{c_0, \ldots, c_m \in \mathbb{F}_p, c_i=1} \sum_{j=0}^{m}c_j a^{-\frac{1}{p^m}}_{j}}.$$

\begin{remark} \label{rem: permuting alphas} Let $\sigma \in \Sym\left( \{0,1, \ldots, m\} \right)$ be an arbitrary permutation, and let $\bar{a}' := \left( a_{\sigma(0)}, \ldots, a_{\sigma(m)} \right)$. Let $\bar{\alpha}' := \left(\alpha'_0, \ldots, \alpha'_m \right)$ be the tuple given in Fact \ref{fac: expl iso fa} with respect to the tuple $\bar{a}'$, so that the map $f_{\bar{a}'}: G_{\bar{a}'}(\mathbb{K}) \to (\mathbb{K},+)$ given by $f_{\bar{a}'}(x_0, \ldots, x_m) := \sum_{j=0}^{m} \alpha'_j x_j$ is a group isomorphism.
Then $\val\left( \alpha'_{i} \right) = \val \left( \alpha_{\sigma(i)} \right)$ for all $0 \leq i \leq m$.
\end{remark}
\begin{proof} As in the proof of Lemma \ref{lem_valalpha}, for any $0 \leq i \leq m$ we have
\begin{align*}
\val(\alpha'_i) &= \val \left( \frac{\Delta\left(a^{-\frac{1}{p^m}}_{\sigma(0)}, \ldots, a^{-\frac{1}{p^m}}_{\sigma(i-1)}, a^{-\frac{1}{p^m}}_{\sigma(i+1)}, \ldots, a^{-\frac{1}{p^m}}_{\sigma(m)}\right)}{\Delta\left(a^{-\frac{1}{p^m}}_{\sigma(0)}, \ldots, a^{-\frac{1}{p^m}}_{\sigma(m)} \right)}\right) \\
&=\val \left( \frac{ (-1)^{\sign(\sigma')}\Delta\left(a^{-\frac{1}{p^m}}_{0}, a^{-\frac{1}{p^m}}_{1},  \ldots, a^{-\frac{1}{p^m}}_{\sigma(i) -1}, a^{-\frac{1}{p^m}}_{\sigma(i)+1}, \ldots, a^{-\frac{1}{p^m}}_{m}\right)}{(-1)^{\sign(\sigma)}\Delta\left(a^{-\frac{1}{p^m}}_{0}, \ldots, a^{-\frac{1}{p^m}}_{m} \right)}\right)\\
&=\val \left(\alpha_{\sigma(i)} \right), 
\end{align*}
where $\sigma' := \sigma\restriction \{0, \ldots, \sigma(i) -1, \sigma(i) +1, \ldots, m \}$.
\end{proof}

\begin{cor} \label{cor_minvalal} Let $0 \leq l \leq m$ be arbitrary. Suppose that $ \val(a_s)< \val(a_l)$ for all $0 \leq s \neq l \leq m$, then $\val(\alpha_l)= \val(a_l)$ and $\val(\alpha_s) < \val(a_l)$ for all $0 \leq s \leq m, s \neq l$.
	
	\end{cor}
	
\begin{proof}
 Since all of the $a_i$'s have different valuations by assumption, reordering and using Remark \ref{rem: permuting alphas} we may assume that $0 < \val(a_0)< \dots < \val(a_m)$ and $l=m$. Then using Lemma \ref{lem_valalpha} we immediately obtain the result.
\end{proof}

%\vspace{12pt}
%\noindent {\Large \underline{Compute $\val(x_m)$}}
\begin{lemma}\label{lem_xn0} Suppose that $0< \val(a_0)< \dots < \val(a_m)$ and let $y \in K$ be such that $\val(a_m) < \val(y)$.  Let $(x_0, \dots, x_m) \in G_{\bar a}(K)$ be such that $f^{-1}(y)=(x_0, \dots, x_m) $. Then
	$\val(x_m)= \val(y)  - \val(a_m)$. In particular, $\val(x_m)>0$.
\end{lemma}

\begin{proof}
By the formula for $f^{-1}$ in Subsection \ref{subsection_iso}, we have that
$$x_m = \sum_{j=0}^m \beta_{m,j} y^{p^j},  \ \ \ \ (\ast)$$
where $\beta_{m,j}$ is the $(m,j)$-entry of the inverse of the matrix
$$D= (\alpha_j^{p^i})_{i,j}, 0 \leq i,j \leq m. %\begin{pmatrix}  \end{pmatrix}
$$
So $\beta_{m,j}= \frac{1}{\det D} C_{j,m}$, where $C_{j,m}$ is the $(j,m)$-cofactor of $D$.
We determine the valuation of each of the summands in $(\ast)$ separately. Fix  $0\leq j \leq m$, then
$$ \val(\beta_{m,j} y^{p^j} ) = \val(\beta_{m,j}) + p^j \val(y),$$
and $$\val(\beta_{m,j})= \val(C_{j,m})- \val(\det D).$$
Again let us compute the valuations on the right hand side of the above equation separately. First,
$$\val(\det D)= \val\left( \sum_{ \pi \in \Sym\left(\{0, \ldots, m\} \right)} \sign(\pi) \prod_{i=0}^m \alpha_{\pi(i)}^{p^i} \right)$$
As $\val(\alpha_i)$ is strictly increasing with $i$ by Lemma \ref{lem_valalpha}, note that $ \val\left( \prod_{i=0}^m \alpha_{\pi(i)}^{p^i} \right)$ is strictly minimal if $\pi$ maps $i$ to $m-i$ for all $0 \leq i \leq m$. So
$$\val(\det D)= \val \left(  \prod_{i=0}^m \alpha_{m-i}^{p^i} \right) =   \sum_{i=0}^m  p^i \val(\alpha_{m-i} ).$$
Similarly,
$$\val (C_{j,m} ) =
\sum_{\pi \in \Sym\left(\{ 0, \ldots, m-1 \} \right) } \sign(\pi) \prod_{i =0}^{j-1} \alpha^{p^i}_{\pi(i)} \cdot \prod_{i=j+1}^{m} \alpha^{p^i}_{\pi(i-1)}.$$
Again, as $\val(\alpha_i)$ is strictly increasing with $i$, we conclude that 
$$\val \left(\prod_{i =0}^{j-1} \alpha^{p^i}_{\pi(i)} \cdot \prod_{i=j+1}^{m} \alpha^{p^i}_{\pi(i-1)} \right)$$
 is strictly minimal if $\pi(i) = (m-1)-i$ for all $0 \leq i \leq m-1$. Hence
\begin{align*}
\val (C_{j,m} ) &= \val \left(\prod_{i =0}^{j-1} \alpha^{p^i}_{m-1 -i} \cdot \prod_{i=j+1}^{m} \alpha^{p^i}_{m-i} \right)\\
&=\sum_{i=0}^{j-1} p^i \val\left( \alpha_{m-1-i} \right) + \sum_{i=j+1}^{m} p^i \val(\alpha_{m-i}).
\end{align*}
Thus 
\begin{align*}
	\val(\beta_{m,j}) &= \val(C_{j,m})- \val(\det D)\\
	&=\sum_{i=0}^{j-1} p^i \val\left( \alpha_{m-1-i} \right) + \sum_{i=j+1}^{m} p^i \val(\alpha_{m-i}) -    \sum_{i=0}^m  p^i \val(\alpha_{m-i} )  \\
	&=\sum_{i=0}^{j-1} p^i \val\left( \alpha_{m-1-i} \right)-\left( \val(\alpha_m) +  \sum_{i=1}^{j} p^{i} \val(\alpha_{m-1})\right)\\
	&=\sum_{i=0}^{j-1} p^i \val\left( \alpha_{m-1-i} \right)- \val(\alpha_m) -  \sum_{i=0}^{j-1} p^{i+1} \val(\alpha_{m-1-i})\\
	&= - \val(\alpha_m) - \sum_{i=0}^{j-1}\left( p^{i+1} - p^i \right) \val(\alpha_{m-1-i})\\
	& = - \val(\alpha_m) - \sum_{i=1}^{j} p^{i-1}\left(p-1 \right) \val(\alpha_{m-i}).\\
	\end{align*}
Note that for any $1 \leq j \leq m$ we have
\begin{align*}
	\val(\beta_{m,j-1} y^{p^{j-1}})&=  \val(\beta_{m,j-1}) +  p^{j-1}\val(y) \\
	& =  \val(\beta_{m, j}) + p^{j-1}(p-1)\val(\alpha_{m-j}) +  p^{j-1}\val(y) \\
	&< \val(\beta_{m, j})+  p^j\val(y)\\
	&= \val(\beta_{m,j} y^{p^j}).
\end{align*}
Thus $ \val(\beta_{m,0}y)$ is strictly minimal amongst them, so
	$$\val( x_m)  = \val(\beta_{m,0}y)= - \val(\alpha_m) + \val(y)= - \val(a_m)+ \val(y)$$
	using Lemma \ref{lem_valalpha}.
\end{proof}

\begin{lemma}\label{lem_ordxi}
Let $(x_0, \dots, x_m) \in G_{\bar a}(K)$ be arbitrary.
\begin{enumerate}
\item   Assume that for a  fixed $0 \leq l \leq m$ we have that $\val(a_l) > \val(a_s)$  for all $0 \leq s \neq l \leq m$ and $\val(x_l)\geq 0$. Then for any $0 \leq s \neq l \leq m$,  if $\val(x_s)>0$  we obtain $\val(\alpha_s x_s)>\val(\alpha_l x_l)$.
\item  Suppose that $0 \leq s \neq t \leq m$ are such that $\val(a_s)< \val(a_t)$, $\val(x_s)=0$ and  $\val(x_t)\geq 0$. Then $\val(\alpha_s x_s)<\val(\alpha_t x_t)$.
%\item Suppose that $\val(x_i)=0$ and $i\neq m-1$, then $\val(\alpha_i x_i)<\val(\alpha_m x_m)$.
\end{enumerate}
\end{lemma}
\begin{proof} Reordering the $a_i$'s if necessary (relying on Remark \ref{rem: permuting alphas}), we may assume that $0 < \val(a_0)< \dots < \val(a_m)$, so in particular $l=m$, $s<t$ and $\val(x_m) \geq  0$ by assumption.
\begin{enumerate}
\item 

By Remark \ref{rem: OFE} and assumption $\val(x_s) > 0$, we have
$$\val(x_s)= \val(a_m)-\val(a_s) + \val(x_m^p-x_m). \ \  \ (\dagger)$$
Moreover, in general we have 
$$ \frac{1}{p^{m-s}} +  \sum_{s+1 \leq j \leq m} \frac{p-1}{p^{m+1 - j}}=1.   \ \ \  (\ast)$$
Then
\begin{align*}
	\val(\alpha_s x_s)&= \val(\alpha_s) +\val(x_s)\\
	&\underset{(\dagger)}{\overset{\ref{lem_valalpha}}{=}} \frac{1}{p^{m-s}} \val(a_s) + \sum_{j = s+1}^{m} \frac{p-1}{p^{m+1 - j}} \val(a_j)+ \val(a_m)-\val(a_s)+ \val(x_m^p-x_m) \\
	&\underset{(\ast)}{\overset{\ref{cor_minvalal}}{=}} \val(\alpha_m) + \underbrace{\val(x_m^p-x_m)}_{\geq \val(x_m)} +   \underbrace{\sum_{j = s+1}^{m} \frac{p-1}{p^{m+1 - j}} (\val(a_j)- \val(a_s))}_{>0}\\
	&> \val(\alpha_m) + \val(x_m)\\
	& = \val(\alpha_m x_m).
	\end{align*}
\item By Lemma \ref{lem_valalpha} the sequence $\val(\alpha_i)$'s is strictly increasing. Thus, if $\val(x_s)=0 $, then	
\begin{align*}
	\val(\alpha_s x_s)&= \val(\alpha_s) \\
& {\overset{\ref{lem_valalpha} + \tiny \mbox{ass.}}{<}} \val(\alpha_t) + \val(x_t)\\
	& = \val(\alpha_t x_t).
	\end{align*}
\end{enumerate}
\end{proof}

\begin{lemma}\label{lem_xi0} Let $y \in K$ be such that $\val(a_j) < \val(y)$ for all $0 \leq j \leq m$. Now let  $(x_0, \dots, x_m) \in G_{\bar a}(K)$ be equal to $f^{-1}(y)$. Then $\val(x_j) > 0$ for all $0 \leq j\leq m$.
	\end{lemma}

	\proof Up to reordering the $a_j$'s (using Remark \ref{rem: permuting alphas}), we may assume that $0 < \val(a_0)< \dots < \val(a_m)$.  Then, $\val(x_m) > 0$ by Lemma \ref{lem_xn0} and $(\val(\alpha_j) : 0 \leq j \leq m)$ is strictly increasing by Lemma \ref{lem_valalpha}. For any $0 \leq j \leq m$, 
	by Remark \ref{rem: OFE} we have $$\val(x_j^p - x_j)= \val(a_{m}) - \val(a_j) + \val(x_{m}^p-x_{m})>0,$$
	thus $\val(x_j) \geq 0$. 
	
	If  $\val(x_j)> 0$ for all $j$, we are done. Otherwise, let $I =\{1 \leq  j \leq m: \val(x_j) = 0\} \neq \emptyset$ and let $j_\ast := \min\{I\}$. 
	\begin{claim*}
		 $\val(\alpha_{j_\ast} x_{j_\ast})$ is strictly minimal in  $\left\{ \val(\alpha_j x_j) : 1 \leq j \leq m \right\}$.
	\end{claim*}
	\begin{proof}
		Assume first $j \in I \setminus \{j_\ast\}$, then $\val(x_j) = 0$ and $j > j_\ast$. Hence 
		$$\val(\alpha_{j_\ast} x_{j_\ast}) = \val(\alpha_{j_\ast})  < \val(\alpha_j) = \val(\alpha_{j}x_j).$$
		Otherwise $j \notin I$, i.e. $\val(x_j) > 0$. Then 
		$$\val(\alpha_j x_j) \overset{\ref{lem_ordxi}(1)}{>} \val(\alpha_m x_m) \overset{\ref{lem_ordxi}(2)}{>} \val(\alpha_{j_\ast} x_{j_\ast}).$$
	\end{proof}
 Thus $\val(y)= \val \left( \sum_{j=0}^{m} \alpha_j x_j \right) = \val(\alpha_{j_\ast} x_{j_\ast}) = \val(\alpha_{j_\ast}) \overset{\ref{cor_minvalal}}{<} \val(a_m) < \val(y)$,
	which yields a contradiction. Hence $\val(x_j)>0$ for all $0 \leq j\leq m$.
\qed	
	
\subsection{Proof of Proposition \ref{prop_mainhensilanity}}\label{subsec: proof of Hens}

We may assume that $(K, \mathcal O_1, \mathcal O_2)$ is $\aleph_0$-saturated. 
Let $\K$ be the algebraic closure of $K$ and let $\mathcal K := \bigcap_{n \in \mathbb{N}} K^{\frac{1}{p^n}}$ be the largest perfect subfield of $K$.
Let $\ell$ be the natural number given by Proposition \ref{prop: chain cond for mult fam of subgrps} for the uniformly defined  subgroups $x_0 \cdot \ldots \cdot x_{n-1} \cdot \wp(K)$ and $x_0 \cdot \ldots \cdot x_{n-1} \cdot \wp(J)$ of $(K,+)$. 

Let $y \in J$ be arbitrary, we will show that it has an Artin-Schreier root in $J$. We stress that this element $y$ will be fixed until the  end of the proof, and several additional parameters depending on this $y$ will be chosen in the course of the proof.
\begin{claim}\label{cla: proof of hens 1}
There exists an infinite sequence $(d_i)_{i \in \mathbb{N}}$ of elements of $\mathcal{K}$ such that $0 < k \cdot \val_{t}(d_{i+1}) < \val_t(d_{i}) < \val_t(y)$ holds for all $i, k \in \mathbb N$ and both $t\in \{1,2\}$ simultaneously.
\end{claim}
\begin{proof}
Let $e_0 := y$, and we define the elements $e_k \in J$ by induction on $k \in \mathbb{N}_{\geq 1}$ as follows: Assume $e_k$ for $k \in \mathbb{N}$ is given. As the field $K$ is Artin-Schreier closed by Fact \ref{fac: k-dep AS closed}, we let $e \in K$ be an Artin-Schreier root of $\frac{1}{e_k}$, i.e.~$\frac{1}{e_k} = e^p - e$. For any $t \in \{1,2\}$ we have: $\val_t(e_k) > 0$ by assumption, hence $\val_t\left( \frac{1}{e_k}\right) < 0 $. Consequently $ \val_t\left( e^p - e \right)< 0$ as well as $\val_t(e)<0$, and so $ \val_t\left( e^p - e \right) = p \val_t(e)$. Hence, letting $e_{k+1} := \frac{1}{e}$, we obtain that $$\val_t(e_k) = - \val_t \left( \frac{1}{e_k} \right) = - p \val_t(e) = p \val_t \left(\frac{1}{e} \right)=p \val_t \left(e_{k+1} \right).$$ 
It follows from the construction that for every $k \in \mathbb{N}_{\geq 1}$ the elements $e_k \in K$ satisfy $0 < \val_t(e_k) = \frac{1}{p^k}\val_t(y) < \val_t(y)$ for both $t\in \{1,2\}$ simultaneously. Let now $m,k \in \mathbb{N}$ be arbitrary, and we define $d_i := e_{ik+m}^{p^{m}}$ for $i\in \mathbb N$. Note that 
\begin{align*}
	\val_t(d_i)& = \val_t( e_{ik+m}^{p^{m}}) = p^{m} \val_t(e_{ik+m}) = p^{m} \frac{1}{p^{ik+m}} \val_t(y) \\
	&= \frac{p^k}{p^k} p^{m} \frac{1}{p^{ik+m}} \val_t(y) = {p^k} p^{m} \frac{1}{p^{(i+1)k+m}} \val_t(y)  \\
	&= p^k p^m\val_t(e_{(i+1)k+m})= p^k  \val_t(e_{(i+1)k+m}^{p^m})\\
	&= p^k \val_t(d_{i+1})
	\end{align*} 
and $p^{k-1} \val_t(d_{i}) = \frac{ p^{k-1} }{p^{ik}} \val_t(y)< \val_t(y)$ for all $i\in \mathbb N$ . Then:
\begin{itemize}
	\item $d_i \in K^{\frac{1}{p^{m}}}$  for all $i \in \mathbb N$;
	\item for each $t \in \{1,2\}$  and all $i \in \mathbb N$ we have: $0 <p^{k-1} \cdot  \val_{t}(d_{i+1}) < \val_t(d_{i}) < \val_t(y)$.
\end{itemize} 
As $m,k$ were arbitrary and $\mathcal{K}$ is type-definable over $\emptyset$, the claim follows by saturation of $(K, \mathcal{O}_1, \mathcal{O}_2)$.
\end{proof}

Now let $(d_i)_{i \in \mathbb{N}}$ be a sequence in $\mathcal{K}$ given by Claim \ref{cla: proof of hens 1}. Then we can choose from it elements $\{b_{j,l}: j < n, l< \ell\}$ in $\mathcal K$ such that: for all $j < n, l< \ell$ we have
\begin{itemize}
	\item $\val_t(b_{n-1,l}) < \val_t(b_{0,l+1})  $,
	\item  $ 0 <(j+1) \cdot \val_t(b_{j,l}) < \val_t(b_{j,l+1}) $,
	\item $ n \cdot \val_t(b_{n-1,\ell-1}) <  \val_t(y)$
	\end{itemize}
	for both $t \in \{1,2\}$ simultaneously.
	For each $(l_0, \dots, l_{n-1} )\in \ell^n$ we define $$b_{l_0, \dots, l_{n-1}} := \prod_{j=0}^{n-1} b_{j, l_j} \in \mathcal{K}.$$ 

\begin{claim}\label{cla: proof of hens 2}
For each $t \in \{1,2\}$ we have: 
\begin{itemize}
\item $0 < \val_t(b_{l_0, \dots, l_{n-1}}) < \val_t(y)$ for all $(l_0, \dots, l_{n-1} )\in \ell^n$;
	\item $\val_t( b_{l_0, \dots, l_{n-1}})< \val_t (b_{p_0, \dots, p_{n-1}})$ if and only if $(l_{n-1}, \ldots, l_0)<_{\operatorname{lex}} (p_{n-1}, \ldots, p_0 )$.
\end{itemize}
\end{claim}
\begin{proof}
The first item is clear by the choice of $b_{j,l}$, and we check the second one.
Assume $(l_{n-1}, \ldots, l_0)<_{\operatorname{lex}} (p_{n-1}, \ldots, p_0 )$, and let $0 \leq j^* < n$ be maximal such that $l_{j^*}  < p_{j^*}$.
Then by the choice of $b_{j,l}$ we have
\begin{align*}
&\val_t(b_{l_0, \ldots, l_{n-1}}) = \sum_{j=0}^{n-1} \val_t(b_{j, l_j}) = \sum_{j=0}^{j^*} \val_t(b_{j,l_j}) + \sum_{j=j^* + 1}^{n-1} \val_t(b_{j,p_j})\\
&\leq (j^*+1) \cdot \val_t(b_{j^*, l_{j^*}}) + \sum_{j=j^* + 1}^{n-1} \val_t(b_{j,p_j}) < \val_t(b_{j^*, p_{j^*}}) + \sum_{j=j^* + 1}^{n-1} \val_t(b_{j,p_j})\\
& \leq  \sum_{j=0}^{n-1} \val_t(b_{j, p_j}) =  \val_t(b_{p_0, \ldots, p_{n-1}}).
\end{align*}
\end{proof}
By the choice of $\ell$ and Proposition \ref{prop: chain cond for mult fam of subgrps}, there must exist some $(l^*_0, \ldots, l^*_{n-1}) \in \ell^n$ such that
$$\bigcap_{(l_0, \ldots, l_{n-1}) \in \ell^n} b_{l_0,\ldots, l_{n-1}}\cdot \wp(J) = \bigcap_{(l_0, \ldots, l_{n-1}) \in \ell^n \setminus \{(l^*_0, \ldots, l^*_{n-1}) \}} b_{l_0,\ldots, l_{n-1}}\cdot \wp(J),$$
$$\bigcap_{(l_0, \ldots, l_{n-1}) \in \ell^n} b_{l_0,\ldots, l_{n-1}}\cdot \wp(K) = \bigcap_{(l_0, \ldots, l_{n-1}) \in \ell^n \setminus \{(l^*_0, \ldots, l^*_{n-1}) \}} b_{l_0,\ldots, l_{n-1}}\cdot \wp(K).$$
 Let now $m := \ell^n - 1$. The following is straightforward:
 \begin{claim}\label{cla: proof of hens 3}
There exists a tuple $(a_j: 0 \leq j \leq m)$ enumerating the set 
 $$\left\{ b_{l_0, \dots, l_{n-1}} : (l_0, \ldots, l_{n-1}) \in \ell^n \right\}$$
  so that the following holds:
 	 \begin{enumerate}
 	\item $\bigcap_{j=0}^{m} a_j \wp(J)=  \bigcap_{j=0}^{m-1} a_j \wp(J)$;
 	\item $\bigcap_{j=0}^{m} a_j \wp(K)=  \bigcap_{j=0}^{m-1} a_j \wp(K)$;%(NEED THIS SO THAT $\pi$ AND THEREFORE $\rho$ IS ONTO)
 	\item $0 < \val_t(a_0)< \dots < \val_t(a_{m-1})$ and 
 $\left( \val_t(a_j) \right)_{0 \leq j \leq m}$ are pairwise distinct, for both $t \in \{1,2\}$.
 \end{enumerate}
 \end{claim}

From now on we fix such a tuple $(a_j : 0 \leq j \leq m)$ for the rest of the proof. Note that we still have $a_j \in \mathcal{K}$  and $0 < \val_t(a_j)< \val_t (y)$ for any $0\leq j \leq m$ and $t \in \{1,2\}$ by the choice of the elements $b_{l_0, \ldots, l_{n-1}}$ and Claim \ref{cla: proof of hens 2}.
Let $\bar{a} := (a_0, \ldots, a_m)$ and $\bar{a}' := (a_0, \ldots, a_{m-1})$, both tuples satisfy the assumption of Subsection \ref{sec: f on val}.
Then the following diagram is commutative:
$$\xymatrix{
G_{\bar{a}}(\mathbb{K}) \ar[0,1]^{\pi} \ar[d]^\simeq_{f_{\bar{a}}}
 & G_{\bar{a}'}(\mathbb{K}) \ar[d]^\simeq_{f_{\bar{a}'}}\\
(\mathbb{K}, +)  \ar[0,1]^{\rho} & (\mathbb{K},+) \\
(\mathbb{K},+),  \ar[u]_\simeq^{\mu} \ar[ur]^{\rho'}}$$
where $\pi$ is the natural projection $(x_0, \ldots, x_{m}) \mapsto (x_0, \ldots, x_{m-1})$; the maps  
$$f_{\bar{a}}: G_{\bar a}(\mathbb{K}) \rightarrow (\mathbb{K},+), \bar x = (x_0, \ldots, x_m) \mapsto \sum_{j=0}^{m} \alpha_j x_j, \textrm{ and}$$
$$f_{\bar{a}'}: G_{\bar a'}(\mathbb{K}) \rightarrow (\mathbb{K},+), \bar x = (x_0, \ldots, x_{m-1}) \mapsto \sum_{j=0}^{m-1} \alpha'_j x_j$$
with $\bar{\alpha}=(\alpha_j : 0 \leq j \leq m), \bar{\alpha}' = (\alpha'_j : 0 \leq j \leq m-1)$ in $\mathcal{K}$ are the isomorphisms given by Fact \ref{fac: expl iso fa} for $\bar a$ and $\bar a'$ respectively; $\rho$ is the algebraic morphism over $\mathcal{K}$ that makes the rectangle commute; and $\mu, \rho': (\mathbb{K},+) \to (\mathbb{K},+)$ is given by
$$ \mu(t) := \alpha_m \cdot t, \quad \rho'(t) :=  \rho(\alpha_m t). $$ 
Note that all these groups and morphisms are defined over $\mathcal{K} \subseteq K$, hence the diagram still commutes with $\mathbb{K}$ replaced by $K$.

We recall the following general fact about additive polynomials.
\begin{fact}\cite[Remark 4.2]{kaplan2011artin} \label{fac: add poly class}
	Let $F$ be an algebraically closed fields, and $g: F \to F$ an additive polynomial (i.e.~$g(t+t') = g(t) + g(t')$ for all $t,t' \in F$) with $\ker(g) = \mathbb{F}_p$. Then $g = c \cdot (t^p - t)^{p^k}$ for some $k \in \mathbb{N}$ and $c \in F$.
\end{fact}

Using it, as in the proof of \cite[Theorem 4.3]{kaplan2011artin}, we get the following explicit expression for $\rho'$ (we include a proof for completeness).
\begin{claim}\label{cla: proof of hens 4}
$ \rho'(t)= c(t^p- t)$ for some $c \in \mathcal K$.
\end{claim}
\begin{proof}
	We have $\pi \restriction_{G_{\bar{a}}(K)}$ is onto $G_{\bar{a}'}(K)$ by Claim \ref{cla: proof of hens 2}(1), hence $\rho \restriction_{K}$ is onto as well. From the diagram we have $|\ker(\rho)| = |\ker(\pi)| = p$, and $\rho$ is an algebraic group morphism of $(\mathbb{K},+)$, hence an additive polynomial. 
	Note that $0 \neq \alpha_{m} \in \ker(\rho)$ as $f_{\bar{a}}^{-1}(\alpha_m) = (0,\ldots, 0, 1) \in K^{m+1}$, and $\pi \left ((0,\ldots, 0, 1) \right) = (0, \ldots, 0) \in K^{m}$.  It follows that $\rho': \mathbb{K} \to \mathbb{K}$ is also an additive polynomial, with $\ker(\rho') = \mathbb{F}_p$. By Fact \ref{fac: add poly class} we have $\rho'(t) = c \cdot (t^p - t)^{p^k}$ for some some $c \in \mathbb{K}$ and $k \in \mathbb{N}$. In fact, $c \in \mathcal{K}$ as $\alpha_m \in \mathcal{K}$ and $\rho$ is over $\mathcal{K}$. Finally, we must have $k=0$ as the degree of $\pi$ as an algebraic morphism is $p$, hence the degree of $\rho'$ is also $p$ as the vertical arrows $f_{\bar{a}}, f_{\bar{a}'}, \mu$ are algebraic isomorphisms.
\end{proof}

We fix $c \in \mathcal{K}$ given by Claim \ref{cla: proof of hens 4} for the rest of the proof. The following is a crucial claim relying on the analysis of the effect of the special isomorphisms $f_{\bar{a}}$ on the valuation in Section \ref{sec: f on val}.

\begin{claim}\label{cla: proof of hens 5} Let $u \in K$ be arbitrary with $\val_t(u)> \max\{\val_t(a_{m-1}), \val_t(a_m)\}$ for both $t \in \{1,2\}$.  Then there exists some $w \in J$ with $ \val_t(w)< \val_t(u)$  for both $t \in \{1,2\}$ and such that $\rho'(w)=u$.
	
	\end{claim}

\begin{proof}

Let $(x_0, \dots, x_{m-1}):= (f_{\bar{a}'})^{-1}(u) \in G_{\bar a'} (K)$. Note that 
$$0 < \val_t(a_0) < \ldots < \val_t(a_{m-1}) < \val_t(u)$$
 for both $t \in \{1,2\}$ by Claim \ref{cla: proof of hens 3}(3) and assumption,  hence by Lemma \ref{lem_xi0} we have that $\val_t(x_j)>0$ for all $t \in \{1,2\}$ and $0\leq j \leq m-1$. Whence $x_j \in J$ for all $0\leq j \leq m-1$. By Claim \ref{cla: proof of hens 3}(1) there is some $x_{m}\in J$ such that $a_{m} (x_{m}^p-x_{m})= a_{j} (x_{j}^p-x_{j})$ for all $0\leq j < m$. So there is a preimage of $(x_0, \dots, x_{m-1})$ under $\pi$, namely $(x_0, \dots, x_{m-1}, x_m) $,  which lies in $J^{m+1}$. Now let
$$ w :=\alpha_m^{-1}f(x_0, \dots, x_{m-1}, x_m) = \alpha_m^{-1} \sum_{j=0}^{m} \alpha_j x_j,$$
 thus $w$ is a preimage of $u$ under $\rho'$. Then for each $t \in \{1,2\}$ we have
\begin{align*}
	\val_t(w)&= \val_t \left(\alpha_m^{-1} \sum_{j=0}^{m} \alpha_j x_j\right)\\
	&= - \val_t\left(\alpha_m \right) + \val_t \left(\sum_{j=0}^{m} \alpha_j x_j\right)\\
	&\overset{\ref{lem_ordxi}(1)}{=}  \left\{ \begin{array}{cc} 
		                - \val_t(\alpha_m) + \val_t ( \alpha_m)+ \val_t( x_m) &  \mbox{ if } \val_t(a_{m-1})< \val_t(a_{m}) \\
		                - \val_t(\alpha_m) + \val_t ( \alpha_{m-1}) +\val_t( x_{m-1})&  \mbox{ if } \val_t(a_{m})< \val_t(a_{m-1})
		                \end{array} \right.\\
					 &>0,
\end{align*}
where the last inequality is by Remark \ref{rem: permuting alphas}, Lemma \ref{lem_valalpha} and Lemma \ref{lem_xn0}. Also, by Lemma \ref{lem_xn0} with respect to $\bar{a}'$, we have $\val_t(x_{m-1}) = \val_t(u) - \val_t(a_{m-1})$, hence 
$$\val_t(u)= \val_t(a_{m-1}) + \val_t(x_{m-1}) \overset{\ref{rem: OFE}}{=}  \val_t(a_{m}) + \val_t(x_{m}).$$
If $\val_t(a_{m-1}) < \val_t(a_m)$, then we have 
\begin{align*}
	\val_t(u) &= \val_t(a_{m}) + \val_t(x_{m})\overset{\ref{cor_minvalal} }{=} \val_t(\alpha_{m}) + \val_t(x_{m})  \\
	&> - \val_t(\alpha_m) + \val_t ( \alpha_m)+ \val_t( x_m) = \val_t(w).
	\end{align*}
If $\val_t(a_{m}) <  \val_t(a_{m-1})$, then we have 
\begin{align*}
	\val_t(u) &= \val_t(a_{m-1}) + \val_t(x_{m-1})\overset{\ref{cor_minvalal} }{=} \val_t ( \alpha_{m-1}) + \val_t(x_{m-1})\\
	& > - \val_t(\alpha_m) + \val_t ( \alpha_{m-1}) +\val_t( x_{m-1})  = \val_t(w).
\end{align*} 
In either case, we obtain $\val_t(u) > \val_t(w)$.
\end{proof}

Let now $w$ be as given by Claim \ref{cla: proof of hens 5} for $u := y$. Then for both $t \in \{1,2\}$ we have
$$\val_t(c)= \val_t(y) - \val_t(w^p-w)= \val_t(y) - \val_t(w) >0. $$
Then $\val_t(cy)> \max\{\val_t(a_{m-1}), \val_t(a_m)\}$. Let $w' \in J$ be as given by Claim \ref{cla: proof of hens 5} applied to $u := cy$, then 
$$ cy = \rho'(w')= c((w')^p-w'),$$
i.e.~$y= (w')^p-w'$. Thus $w'$ is an Artin-Schreier root of $y$ in $J$. As $y \in J$ was arbitrary, this finishes the proof of Proposition \ref{prop_mainhensilanity}.

\subsection{Generic multi-ordered/multi-valued fields}\label{sec: gen multi-ordered}

An analog of Corollary \ref{cor: main hens} for valued fields of characteristic $0$ currently appears out of reach. In this section we at least provide some evidence towards it by demonstrating that the model-companion of the theory of fields with several valuations and orderings introduced by van den Dries \cite[Chapter III]{LouThesis} is not $n$-dependent for any $n$. We use Johnson's PhD thesis \cite[Chapter 11]{johnson2016fun} as our reference.

Fix $k \in \mathbb{N}$. For each $1 \leq i \leq k$, let $T_i$ be one of the theories $\ACVF$ (Algebraically Closed Valued Fields), $\RCF$ (Real Closed Fields) or $\pCF$ ($p$-adically Closed Fields), and let $\mathcal{L}_i$ denote the language of $T_i$ and $\mathcal{L}_i \cap \mathcal{L}_j = \mathcal{L}_{\textrm{rings}}$ (i.e.~$\mathcal{L}_i$ additionally contains a binary predicate $x <_i y$ if $T_i$ is $\RCF$, or $\val_i(x) < \val_i(y)$ if $T_i$ is $\ACVF$ or $\pCF$). Let $\mathcal{L} := \bigcup_{i=1}^{k} \mathcal{L}_i$, and let $T_0 := \bigcup_{i=1}^k \left( T_i \right)_{\forall}$.

\begin{fact}\label{fac: properties of T}
\cite[Theorem 11.2.3]{johnson2016fun} The theory $T_0$ has a model companion $T$, and $K \models T_0$ is a model of $T$ if:
	\begin{enumerate}
	\item $K$ is existentially closed with respect to finite extensions, i.e.~if $L$ is a finite algebraic extension of $K$ and $L \models T_0$, then $L=K$;
	\item For any $m$, let $V$ be an $m$-dimensional absolutely irreducible variety over $K$. For $1 \leq i \leq k$, let $\varphi_i(x)$ be a $V$-dense quantifier-free $\mathcal{L}_i$-formula with parameters from $K$. Then $\bigcap_{i=1}^k \varphi_i (K) \neq \emptyset$.
	
\noindent (Where ``\emph{$V$-dense}'' means that $\varphi_i(K)$ is Zariski-dense in $V(K^{\alg})$, see \cite[Section 11.1.1]{johnson2016fun}.)
	\end{enumerate}
%	\item \cite[Theorem 11.2.15]{johnson2016fun} In $T$, every $\mathcal{L}$-formula $\varphi(x)$ is equivalent to one of the form 
%		$$\exists y P(y,x) = 0 \land \psi(y,x),$$
%		where $y$ is a singleton, $\psi(y,x)$ is quantifier-free, and $P(y,x)$ is a polynomial in $y$ and the coordinates of $x$, with integer coefficients, monic as a polynomial in $y$.
	\end{fact}
	
We use the following result established in the proof of \cite[Claim 11.5.2]{johnson2016fun}.
\begin{fact} \label{fac: Johnson multival}Let $K \models T$.
	For each $i$, let $\chi_i(y)$ be the formula: 
	\begin{enumerate}
		\item $y>_i0$ if $T_i$ is $\RCF$;
		\item $\val_i(y-\frac{1}{4}) > 0$ if $T_i$ is $\ACVF$ or $\pCF$.
	\end{enumerate}
	Let $\chi(y) := \bigwedge_{i=1}^k \chi_i(y)$. Then $\chi(K)$ is infinite, and there exists some $\mathcal{L}$-formula $\psi(x,y)$ such that:
	for any $m \in \mathbb{N}$, any $a_1, \ldots, a_m \in \chi(K)$ pairwise distinct and any $A \subseteq \{1, \ldots, m \}$ there exists some $b$ such that $\models \psi(b, a_j) \iff j \in A$.
\end{fact}
This immediately implies that $T$ is not dependent, and we show that the argument can be generalized to show that $T$ is not $n$-dependent for any $n$ as follows.	
\begin{prop}
		$T$ is not $n$-dependent for any $n \geq 1$.
\end{prop}
\begin{proof}
Let $K \models T$ be a saturated model.
Fix $n \in \mathbb{N}$. Note that it is enough to find some sequences $(c^1_{\alpha_1}, \ldots, c^n_{\alpha_n}: \alpha_1, \ldots, \alpha_n \in \omega)$ of elements in $K$ and $e \in K$ such that all elements in the set $\{c^1_{\alpha_1} \cdot \ldots \cdot c^n_{\alpha_n} + e : \alpha_1, \ldots, \alpha_n \in \omega \}$ are \emph{pairwise distinct} and satisfy $\chi(y)$, as then the formula $\psi'(x; y_1, \ldots, y_n) = \psi(x; y_1 \cdot \ldots \cdot y_n + e)$ with $\psi$ given by Fact \ref{fac: Johnson multival} is not $n$-dependent.

Since the formulas $\val_i(\frac{1}{4} - x) > 0$ and $x >_i 0$ are $V$-dense for $V = \mathbb{A}^1$, by Fact \ref{fac: properties of T}(2) we can choose  $e \in K$  such that  $0 < \val_i(\frac{1}{4} - e) < \infty$ for all $1 \leq i \leq k$ for which $T_i$ is $\ACVF$ or $\pCF$, and $e >_i 0$ for all $i$ for which $T_i$ is $\RCF$.  Let  $\gamma_i := \val_i(\frac{1}{4} - e)$.

By induction on $1 \leq t \leq n$ we choose, using saturation of $K$, sequences $(c^t_\alpha : \alpha \in \omega)$ in $K$ such that the following holds for each $1\leq i \leq k$:
\begin{enumerate}
\item If $T_i$ is $\ACVF$ or $\pCF$:

 \begin{enumerate}

	\item $\val_{i}(c^t_{\alpha + 1}) > n \cdot \val_i(c^t_\alpha) > \gamma_i$ for all $1 \leq t \leq n$ and $\alpha \in \omega$;
	\item $\val_i(c^{t+1}_\alpha) > \val_i(c^t_\beta)$ for all $ 1 \leq t \leq n-1$ and $\alpha, \beta \in \omega$.
	\end{enumerate}
	
\item If $T_i$ is $\RCF$:
 \begin{enumerate}

	\item $c^t_{\alpha+1} >_i (c^t_\alpha)^n > 0$ and  for all $1 \leq t \leq n$ and  $\alpha, \beta \in \omega$;
	\item $c^{t+1}_\alpha >_i c^t_\beta$ for all $1 \leq  t  \leq n-1$ and $\alpha, \beta \in \omega$.
	\end{enumerate}
\end{enumerate} 
To chose an element $c^t_\alpha \in K$ we only need to satisfy finitely many quantifier-free formulas with parameters from $\{c^s_\beta : s < t \lor (s=t \land \beta < \alpha) \} \subseteq K$. All of these are implied by a single condition of the form $\val_i(x) > \val_i(c)$ or $x >_i c$ for each $i$ and some $c \in K$. Thus they can be satisfied in $K$ by Fact \ref{fac: properties of T}(2) since these formulas are $V$-dense for $V = \mathbb{A}^1$.

Assume first that $T_i$ is $\ACVF$ or $\pCF$ for some $1 \leq i \leq k$. Note that in this case for any $(\alpha_1, \ldots, \alpha_n) \in \omega^n$ we have $\val_i \left(\prod_{t=1}^n c^t_{\alpha_t}\right) = \sum_{t=1}^{n} \val_i(c^t_{\alpha_t}) > \gamma_i = \val_i(\frac{1}{4} - e)$ by 1(a). Then $\val_i \left( \left(\prod_{t=1}^{n} c^t_{\alpha_t} + e \right) - \frac{1}{4} \right) = \val_i \left(\prod_{t=1}^{n} c^t_{\alpha_t} - (\frac{1}{4}-e) \right) =  \val_i \left(\frac{1}{4} - e \right) = \gamma_i > 0$, and therefore $\models \chi_i \left( \prod_{t=1}^{n} c^t_{\alpha_t} + e \right)$.

As in the proof of Claim \ref{cla: proof of hens 2}, we get that $(\alpha_1, \ldots, \alpha_n) < (\beta_1, \ldots, \beta_n)$  in the lexicographic ordering on $\omega^n$ if and only if 
$$\val_i \left( \prod_{t=1}^n c^t_{\alpha_t} \right) < \val_i \left( \prod_{t=1}^n c^t_{\alpha_t}\right),$$
 so in particular $\prod_{t=1}^n c^t_{\alpha_t} +e \neq \prod_{t=1}^n c^t_{\beta_t} +e$, hence all these elements are pairwise distinct.

 Otherwise $T_i$ must be $\RCF$ for all $1\leq i \leq k$. Then a similar calculation using 2(a) and 2(b) shows that $\prod_{t=1}^n c^t_{\alpha_t}+e \models \chi_i$, and that all these elements are pairwise distinct.
\end{proof}

\section{Connected components of $n$-dependent groups}\label{sec: n-dep conn comp}
\subsection{Connected components of (type-)definable groups}\label{sec: Connected components}

We begin by recalling some facts about model-theoretic connected components of (type-)definable groups and state the main theorem of the section.
\begin{defn} Let $T$ be an arbitrary theory, $G = G(\C)$ an $\emptyset$-definable group and $A \subseteq \C$ a small set of parameters. Then $G^0_A$ is the intersection of all $A$-definable subgroups of $G$ of finite index, and $G^{00}_A$ is the intersection of all subgroups of $G$  that are type-definable over $A$ and of bounded index.
\end{defn}
As $A$ is small, $G^0_A$ and $ G^{00}_A$ are type-definable subgroups of $G$ of bounded index.	 In fact we have the following standard results, see e.g.~\cite[Remark 3.5(2) and Lemma 2.2]{Kuba}.

\begin{fact}\label{fac: Lem_CtblIntersec}
	If $H$ is a type-definable subgroup of $G = G(\mathbb{M})$ of bounded index, then it can be written as an intersection of subgroups of $G$ of bounded index each of which is defined by a partial type consisting of countably many formulas.
	\end{fact}
	
\begin{fact}\label{Lem_[G:H]small}
	Let $G = G(\C)$  be a type definable group and let 
	$H$ be a subgroup of $G$ of bounded index which is type-definable over a small set of parameters $C$. Then $[G:H]\leq 2^{|T|+|C|} $.
	\end{fact}

	A fundamental fact about dependent groups is the absoluteness of their connected components:
\begin{fact}	 \cite{shelah2008minimal} Let $T$ be dependent. Then
 $G^{00}_A = G^{00}_{\emptyset}$ for every small set $A$. In particular, the intersection of all subgroups of $G$ that are type definable over a small set of parameters and of bounded index is a normal subgroup type-definable over $\emptyset$ and of index $\leq 2^{|T|}$.
\end{fact}
This does not remain true for $2$-dependent groups:
\begin{expl}
Let $G$ be the group $\mathbb{F}_2^{(\omega)}$, where $\mathbb{F}_2$ is the finite field with $2$ elements. Let $\mathcal{M} :=(G, \mathbb{F}_2, 0, +, \cdot)$ be the structure with $+$ the addition in $G$ and $\cdot$ the bilinear form $(a_i)\cdot(b_i) = \sum_i a_i b_i$ from $G^2$ to $\mathbb{F}_2$. Then $\Th(\mathcal{M})$ is simple and $2$-dependent, and $G^{00}_A = \{ g \in G : g\cdot a = 0 \mbox{ for all } a \in A\}$ (see \cite[Section 3]{hempel2016n}), so the group $G^{00}_A$ gets smaller as $A$ grows.
\end{expl}
However, in this example for any small sets $A,B$ we have $G^{00}_{A  B} = G^{00}_A \cap G^{00}_B$. The following theorem of Shelah shows that, up to a ``small'' error, this holds in an arbitrary $2$-dependent group:

\begin{fact}\label{fac: Shelah 2-dep}\cite{MR3666349} Let $T$ be $2$-dependent, $G = G(\C)$ an $A$-type definable group, $\kappa := \beth_2(|A|+|T|)^+$, $\M \supseteq A$ a $\kappa$-saturated model, and $\bar{b}$ an arbitrary finite tuple in $\C$. Then 
$$G^{00}_{\M \bar{b}} = G^{00}_{\M} \cap G^{00}_{C \bar{b}}$$
 for some $C \subseteq \M$ with $|C| < \kappa$.
	
\end{fact}

In this section, we generalize this result to $n$-dependent groups for arbitrary $n$. In order to state our generalization, we need to introduce an appropriate notion of independence.

\begin{defn}(\emph{$\kappa$-coheirs}) For any cardinal $\kappa$, any model $\M$, and any tuple $a$ we write
	$$a \ind_{\M}^{u,\kappa} B$$
if for any set $C \subset B \cup \M$ of size $<\kappa$, $\tp(a/C)$ is realized in $\M$. 
\end{defn}
When $\kappa= \aleph_0$ we get $a\ind^{u, \aleph_0}_{\mathcal{M}}B$ if and only if $\tp(a/B\M)$ is finitely satisfiable in $\mathcal{M}$. In this case we simply write $a \ind^u_{\mathcal{M}} B$ (the usual notation for coheir independence). 

Recall that for an infinite cardinal $\kappa$ and $n \in \omega$, the cardinal $\beth_n(\kappa)$ is defined inductively by $\beth_0(\kappa) = \kappa$ and $\beth_{n+1}(\kappa) = 2^{\beth_n(\kappa)}$. Then the Erd\H{o}s-Rado theorem says that $\left( \beth_r(\kappa) \right)^+ \rightarrow (\kappa^+)^{r+1}_{\kappa}$ for all infinite $\kappa$ and $r\in \omega$.

\begin{defn}(\emph{Generic position})\label{def: generic position}
Let $\M$ be a small model, $A$ a subset of $\M$ and $\bar{b}_1, \ldots, \bar{b}_{n-1}$ finite tuples in $\C$. We say that $(\M, A, \bar{b}_1, \ldots, \bar{b}_{n-1})$ are in a \emph{generic position} if there exist regular cardinals $\kappa_1 < \kappa_2 < \ldots < \kappa_{n-1}$ and models $\M_0 \preceq \M_1 \preceq \ldots \preceq \M_{n-1} = \M$ such that $A \subseteq \M_0$, $\beth_2(|\M_i|)^+ \leq \kappa_{i+1}$ for $i=0, \ldots, n-2$ and 
$$\bar b_i \ind_{\M_i}^{u,\kappa_i} \bar b_{<i} \M_{n-1}$$
for all $1\leq i \leq n-1$. 
\end{defn}

One general method to find elements in a generic position is given in Remark \ref{expl: generic postion}. But first let us state the main result of the section.

\begin{theorem}\label{thm: connected comp type-def main}
Let $n \geq 1$, $T$ an $n$-dependent theory, $A \subseteq \C \models T$ a small parameter set and $G = G(\C)$ a type-definable group in $T$ over $A$ be given. Let $\M \supseteq A$ be a small model and $\bar b_1, \dots, \bar b_{n-1}$ finite tuples in $\C$ such that $(\M, A, \bar{b}_1, \ldots, \bar{b}_{n-1})$ are in a generic position.

 Let $(H_\alpha: \alpha \in I)$ be any family of subgroups of $G$ of bounded index, with each $H_\alpha$ type-definable over $\M \bar b_1 \dots \bar b_{n-1}$. Then there is some $J \subseteq I$ with $|J| \leq \beth_2(|T|+|A|)$ such that 
 \begin{gather*}
 	\bigcap_{\alpha \in I} H_\alpha \cap \bigcap_{i=1, \dots, n-1} G_{\M \cup \bar{b}_1 \cup \ldots \cup \bar{b}_{i-1} \cup \bar{b}_{i+1} \cup \ldots \cup  \bar{b}_{n-1}}^{00} = \\
 	\bigcap_{\alpha \in J} H_\alpha \cap \bigcap_{i=1, \dots, n-1} G_{\M \cup \bar{b}_1 \cup \ldots \cup \bar{b}_{i-1} \cup \bar{b}_{i+1} \cup \ldots \cup  \bar{b}_{n-1} }^{00}.
 \end{gather*}
\end{theorem}

Before we prove Theorem \ref{thm: connected comp type-def main}, let us observe the following corollary for the connected components of $n$-dependent groups.

\begin{cor}\label{cor: G00 2-dep} Let $T$ be a $n$-dependent theory, $A \subseteq \C \models T$ a small parameter set and $G = G(\C)$ a type-definable group in $T$ over $A$ be given. Let $\M \supseteq A$ be a small model and $\bar b_1, \dots, \bar b_{n-1}$ finite tuples in $\C$ in a generic position.
Then there is some $C \subseteq \M$ with $|C| \leq \beth_2(|T|+|A|)$ such that 
$$G^{00}_{\M \cup \bar b_1 \cup \dots \cup \bar b_{n-1}} = \bigcap_{i=1, \dots, n-1} G_{\M \cup \bar{b}_1 \cup \ldots \cup \bar{b}_{i-1} \cup \bar{b}_{i+1} \cup \ldots \cup  \bar{b}_{n-1} }^{00} \cap G^{00}_{C \cup \bar b_1\cup \dots\cup \bar b_{n-1}}.$$

\end{cor}
\begin{proof}
By Fact \ref{fac: Lem_CtblIntersec} we have 
\begin{gather*}
	G^{00}_{\M \cup \bar b_1 \cup \dots \cup \bar b_{n-1}} =\\
	 \bigcap_{i=1, \dots, n-1} G_{\M \cup \bar{b}_1 \cup \ldots \cup \bar{b}_{i-1} \cup \bar{b}_{i+1} \cup \ldots \cup  \bar{b}_{n-1}}^{00} \cap \bigcap \left\{ G^{00}_{B \cup \bar b_1 \cup \dots \cup \bar b_{n-1}} : B \subseteq \M \textrm{ countable} \right\}.
\end{gather*}

By Theorem \ref{thm: connected comp type-def main}, it is already given by some sub-intersection of size at most $\beth_2(|T|+|A|)$. Letting $C \subseteq \M$ be the set containing all of the sets $B$ appearing in this sub-intersection gives the desired result.
\end{proof}

\begin{remark}
 Some variant of Corollary \ref{cor: G00 2-dep} is alluded to in  \cite[Discussion 2.14(2)]{MR3666349}, but we are not aware of any followup.
\end{remark}

\begin{remark}
\begin{enumerate}
	\item For $n=1$ the assumptions of Corollary \ref{cor: G00 2-dep} are trivially satisfied by any sufficiently large model $\M = \M_0$, and the conclusion gives $G^{00}_{\M} = G^{00}_{C}$ for some small subset $C$ of $\M$ (since the first intersection on the right hand side is over the empty set). This easily implies absoluteness of $G^{00}$. 
	\item For $n=2$, the assumption $\bar{b}_1 \ind ^{u, \kappa_1}_{\M_1} \M_{1}$ is clearly satisfied by any $\kappa_1$-saturated model $\M_1 \supseteq A$ (taking $A \subseteq \M_0 \preceq \M_1$ arbitrary), and the conclusion gives $G^{00}_{\M_1 \bar{b}_1} = G^{00}_{\M_1} \cap G^{00}_{\bar{b}_1 C}$ --- hence Fact \ref{fac: Shelah 2-dep} follows.
\end{enumerate}
\end{remark}

\begin{problem}
In order for our proof of Theorem \ref{thm: connected comp type-def main} to work, the  tuples $b_1, \dots, b_{n-1}$ are assumed to be in a generic position (which we think of intuitively as ``sufficiently independent from each other over $\mathcal{M}$''). On the other hand, in the extreme opposite case when one of the tuples is in the algebraic closure of the other ones the result holds trivially. Thus we ask if any hypothesis on the tuples $b_1, \ldots, b_{n-1}$ is needed for the result to hold.
\end{problem}

The following remark provides a general method for finding tuples in a generic position.
\begin{remark}\label{expl: generic postion}
	Let $T$ be an arbitrary $\mathcal{L}$-theory, $A$ a small subset of $\C$ and $G = G(\C)$ a group type-definable over $A$. Let $T' := T^{\Sk}$ be a Skolemization of $T$ in a language $\mathcal{L}' \supseteq \mathcal{L}, |\mathcal{L}'| = |\mathcal{L}|$. 
	Let $\kappa_0 := |A| + |\mathcal{L}|$ and $\kappa_{i+1} := \beth_2\left( \kappa_i \right)^+$, a regular cardinal, for $i=1, \ldots, n-2$.
	Next, choose  mutually $\mathcal{L}'(A)$-indiscernible sequences  
	 $J_1= (\bar b_{1,j}: j\in \kappa_1 +1), \ldots, J_{n-1}= (\bar b_{n,j}: j\in \kappa_n +1)$ in some monster model $\mathbb{M}'$ of $T'$. Then for $i=1, \dots, n-1$, let $I_i= (\bar b_{i,j}: j\in \kappa_i)$, $\bar b_i= \bar b_{i, \kappa_i+1}$ (i.e.~the last element in the sequence  $J_i$), and $\M_i := \Sk(A I_1 \ldots I_i)$ (where $\Sk$ is the operation of taking the Skolem hull). Thus $A \subseteq \M_0 \preceq \M_1 \preceq \ldots \preceq \M_{n-1}$ and $|\M_i| = |A \cup \bigcup_{1 \leq s \leq i} I_s | + |\mathcal{L}'| = \kappa_i$, hence $\beth_2\left(|\M_i|\right)^+ \leq \kappa_{i+1}$ holds for all $0 \leq i \leq n-2$. 
	 
	 We claim that $\bar b_i \ind_{\M_i}^{u,\kappa_i} \bar b_{<i} \M_{n-1}$ for every $1 \leq i \leq n-1$. To see this, fix $i$, and let $C \subseteq \M_{n-1}$ with $|C| < \kappa_i$ be arbitrary. By construction, there exist some $S_i \subseteq \kappa_i$ with $|S_i| < \kappa_i$ such that $C \subseteq \dcl_{\mathcal{L}'}\left( A(b_{1,j})_{j \in S_1} \ldots (b_{n-1,j})_{j \in S_{n-1}} \right)$. As $\kappa_i > |S_i|$ is regular, there exists some $j^* > S_i$ in $\kappa_i$. But then by mutual $\mathcal{L}'(A)$-indiscernibility of $J_1, \ldots, J_{n-1}$ we have that 
$$\bar{b}_{i,j^*} \equiv_{A(b_{1,j})_{j \in S_1} \ldots (b_{n-1,j})_{j \in S_{n-1}} \bar{b}_{<i}} \bar{b}_i,$$
hence $\bar{b}_{i,j^*} \equiv_{C \bar{b}_{<i}} \bar{b}_i$ and $\bar{b}_{i,j^*}$ is in $\M_i$. 

Thus $(A,\M_{n-1}, \bar{b}_1, \ldots, \bar{b}_{n-1})$ are in a generic position (both in the sense of $T$ and $T'$).
 \end{remark}

\subsection{Proof of Theorem \ref{thm: connected comp type-def main}}\label{sec: type-def connected comp}

%\begin{proof}
%Suppose not. Then  there is a sequence $(g_{\gamma})_{\gamma \in  (2^{|T|+|C|})^+}$ of representative of cosets of $H$ in $G$, i.e.~$\left\{g_{\gamma} H: \gamma \in (2^{|T|+|C|})^+\right\}$ is a list of distinct cosets of $H$ in $G$. Now, for $\gamma< \delta< (2^{|T|+|C|})^+$ we have that $g_{\gamma}^{-1}g_{\delta} \not \in H$. Thus there is a formula $\psi_{\gamma, \delta} \in \mathcal{L}(C)$ in the partial type defining $H$ such  that $\models \neg \psi_{\gamma,\delta}(g_{\gamma}^{-1} \cdot g_{\delta})$. As there are at most $\left( |T|+|C| \right)$-many formulas over $C$, by Erd\H{o}s-Rado there exists an infinite subset $I$ of $(2^{|T|+|C|})^+$ and a single formula $\psi$ from the partial type defining $H$ such that for all 
%$\gamma< \delta \in I$ we have that $\psi_{\gamma,\delta}=\psi$. By compactness, for any small cardinal $\kappa$ we can then find a sequence 
%$(k_\gamma)_{\gamma \in \kappa}$ of elements of $G$ such that for any  $\gamma< \delta< \kappa$, we have that  
%$\models \neg \psi(k_{\gamma}^{-1}k_{\delta})$. This implies that $k_{\gamma}^{-1}k_{\delta} \not \in H$ for all $\gamma<\delta<\kappa$, which contradicts that $H$ has bounded index in $G$.
%\end{proof}
We are ready to prove the main theorem.
\proof[ Proof of Theorem \ref{thm: connected comp type-def main}]
Assume that the conclusion fails, and let $A \subseteq \M_0 \preceq \M_1 \preceq \ldots \preceq \M_{n-1} = \M$ witness the generic position as in Definition \ref{def: generic position}. Then, using Fact \ref{fac: Lem_CtblIntersec}, we can find inductively a sequence of $(\M_{n-1} \cup \bar b_1 \cup \dots \cup \bar b_{n-1})$-type-definable subgroups $H_{\alpha}$, $\alpha < \kappa := \beth_2(|T|+|A|)^+$ of $G$ of bounded index such that $H_\alpha  = \bigcap_{m < \omega} \psi^\alpha_m(G; \bar c_{\alpha}, \bar b_1 , \dots , \bar b_{n-1})$ for some countable $\bar{c}_\alpha$ from $\M_{n-1} $, and elements $(d_\alpha)_{\alpha < \kappa}$ in $G$ such that: 

\begin{enumerate}
\item $d_\alpha \in \bigcap_{i=1, \dots, n-1} G_{\M_{n-1} \cup \bar{b}_1 \cup \ldots \cup \bar{b}_{i-1} \cup \bar{b}_{i+1} \cup \ldots \cup  \bar{b}_{n-1} }^{00}  \cap \bigcap_{\beta < \alpha} H_\beta$,
\item $d_\alpha \not\in H_\alpha $.
	\end{enumerate}

\noindent Using compactness, and possibly replacing each $\psi_m^{\alpha}$ by a finite conjunction of  $\psi_i^{\alpha}$'s, we may assume additionally that the following hold:

\begin{enumerate}
\setcounter{enumi}{2}
\item $\models \neg \psi_0^\alpha(d_\alpha; \bar c_{\alpha}, \bar b_1 , \dots , \bar b_{n-1})$,
\item $\{\psi^\alpha_{m+1}(x; \bar c_{\alpha}, \bar b_1 , \dots , \bar b_{n-1}), \psi^\alpha_{m+1}(y; \bar c_{\alpha}, \bar b_1 , \dots , \bar b_{n-1})\}\ \vdash\ \psi^\alpha_{m}(x \cdot y; \bar c_{\alpha}, \bar b_1 , \dots , \bar b_{n-1})\ \wedge\  
\psi^\alpha_{m}(x^{-1}; \bar c_{\alpha}, \bar b_1 , \dots , \bar b_{n-1})\ \wedge\  \psi^\alpha_{m}(xy^{-1}; \bar c_{\alpha}, \bar b_1 , \dots , \bar b_{n-1})$ for all $m \in \omega$.

	\end{enumerate}

As there are only $|T|^{< \aleph_0}$ many formulas and $\operatorname{cf}(\kappa) > |T|^{\aleph_0}$, by the pigeonhole principle and after dropping some of the $H_\alpha$'s, we can find $(\psi_m)_{m\in \omega}$ such that for all $\alpha < \kappa$, $$\psi_m^\alpha= \psi_m.$$

\begin{claim} \label{G00 claim ij}
	In addition to (1)--(4), we may assume that for all $i, j \in \omega$, we have that 
	$$d_i \in H_j \Leftrightarrow i \neq j.$$
	\end{claim}
	
\proof
By Fact \ref{Lem_[G:H]small}, for each $\beta < \kappa$ there is a partition $\{ g_{\beta, \nu}H_\beta: \nu < \theta_{\beta}\}$ of $G$, where $\theta_\beta = [G : H_\beta ] \leq 2^{\aleph_0}$. Therefore, considering $d_\alpha$ for some $\alpha < \kappa$, there is $\nu_{\alpha, \beta} < \theta_\beta$  such that 
$$d_\alpha \in g_{\beta, \nu_{\alpha, \beta}}H_\beta.$$
As $\kappa = \beth_2(|T| + |A|)^+$ by assumption, by Erd\H{o}s-Rado we can find an infinite subset $J$ of $\kappa$ and $\nu<2^{\aleph_0}$, such that for all $\alpha< \beta $ in $J$, we have that 
$$d_\alpha \in g_{\beta, \nu}H_\beta.$$
We may assume that $J = \omega$. Now let $i<j<m \in \omega$. Then by the above we have that
$$d_i\in g_{m, \nu}H_m\ \ \mbox{and}\ \  d_j \in g_{m, \nu}H_m,$$
or in other words
$$d_i =  g_{m, \nu}h_i  \ \ \mbox{and}\ \  d_j = g_{m, \nu} h_j$$
for some $h_i, h_j \in H_m$. Thus

$$d_i^{-1}d_j = h_i^{-1} g_{m, \nu}^{-1} g_{m, \nu} h_j = h_i^{-1}  h_j \in H_m.\ \ \ \ (\ast)$$
Now for $i \in \omega$, let
$$e_i=d_{2i}^{-1}d_{2i+1},\ \ K_i= H_{2i}, \ \  \varphi_m= \psi_{m+1},\ \ \mbox{and}\ \ \bar f_i= \bar c_{2i}. $$
Then, after replacing $d_i$ by $e_i$, $H_i$ by $K_i$, $\psi_i$ by $\varphi_{i}$, and $c_i$ by $f_i$, $(1)$ and $(4)$ are still satisfied. To show that condition $(3)$ remains true, assume the opposite, i.e.~$\models \varphi_0(e_i; \bar f_i, \bar b_1, \dots, \bar b_{n-1} )$, which is equivalent to $\models   \psi_1(e_i; \bar c_{2i}, \bar b_1, \dots, \bar b_{n-1} )$. By $(1)$, we know that $\models \psi_1(d_{2i+1}; \bar c_{2i}, \bar b_1, \dots, \bar b_{n-1} )$. Now, using that $d_{2i}= d_{2i+1}e_i^{-1}$ and $(4)$, we can conclude that $\models  \psi_0(d_{2i}; \bar c_{2i}, \bar b_1, \dots, \bar b_{n-1} )$ --- contradicting $(3)$ for the original sequence. Finally, 
\begin{itemize}
	\item if $i\neq j$, then 
	$$e_i = d_{2i}^{-1}d_{2i+1} \overset{\tiny \mbox{by} (1)\mbox{ if }j<i }{ \underset{\tiny \mbox{by} (\ast)\mbox{ if } i<j}{\in} }H_{2j}= K_j; $$
	
	\item if $i = j$, then by $(1)$ and $(2)$, we have that $d_{2i} \not \in H_{2i}$ but $d_{2i+1} \in H_{2i}$, so $e_i \not \in H_{2i}=K_i$. This also shows that condition $(2)$ is still satisfied.
	\end{itemize}

This finishes the proof of the claim.
\qed$_{\operatorname{claim}}$

Let $\bar {\bf c} = (\bar c_i : i \in \omega)$.

\begin{claim} \label{claim G00main} There are sequences $(\bar{b}_{1,\gamma}: \gamma<\omega), (\bar{b}_{2,\gamma}: \gamma<\omega), \dots, (\bar{b}_{n-1,\gamma}: \gamma<\omega)$ in $\M_{n-1}$ and elements $(d_{i,\gamma_{1}, \dots, \gamma_{n-1}} : (i,  \gamma_1, \dots, \gamma_{n-1})\in \omega^n )$ in $G$ such that
\begin{align*}
	&d_{i,\gamma_{1}, \dots, \gamma_{n-1}} \in H_{j,\delta_{1}, \dots, \delta_{n-1}} :=  \bigcap_{m < \omega} \psi_m(G; \bar c_j, \bar b_{1, \delta_1} , \dots , \bar b_{n-1, \delta_{n-1}}) \\ 
	\Leftrightarrow\ & (i,\gamma_{1}, \dots, \gamma_{n-1}) \neq (j, \delta_{1}, \dots, \delta_{n-1}) \\
	\Leftrightarrow\  & \models \psi_{n-1}(d_{i,\gamma_{1}, \dots, \gamma_{n-1}};\bar c_j, \bar b_{1,\delta_1}, \dots, \bar b_{n-1,\delta_{n-1}}).
	\end{align*}
\end{claim}

\proof
We will prove the following claim by reverse induction on $l=1, \dots, n$.

\

There are sequences $(\bar{b}_{l,\gamma}: \gamma<\omega), (\bar{b}_{l+1,\gamma}: \gamma<\omega), \dots, (\bar{b}_{n-1,\gamma}: \gamma<\omega)$ in $\M_{n-1}$ and elements $\left(d_{i,\gamma_{l}, \dots, \gamma_{n-1}} : (i,  \gamma_l, \dots, \gamma_{n-1})\in \omega^{n-l+1} \right)$ in $G$ such that $(\dagger_1)$--$(\dagger_4)$ below hold:
\begin{align*}
	&d_{i,\gamma_{l}, \dots, \gamma_{n-1}} \in H_{j,\delta_{l}, \dots, \delta_{n-1}}: =   \bigcap_{m < \omega} \psi_m(G; \bar c_{j}, \bar b_1, \dots, \bar b_{l-1}, \bar b_{l, \delta_l} , \dots , \bar b_{n-1, \delta_{n-1}})  \\ 
	\Leftrightarrow\ & (i,\gamma_{l}, \dots, \gamma_{n-1}) \neq (j, \delta_{l}, \dots, \delta_{n-1})& (\dagger_1) \\
	\Leftrightarrow\  & \models \psi_{n-l}(d_{i,\gamma_{l}, \dots, \gamma_{n-1}};\bar c_j,\bar b_1, \dots, \bar b_{l-1}, \bar b_{l,\delta_l}, \dots, \bar b_{n-1,\delta_{n-1}})  & (\dagger_2)\\
	&\mbox{and}\\
	&d_{i,\gamma_{l}, \dots, \gamma_{n-1}} \in  \underset{(\delta_{l}, \dots, \delta_{n-1})\in \omega^{n-l} }{\bigcap_{t=1, \dots, l-1;}} G_{\M_{l-1} \bar {\bf c} \, \bar b_1 \dots \bar b_{t-1} \bar b_{t+1} \ldots \bar b_{l-1} \bar b_{l, \delta_l} \dots \bar b_{n-1, \delta_{n-1}}}^{00} & (\dagger_3)
	\end{align*}
	for all $\gamma_{l}, \ldots, \gamma_{n-1}, \delta_l, \ldots, \delta_{n-1}, i,j \in \omega$ (where $\M_0 \preceq  \dots \preceq \M_{n-1}$ are given by the assumption),
	and
	\begin{align*}	
		 \left. \begin{aligned}
&\psi_{m+1}(x; \bar c_i, \bar b_1, \dots, \bar b_{l-1},  \bar b_{l, \gamma_{l}},  \dots, \bar{b}_{n-1,\gamma_{n-1}}) \ \land \\
&\psi_{m+1}(y; \bar c_i, \bar b_1, \dots, \bar b_{l-1},  \bar b_{l, \gamma_{l}},  \dots, \bar{b}_{n-1,\gamma_{n-1}} ) 
\\
&\vdash\psi_{m}(x \cdot y; \bar c_i,  \bar b_1, \dots, \bar b_{l-2},  \bar b_{l-1, \gamma_{l-1}},  \dots, \bar{b}_{n-1,\gamma_{n-1}} ) \ \ \land\\ 
&\psi_{m}(x^{-1}; \bar c_i, \bar b_1, \dots, \bar b_{l-1},  \bar b_{l, \gamma_{l}},  \dots, \bar{b}_{n-1,\gamma_{n-1}} )\ \ \land\\ 
&\psi_{m}(x \cdot y^{-1}; \bar c_i,\bar b_1, \dots, \bar b_{l-1},  \bar b_{l, \gamma_{l}},  \dots, \bar{b}_{n-1,\gamma_{n-1}} ) \mbox{ for all } i,m \in \omega.  \end{aligned} \ \  \right\} 
  && (\dagger_4)
	\end{align*}

\ 

For $l=1$, this completes the proof of Claim \ref{claim G00main}.

For $l=n$, this is Claim \ref{G00 claim ij}  together with $(1)$, $(3)$ and $(4)$.

Now suppose the claim is true for $1<l<n$, and we want to prove the claim for $l-1$. First,  as $\bar b_{l-1} \ind_{\M_{l-1}}^{u,\kappa_{l-1}} \bar b_{<l-1} \M_{n-1}$, we can choose sequences $( \bar{b}_{l-1, \gamma}: \gamma < \kappa_{l-1})$ in $\M_{l-1}$  and $( d_{i,\gamma_{l-1}, \dots, \gamma_{n-1}}: (i,\gamma_{l-1}, \dots, \gamma_{n-1}) \in \omega \times \kappa_{l-1} \times \omega^{n-l})$ in $G$ such that for any $\gamma< \kappa_{l-1}$:
\begin{align*} % use align* if don't want equation numbered
  \left. \begin{aligned}
    & ( \bar b_{l-1, \gamma}, d_{i,\gamma,\gamma_{l} \dots, \gamma_{n-1}}: (i,\gamma_{l}, \dots, \gamma_{n-1}) \in \omega^{n-l+1})\ \  \mbox{ has the same type as }  \\ 
    & (\bar b_{l-1}, d_{i,\gamma_{l} \dots, \gamma_{n-1}}: (i,\gamma_{l}, \dots, \gamma_{n-1})\in \omega^{n-l+1})\ \ \mbox{over} \\ 
    & \M_{l-2}, \bar {\bf c} ,  
 \bar b_1, \dots, \bar b_{l-2} \cup 
 \{b_{l-1, \delta}\,: \delta < \gamma\}\cup  
\{ \bar b_{j,\delta}: j=l, \dots, n-1, \delta < \omega\}.\\
  \end{aligned} \ \  \right\} 
  && (\star)
\end{align*}
Then 
\begin{enumerate}
	\setlength\itemsep{1em}
\setcounter{enumi}{4}
\item By $(\dagger_3)$ for $t=l-1$ and $(\star)$, using that $\bar b_{l-1, \gamma}\in \M_{l-1}$ for all  $\gamma < \kappa_{l-1}$, we obtain  $$ d_{i, \gamma_{l-1}, \dots, \gamma_{n-1} }  \in \bigcap \{ H_{j,\delta_{l-1}, \dots, \delta_{n-1}}: (j,\delta_{l}, \dots, \delta_{n-1}) \in \omega^{n-l+1}, \delta_{l-1} < \gamma_{l-1}\}$$
for all $i \in \omega$. 
\item By  $(\dagger_1)$ and $(\star)$ 
$$d_{i,\gamma_{l-1},\gamma_l, \dots, \gamma_{n-1} } \in \bigcap \{H_{j,\gamma_{l-1},\delta_{l} \dots, \delta_{n-1}) }: (i,\gamma_{l}, \dots, \gamma_{n-1}) \neq (j,\delta_{l}, \dots, \delta_{n-1} ) \}.$$
\item By $(\dagger_1)$ and $(\star)$ $$d_{i,\gamma_{l-1}, \dots, \gamma_{n-1} } \not \in H_{i,\gamma_{l-1}, \dots, \gamma_{n-1} };$$ 
In particular by $(\dagger_2)$ and $(\star)$  $$\models \neg \psi_{n-l}(d_{i, \gamma_{l-1}, \dots, \gamma_{n-1} }; \bar c_i,\bar b_1, \dots, \bar b_{l-2},  \bar b_{l-1, \gamma_{l-1}},  \dots, \bar{b}_{n-1,\gamma_{n-1}});$$ 
\item  By $(\dagger_4)$ and $(\star)$ the formula 
$$\psi_{m+1}(x; \bar c_i, \bar b_1, \dots, \bar b_{l-2},  \bar b_{l-1, \gamma_{l-1}},  \dots, \bar{b}_{n-1,\gamma_{n-1}}) \land$$
$$\psi_{m+1}(y; \bar c_i, \bar b_1, \dots, \bar b_{l-2},  \bar b_{l-1, \gamma_{l-1}},  \dots, \bar{b}_{n-1,\gamma_{n-1}} )$$
implies the formula
$$\psi_{m}(x \cdot y; \bar c_i,  \bar b_1, \dots, \bar b_{l-2},  \bar b_{l-1, \gamma_{l-1}},  \dots, \bar{b}_{n-1,\gamma_{n-1}} )\land$$
 $$  
\psi_{m}(x^{-1}; \bar c_i, \bar b_1, \dots, \bar b_{l-2},  \bar b_{l-1, \gamma_{l-1}},  \dots, \bar{b}_{n-1,\gamma_{n-1}} ) \land $$
 $$\psi_{m}(x \cdot y^{-1}; \bar c_i,\bar b_1, \dots, \bar b_{l-2},  \bar{b}_{l-1, \gamma_{l-1}},  \dots, \bar{b}_{n-1,\gamma_{n-1}} );$$
\item By $(\dagger_3)$ and $(\star)$, and as $\bar{b}_{l-1, \delta_{l-1}}$ are all in $\M_{l-1}$,  $$d_{i,\gamma_{l-1}, \dots, \gamma_{n-1}} \in \underset{\delta_{l-1}\leq \gamma_{l-1},\ (\delta_{l}, \dots, \delta_{n-1})\in \omega^{n-l} }{\bigcap_{t=1, \dots, l-2;}} G_{\M_{l-2} \bar {\bf c} \, \bar b_1 \dots \bar b_{t-1} \bar b_{t+1} \ldots \bar b_{l-2} \bar{b}_{l-1, \delta_{l-1}}  \dots \bar b_{n-1, \delta_{n-1}}}^{00}.$$
\end{enumerate}

Consider the sequence  of countable tuples $$(\bar b_{l-1,\gamma}, (d_{i,\gamma, \gamma_{l} \dots, \gamma_{n-1} }: (i, \gamma_l, \dots, \gamma_{n-1} )\in \omega^{n-l+1}))_{\gamma\in \kappa_{l-1}}.$$ 
Note that for any $\delta < \kappa_{l-1}$, the group 
$$K_\delta := G^{00}_{\M_{l-2} \bar {\bf c} \bar{b}_1 \ldots \bar{b}_{l-2} \bar{b}_{l-1, \delta} \left\{ \bar b_{l, \delta_{l}} \ldots \bar b_{n-1, \delta_{n-1}} : (\delta_{l}, \dots, \delta_{n-1})\in \omega^{n-l} \right\}}$$
 is type definable over a set of size $|\M_{l-2}| + \aleph_0 = |\M_{l-2}|$ and has bounded index in $G$, hence by Fact \ref{Lem_[G:H]small} its index is at most $2^{|\M_{l-2}|}$. 
 Let $(g_{\delta, \nu} : \nu < 2^{|\M_{l-2}|})$ be a set of representatives of its cosets in $G$. For each $\gamma < \delta < \kappa_{l-1}$, consider a countable tuple 
 $$ \bar{\nu}_{\gamma, \delta} := \left(\nu^{\gamma, \delta}_{i, \delta_l, \ldots, \delta_{n-1}} : i, \delta_{l}, \ldots, \delta_{n-1}\in \omega \right)$$
  listing cosets of the elements $\left(d_{i,\gamma, \gamma_{l} \dots, \gamma_{n-1} }: (i, \gamma_l, \dots, \gamma_{n-1} )\in \omega^{n-l+1} \right)$ with respect to the group $K_\delta$. 
 There are at most $\left(2^{|\M_{l-2}|} \right)^{\aleph_0} = 2^{|\M_{l-2}|}$ possible choices for this tuple. As $\kappa_{l-1} \geq \beth_2(|\M_{l-2}|)^+$ by assumption, applying Erd\H{o}s-Rado  there is an infinite subsequence such that $\bar{\nu}_{\gamma, \delta}$ is constant for all $\gamma < \delta $ from this subsequence. As in the proof of Claim \ref{G00 claim ij}, restricting to this subsequence we have that (5)--(9) still hold, and for any fixed $(i,  \gamma_l, \dots, \gamma_{n-1} )\in \omega^{n-l+1}$ we have 
\begin{align*}
	&d_{i,\gamma', \gamma_{l},  \dots,  \gamma_{n-1} }^{-1} d_{i,\gamma, \gamma_{l}  \dots,  \gamma_{n-1}} \in G^{00}_{\M_{l-2} \bar {\bf c} \bar{b}_1 \ldots \bar{b}_{l-2} \bar{b}_{l-1, \delta} \left\{ \bar b_{l, \delta_{l}} \ldots \bar b_{n-1, \delta_{n-1}} \, : \, (\delta_{l}, \dots, \delta_{n-1})\in \omega^{n-l} \right\}} & ( \dagger \dagger) \\
	&\mbox{ for all } \gamma < \gamma' <  \delta < \omega.
\end{align*}

Next, for $(i,  \gamma_l, \dots, \gamma_{n-1} )\in \omega^{n-l+1}$ and   $\gamma_{l-1} < \kappa_{l-1}$, let 
\begin{align*}
e_{i,\gamma_{l-1}, \dots, \gamma_{n-1} }&= d_{i,2\gamma_{l-1}, \gamma_{l}  \dots, \gamma_{n-1} }^{-1}d_{i,2\gamma_{l-1}+1, \gamma_{l} \dots, \gamma_{n-1}},\\ 
\bar f_{l-1,\gamma_{l-1}}&= \bar b_{l-1, 2\gamma_{l-1}},\\
K_{i,\gamma_{l-1},\gamma_l  \dots, \gamma_{n-1}}&= H_{i, 2\gamma_{l-1}, \gamma_{l}  \dots, \gamma_{n-1}}.
	\end{align*}
We will show that replacing $d_{i,\gamma_{l-1}, \dots, \gamma_{n-1}}$  by $e_{i,\gamma_{l-1}, \dots, \gamma_{n-1}}$, $\bar{b}_{l-1, \gamma_{l-1}}$ by $\bar f_{l-1,\gamma_{l-1}}$, $H_{i,\gamma_{l-1}, \dots, \gamma_{n-1}}$ by $K_{i,\gamma_{l-1}, \dots, \gamma_{n-1}}$ and restricting to the first countably many elements gives the desired sequences for $l-1$.

Note first that 
\begin{align*}
	e_{i,\gamma_{l-1}, \dots, \gamma_{n-1}} \ & \in  \underset{\delta_{l-1} \leq  \gamma_{l-1},\ (\delta_{l}, \ldots, \delta_{n-1})\in \omega^{n-l} }{\bigcap_{t=1, \ldots, l-2;}}  G_{\M_{l-2} \bar {\bf c} \, \bar b_1 \ldots \bar b_{t -1} \bar b_{t+1} \ldots \bar{b}_{l-2}\bar b_{l-1, 2 \delta_{l-1}}, \bar b_{l, \delta_{l}} \dots \bar b_{n-1, \delta_{n-1}}}^{00} & ( \ast \ast)
\end{align*}
by $(9)$.

Now let $(j, \gamma_{l-1}, \dots, \gamma_{n-1}) \neq (i, \delta_{l-1}, \dots, \delta_{n-1})$. We consider two cases:

\noindent {\underline{Case 1:} $ \gamma_{l-1} = \delta_{l-1}$. 

In this case $(j, \gamma_{l}, \dots, \gamma_{n-1}) \neq (i, \delta_{l}, \dots, \delta_{n-1})$. Thus 
$$ d_{i,2\gamma_{l-1}, \gamma_{l},  \dots, \gamma_{n-1} }\in H_{j,2\gamma_{l-1},\delta_{l},  \dots, \delta_{n-1}}\mbox{ by }(6)$$ 
and  
$$d_{i,2\gamma_{l-1}+1, \gamma_{l}, \dots, \gamma_{n-1}}\in H_{j,2\gamma_{l-1},\delta_{l},  \dots, \delta_{n-1}}\mbox{ by }(5).$$
So
\begin{enumerate}
\setcounter{enumi}{9}
\item $e_{i,\gamma_{l-1},  \dots, \gamma_{n-1} }= d_{i,2\gamma_{l-1}, \gamma_{l},  \dots, \gamma_{n-1} }^{-1}
d_{i,2\gamma_{l-1}+1, \gamma_{l} \dots, \gamma_{n-1}} \in H_{j,2\gamma_{l-1},\delta_{l},  \dots, \delta_{n-1}} = K_{j,\delta_{l-1},\delta_{l},  \dots, \delta_{n-1}} $,
\end{enumerate}
and in particular
\begin{enumerate}
\setcounter{enumi}{10}
\item $ \models \psi_{n-l+1} \left(e_{i,\gamma_{l-1}, \dots, \gamma_{n-1}};\bar c_j, \bar f_{l-1,\delta_{l-1}}, b_{l, \delta_{l}} \dots,  \bar b_{n-1,\delta_{n-1}} \right).$
\end{enumerate}

\noindent {\underline{Case 2:}  $ \gamma_{l-1} \neq \delta_{l-1}$.

If $\delta_{l-1} < \gamma_{l-1}$, then $e_{i,\gamma_{l-1},  \dots, \gamma_{n-1} }= d_{i,2\gamma_{l-1}, \gamma_{l},  \dots, \gamma_{n-1} }^{-1}
d_{i,2\gamma_{l-1}+1, \gamma_{l} \dots, \gamma_{n-1}} \in H_{j,2\delta_{l-1},\delta_{l},  \dots, \delta_{n-1}}$ by (5). So without loss of generality $\gamma_{l-1} < \delta_{l-1}$. 
Then $(\dagger \dagger)$ implies in particular that 
 $$e_{i,\gamma_{l-1}, \dots, \gamma_{n-1}} = d_{i,2\gamma_{l-1}+1, \gamma_{l},  \dots,  \gamma_{n-1} }^{-1} d_{i,2\gamma_{l-1}, \gamma_{l}  \dots,  \gamma_{n-1}} \in H_{j,2\delta_{l-1}, \delta_{l}, \dots, \delta_{n-1}}$$
 and
 $$e_{i,\gamma_{l-1}, \dots, \gamma_{n-1}} \  \in  \underset{\delta_{l-1} >  \gamma_{l-1},\ (\delta_{l}, \dots, \delta_{n-1})\in \omega^{n-l} }{\bigcap_{t=1, \dots, l-2;}}  G_{\M_{l-2} \bar{\bf c} \, \bar b_1 \dots \bar{b}_{t-1} \bar{b}_{t+1} \ldots \bar{b}_{l-2} \bar b_{l-1, 2 \delta_{l-1}} \bar b_{l, \delta_{l-1}} \dots \bar b_{n-1, \delta_{n-1}}}^{00}.$$

Together with $(\ast \ast)$, this gives
\begin{enumerate}
\setcounter{enumi}{11}
\item $e_{i,\gamma_{l-1}, \dots, \gamma_{n-1}} \  \in  \underset{(\delta_{l-1}, \dots, \delta_{n-1})\in \omega \times \omega^{n-l} }{\bigcap_{t=1, \dots, l-2;}}  G_{\M_{l-2} \bar {\bf c} \, \bar b_1 \ldots \bar{b}_{t-1} \bar{b}_{t+1} \ldots \bar{b}_{l-2} \bar f_{l-1, \delta_{l-1}} \bar b_{l, \delta_{l-1}} \dots \bar b_{n-1, \delta_{n-1}}}^{00}.$
\end{enumerate}
On the other hand, for given $(j, \gamma_{l-1}, \dots, \gamma_{n-1}) \in \omega^{n-l+2}$,  by (7) we have
$$d_{i,2\gamma_{l-1}, \gamma_{l}  \dots,  \gamma_{n-1}} \not \in H_{i,2\gamma_{l-1}, \gamma_{l}  \dots,  \gamma_{n-1}},$$ 
and by (5) we  have 
$$d_{i,2\gamma_l+1,  \gamma_{l}  \dots,  \gamma_{n-1}} \in H_{i,2\gamma_{l-1},  \gamma_{l},  \dots,  \gamma_{n-1}}.$$ 
Hence
\begin{enumerate}
\setcounter{enumi}{12}
\item $e_{i,\gamma_{l-1}, \dots, \gamma_{n-1} }= d_{i,2\gamma_{l-1}, \gamma_{l}  \dots,  \gamma_{n-1}}^{-1}d_{i,2\gamma_{l-1}+1,  \gamma_{l}  \dots,  \gamma_{n-1}} \not \in H_{i, 2\gamma_{l-1}, \gamma_{l}  \dots,  \gamma_{n-1}} =K_{i,\gamma_{l-1}, \dots, \gamma_{n-1}}.$
\end{enumerate}
	Suppose towards a contradiction that 
	$$ \models \psi_{n-l+1}(e_{i,\gamma_{l-1}, \dots, \gamma_{n-1}};\bar c_i, \bar f_{l-1,\gamma_{l-1}}, b_{l, \gamma_{l}} \dots,  \bar b_{n-1,\gamma_{n-1}}).$$ 
	By $(5)$, we also have that 
	$$\models \psi_{n-l+1}(d_{i,2\gamma_{l-1}+1, \gamma_l,  \dots, \gamma_{n-1}+1}; \bar c_{i},   \bar f_{l-1,\gamma_{l-1}}, \bar b_{l, \gamma_l} \dots,  \bar b_{n-1,\gamma_{n-1}}).$$ 
	Since $d_{i, 2\gamma_{l-1}, \dots, \gamma_{n-1}}= d_{i,2\gamma_{l-1}+1, \gamma_l  \dots, \gamma_{n-1}}e_{i,\gamma_{l-1}, \dots, \gamma_{n-1}}^{-1}$, using $(8)$ we conclude that 
	\begin{align*}\models \  & \psi_{n-l}(d_{i,2\gamma_{l-1}, \gamma_l \dots, \gamma_{n-1}}; \bar c_{i}, \bar f_{l-1,\gamma_{l-1}}, \bar b_{l,\gamma_l} \dots, \bar b_{n-1,\gamma_{n-1}})\\
		&= \psi_{n-l}(d_{i,2\gamma_{l-1}, \gamma_l \dots, \gamma_{n-1}}; \bar c_{i}, \bar b_{l-1,2\gamma_{l-1}}, \dots, \bar b_{n-1,\gamma_{n-1}}),
\end{align*}
 contradicting $(7)$. Thus 
 
\begin{enumerate}
\setcounter{enumi}{13}
\item $ \not \models \psi_{n-l+1}(e_{i,\gamma_{l-1}, \dots, \gamma_{n-1}};\bar c_i, \bar f_{l-1,\gamma_{l-1}}, b_{l, \gamma_{l}} \dots,  \bar b_{n-1,\gamma_{n-1}}).$
\end{enumerate}

Now, for $l-1$, we obtain $(\dagger_1)$ from $(10)$ and $(13)$, $(\dagger_2)$ from $(11)$ and $(14)$, $(\dagger_3)$ from $(12)$, and $(\dagger_4)$ from $(8)$ by replacing $e_{i,\gamma_{l-1}, \dots, \gamma_{n-1}}$ by $d_{i,\gamma_{l-1}, \dots, \gamma_{n-1}}$, $\bar f_{l-1,\gamma_{l-1}}$ by $\bar{b}_{l-1, \gamma_{l-1}}$, and  $K_{i,\gamma_{l-1}, \dots, \gamma_{n-1}}$ by $H_{i,\gamma_{l-1}, \dots, \gamma_{n-1}}$. This finishes the proof of the claim.
\qed$_{\operatorname{claim}}$
\begin{claim}\label{cla: group shattering}
For any $I \subset (\omega)^n$ there is some $d \in G$ such that 
$$\models \psi_{n+1}(d;\bar c_i,b_{1,\gamma_1}, \dots, \bar b_{n-1,\gamma_{n-1}})$$
 if and only if $(i, \gamma_1, \dots, \gamma_{n-1}) \notin I$.
\end{claim}
\proof
Suppose first that $I$ is finite and let $((i_l, \gamma_{1,l}, \dots, \gamma_{n-1, l}): l\leq s)$ be an enumeration of the elements in $I$. Define $d = d_{i_0,\gamma_{1,0}, \dots, \gamma_{n-1, 0}}\cdot \ldots \cdot d_{i_s,\gamma_{1,s}, \dots, \gamma_{n-1,s}}$. 

If $(j,\delta_{1}, \dots, \delta_{n-1})\notin I$, then $d_{i_l,\gamma_{1,l}, \dots, \gamma_{n-1, l}}\in H_{j,\delta_{1}, \dots, \delta_{n-1}}$ for all $l \leq s $ by Claim \ref{claim G00main}, and hence $d \in H_{j,\delta_{1}, \dots, \delta_{n-1}}$ and in particular $\models \psi_{n+1}(d;\bar c_j, \bar b_{1,\delta_{1}}, \dots, \bar b_{n-1,\delta_{n-1}})$.

On the other hand, consider $(i_l, \gamma_{1,l}, \dots, \gamma_{n-1, l}) \in I$. Let $$e_1= d_{i_0,\gamma_{1,0}, \dots, \gamma_{n-1, 0}}\cdot \ldots \cdot d_{i_{l-1},\gamma_{1,l-1}, \dots, \gamma_{n-1, l-1}}$$ and $$e_2 = d_{i_{l},\gamma_{1,l+1}, \dots, \gamma_{n-1, l+1}}\cdot \ldots \cdot d_{i_s,\gamma_{1,s}, \dots, \gamma_{n-1, s}}.$$  Then $d_{i_l,\gamma_{1,l}, \dots, \gamma_{n-1, l}}= e_1^{-1} d e_2^{-1}$. Observe that
\begin{gather*}
	\models \neg \psi_{n-1}(d_{i_l,\gamma_{1,l}, \dots, \gamma_{n-1, l}}, \bar c_{i_l},b_{1,\gamma_{1,l}}, \dots, \bar b_{n-1,\gamma_{n-1,l}}),\\
	\models \psi_{n}(e_2^{-1}, \bar c_{i_l},b_{1,\gamma_{1,l}}, \dots, \bar b_{n-1,\gamma_{n-1,l}}) \textrm{ and}\\
	\models \psi_{n+1}(e_1^{-1}, \bar c_{i_l},b_{1,\gamma_{1,l}}, \dots, \bar b_{n-1,\gamma_{n-1,l}})
\end{gather*}
by Claim \ref{claim G00main}. Using $(8)$, we conclude that
\begin{align*}
	\models &\ \neg \psi_{n-1}(d_{i_l,\gamma_{1,l}, \dots, \gamma_{n-1, l}}, \bar c_{i_l},b_{1,\gamma_{1,l}}, \dots, \bar b_{n-1,\gamma_{n-1,l}})\\
 \rightarrow \ \models & \ \neg \psi_{n}(e_1^{-1}d, \bar c_{i_l},b_{1,\gamma_{1,l}}, \dots, \bar b_{n-1,\gamma_{n-1,l}})\\
  \rightarrow \ \models& \  \neg \psi_{n+1}(d, \bar c_{i_l},b_{1,\gamma_{1,l}}, \dots, \bar b_{n-1,\gamma_{n-1,l}}). \end{align*}
 
The claim  follows by compactness.
\qed$_{\operatorname{claim}}$

Finally, Claim \ref{cla: group shattering} contradicts $n$-dependence of $\psi_{n+1}$, which finishes the proof.
\qed

\section{2-dependence for compositions of dependent relations and binary functions}\label{sec: functions preserve 2-dep}

The aim of this section is prove the following \emph{Composition Lemma} (Theorem \ref{thm: functions into stable are 2-dep } below): a composition of  a relation (of any arity) definable in a model of a dependent theory with arbitrary binary functions is $2$-dependent. 
This result is crucial in our proof of $2$-dependence of non-degenerate bilinear forms over dependent fields in Section \ref{sec: Granger}. Towards this purpose we first develop a general type-counting criterion for $2$-dependent theories in Section \ref{sec: counting types}, and then apply it along with the set-theoretic absoluteness to deduce   the Composition Lemma in Section \ref{sec: Comp Lemma}.

\subsection{Characterization of $2$-dependence by a type-counting criterion}\label{sec: counting types}

We first recall a type-counting criterion characterizing dependent theories. For the following two facts see e.g. \cite{chernikov2016number} or \cite[Section 6]{chernikov2016non} and references there. We recall that for an infinite $\kappa$, $\ded \kappa$ is the supremum of the number of Dedekind cuts among all linear orders of cardinality $\kappa$. 

\begin{fact}(Shelah)\label{fac: NIP by counting types} Let $T$ be a theory in a countable language, and for an infinite cardinal $\kappa$, let $f_T(\kappa) := \operatorname{sup}\{|S_1(M)|: |M|=\kappa, M \models T\}$.
\begin{enumerate}
\item If $T$ is dependent, then $f_T(\kappa) \leq (\ded \kappa)^{\aleph_0}$ for all infinite cardinals $\kappa$.
\item If $T$ is not dependent, then $f_T(\kappa) = 2^\kappa$ for all infinite cardinals $\kappa$.
\end{enumerate}
\end{fact}

In a model of ZFC satisfying the Generalized Continuum Hypothesis, $\ded \kappa = 2^\kappa$ for all infinite cardinals $\kappa$. However, there are models of ZFC in which these two functions are different:

\begin{fact} (Mitchell \cite{mitchell1972aronszajn})\label{fac: Mitchell}
For every cardinal $\kappa$ of uncountable cofinality, there exists a cardinal preserving Cohen extension such that $(\ded \kappa)^{\aleph_0} < 2^\kappa$.
\end{fact}

The combination of Facts \ref{fac: NIP by counting types} and \ref{fac: Mitchell} tells us that it is possible to detect whether a theory is dependent by counting types 
using that dependence of a formula is a set-theoretically absolute property.

We want to provide a formula-free characterization of $n$-dependence of a theory
which does not include any assumption of indiscernibility of the witnessing
sequence over the additional parameters (unlike the characterization in Proposition \ref{prop: char of NIP_k by preserving indisc} where additional indiscernibility over the parameter needs to be assumed). We achieve it here for $2$-dependence by providing an analog of Fact \ref{fac: NIP by counting types} in this case (which characterizes $2$-dependence of a theory when working in a model of ZFC with $\ded(\kappa) < 2^{\kappa}$ for some cardinal $\kappa \geq |T|$).

Given a set $X$ and a family $\mathcal{F}$ of subsets of $X$ and $Y \subseteq X$, one says that \emph{$Y$ is shattered by $\mathcal{F}$} if for every $Z \subseteq Y$ there exists some $S \in \mathcal{F}$ such that $Z = Y \cap S$. In what follows, $T$ is a complete theory in a language $\mathcal{L}$ and we work in a monster model $\mathbb{M} \models T$.

\begin{lemma}
\label{lem: basic no shattering}Let $\varphi\left(x;y_{1},y_{2}\right) \in \mathcal{L}$
be $2$-dependent. Then there is some $n\in\mathbb{N}$ such that
for any $c\in\mathbb{M}_{x}$, any $I\subseteq\mathbb{M}_{y_{1}},J\subseteq\mathbb{M}_{y_{2}}$
endless mutually indiscernible sequences, and any $A\subseteq I$
of size $>n$ there is some $b_{A}\in J$ such that $A$ cannot be
shattered by the family $\left\{ \varphi\left(c,y_{1},b\right):b\in J,b>b_{A}\right\} $.
\end{lemma}

\begin{proof}
Assume that $I,J$ are endless mutually indiscernible sequences and
$c$ is such that the conclusion is not satisfied for any $n\in\omega$.
Let $D\subseteq I\times J$ be any finite set. Let $a_{1}<\ldots<a_{n}$
and $b_{1}<\ldots<b_{m}$ list the projections of $D$ on $I$ and
$J$, respectively. By assumption, there is some $A\subseteq I$ of
size $n$ such that for any $b'\in J$, $A$ is shattered by the family
$\left\{ \varphi\left(c,y_{1},b\right):b\in J,b>b'\right\} $. List $A$
as $a_{1}'<\ldots<a_{n}'$. Then we can choose some $b_{1}'<\ldots<b_{m}'\in J$
such that $\models\varphi\left(c,a'_{i},b_{j}'\right)\iff\left(a_{i},b_{j}\right)\in D$.
As $I,J$ are mutually indiscernible, taking an automorphism of $\mathbb{M}$
sending $a_{i}'$ to $a_{i}$ and $b_{j}'$ to $b_{j}$, for all $1\leq i\leq n,1\leq j\leq m$,
$c$ is sent to some $c_{D}$ such that $\models\varphi\left(c_{D},a_{i},b_{j}\right)\iff\left(a_{i},b_{j}\right)\in D$.
This implies that $\varphi\left(x;y_{1},y_{2}\right)$ is not $2$-dependent,
a contradiction. Hence the conclusion holds for $c,I,J$ for some
$n$.

By compactness we conclude that $n$ can be chosen
depending only on $\varphi$ (and not on $I,J,c$).
\end{proof}
We will need the following fact (originally due to Shelah, with simplifications
by Adler and Casanovas, see e.g. \cite[Lemma 2.7.1]{ChernStability})
\begin{fact}
\label{fact: ded lemma}If $\kappa$ is an infinite cardinal, $\mathcal{F}\subseteq2^{\kappa}$
and $\left|\mathcal{F}\right|>\ded\kappa$, then for each $n\in\omega$
there is some $S\subseteq\kappa$ such that $\left|S\right|=n$ and
$\mathcal{F}\restriction S=2^{S}$.
\end{fact}
\begin{defn}\label{def: localized space of types}
Given sets $B\subseteq\mathbb{M}_{x}$, $A\subseteq\mathbb{M}_{y}$
and a formula $\varphi\left(x,y\right)\in\mathcal{L}$, we denote by
$S_{\varphi,B}\left(A\right)$ the set of all $\varphi$-types over $A$
realized in $B$, where by a $\varphi$-type over $A$ we mean a maximal consistent collection of formulas of the form $\varphi(x,a), \neg \varphi(x,a)$ with $a \in A$.
And $S_{B}\left(A\right)$ denotes the set of all complete
types over $A$ realized in $B$.	
\end{defn}

\begin{prop}
\label{prop: few types on a tail}Let $T$ be $2$-dependent, let
$\kappa\geq\left|T\right|$ be an infinite cardinal, and let $\lambda>\kappa$
be a regular cardinal. Then for any mutually indiscernible sequences
$I=\left(a_{i}:i\in\kappa\right),J=\left(b_{j}:j\in\lambda\right)$
of finite tuples and a finite tuple $c$, there is some $\beta\in\lambda$
such that 
$$\left|S_{J_{>\beta}}\left(I\times\{c\}\right)\right|\leq\left(\ded\kappa\right)^{\left|T\right|}.$$
\end{prop}

\begin{proof}
Let $I,J$ and $c$ be given. We will show that for each $\varphi\left(x;y_{1},y_{2}\right)\in\mathcal{L}$
there is some $\beta_{\varphi}\in\lambda$ such that $\left|S_{\varphi,J_{>\beta_{\varphi}}}\left(I\times\{c\}\right)\right|\leq\ded\kappa$.
This is enough, as then we can take any $\beta\in\lambda$ with $\beta>\beta_{\varphi}$
for all $\varphi\in\mathcal{L}$ (possible as $\lambda=\cof\left(\lambda\right)>\left|T\right|$),
and $\left|S_{J_{>\beta}}\left(I\times\{c\}\right)\right|\leq\left|\prod_{\varphi\in\mathcal{L}}S_{\varphi,J_{>\beta_{\varphi}}}\left(I\times\{c\}\right)\right|\leq\left(\ded\kappa\right)^{\left|T\right|}$.

So let $\varphi(x;y_1,y_2)\in\mathcal{L}$ be fixed, and assume that for any $\beta\in\lambda$,
$\left|S_{\varphi,J_{>\beta}}\left(I\times\{c\}\right)\right|>\ded\kappa$.
Then by Fact \ref{fact: ded lemma}, considering $\mathcal{F}=\left\{ f_{p}:p\in S_{\varphi,J_{>\beta}}\left(I\times\{c\}\right)\right\} $
(where $f_{p}\in2^{\kappa}$ is given by $f_{p}\left(\alpha\right)=1\iff\varphi\left(x;a_{\alpha},c\right)\in p$,
for all $\alpha\in\kappa$), for any $n\in\omega$ there is \emph{some}
$S\subseteq I$, $\left|S\right|=n$, such that $S$ is shattered
by the family $\left\{ \varphi\left(b_j;y_{1},c\right):j\in\lambda,j>\beta\right\} $.
Using regularity of $\lambda$, by transfinite induction we can choose 
a strictly increasing sequence $\left(\beta_{\alpha}:\alpha\in\lambda\right)$
with $\beta_{\alpha}\in\lambda$ such that for each $\alpha\in\lambda$
there is some $S_{\alpha}\subseteq I,\left|S_{\alpha}\right|=n$ shattered
by the family $\left\{ \varphi\left(b_j;y_{1},c\right):j\in\lambda,\beta_{\alpha}<j<\beta_{\alpha+1}\right\} $.
As $\lambda>\kappa=\kappa^{n}$ is regular, passing to a subsequence
we may assume that there is some $S\subseteq I,\left|S\right|=n$
such that $S_{\alpha}=S$ for all $\alpha\in\lambda$, i.e. this set
$S$ can be shattered arbitrarily far into the sequence. Now by Lemma
\ref{lem: basic no shattering}, this contradicts $2$-dependence
of $\varphi$ if we take $n$ large enough.
\end{proof}

\begin{lemma}\label{lem: asym rand graph}
\label{lem: extension exists}For any cardinal $\kappa$ and any regular
cardinal $\lambda\geq2^{\kappa}$ there is a bipartite graph $\mathcal{G}_{\kappa,\lambda}=\left(\kappa,\lambda,E\right)$ (where its parts are identified with the cardinals $\kappa$ and $\lambda$, and $E \subseteq \kappa\times \lambda$ is the edge relation)
satisfying the following: for any sets $A,A'\subseteq\kappa$ with
$A\cap A'=\emptyset$ and any $b\in\lambda$ there is some $b^{*}\in\lambda$,
$b^{*}>b$ satisfying $\bigwedge_{a\in A}E\left(a,b^{*}\right)\land\bigwedge_{a'\in A'}E\left(a',b^{*}\right)$.
\end{lemma}

\begin{proof}
Let $\lambda\geq2^{\kappa}$ be any regular cardinal. Let 
\[
D:=\left\{ \left(A,A',b\right):A,A'\subseteq\kappa,\,A\cap A'=\emptyset,\,b\in\lambda\right\} \mbox{.}
\]
Then $\left|D\right|\leq\lambda$ by assumption, we enumerate it
as $\left(\left(A_{\alpha},A'_{\alpha},b_{\alpha}\right):\alpha<\lambda\right)$.
We define $E_{\alpha}\subseteq\kappa\times\lambda$ by transfinite
induction on $\alpha<\lambda$. On step $\alpha$, we choose some
$c_{\alpha}\in\lambda$ such that $c_{\alpha}>\left\{ b_{\beta},c_{\beta}:\beta<\alpha\right\} $
\textemdash{} possible by regularity of $\lambda$, and we take $E_{\alpha}:=\left\{ \left(a,c_{\alpha}\right):a\in A_{\alpha}\right\} $.
Let $E:=\bigsqcup_{\alpha<\lambda}E_{\alpha}$ \textemdash{} it satisfies
the requirement by construction.
\end{proof}

\begin{prop}\label{lem: pick asym rand graph}
Assume that $\varphi(x,y,z)$ is not $2$-dependent. Then for every regular $\lambda > 2^\kappa$ there exist mutually indiscernible sequences $I=(a_i : i < \kappa)$ in $\mathbb{M}_y$, $(b_j : j < \lambda)$ in $\mathbb{M}_x$ and $c \in \mathbb{M}_z$ such that for every $\beta < \lambda$ we have $|S_{\varphi, J_{> \beta}}(I\times\{c\})| = 2^\kappa$.
\end{prop}
\begin{proof}
	By assumption, for any $\kappa,\lambda$ we can find some
mutually indiscernible sequences $I,J$ such that the family $\left\{ \varphi\left(x,y,c\right):c\in\mathbb{M}_z\right\} $
shatters $I\times J$. In particular, we can find $c$ such that $\mathbb{M}\models\varphi\left(b_j,a_{i},c\right)\iff\mathcal{G}_{\kappa,\lambda}\models E\left(a_{i},b_{j}\right)$, and we can conclude by Lemma \ref{lem: asym rand graph}.
\end{proof}
Propositions \ref{prop: few types on a tail} and \ref{lem: pick asym rand graph} together provide an analog of Fact \ref{fac: NIP by counting types} for $2$-dependent theories that will be used in the proof of the Composition Lemma in Section \ref{sec: counting types}. We conclude this subsection with a brief discussion of some questions arising in connection to this criterion (and not used in the rest of the paper).
\begin{defn}\label{defn_globally 2dep}
We say that a theory $T$ is \emph{globally $2$-dependent} if there
are a cardinal $\kappa$ and a regular
cardinal $\lambda\geq2^{\kappa}$ such that for  $c$ and $\mathcal{G}_{\kappa,\lambda}$ given by Lemma \ref{lem: asym rand graph} the following
holds:  for any mutually indiscernible sequences $I=\left(a_{i}:i\in\kappa\right),J=\left(b_{j}:j\in\lambda\right)$
of finite tuples there are $i\in\kappa$ and $j,j'\in\lambda$
such that $ca_{i}b_{j}\equiv ca_{i}b_{j'}$ but $E\left(i,j\right)\land\neg E\left(i,j'\right)$.
\end{defn}
So the idea is that $T$ is globally $2$-dependent if on mutually
indiscernible sequences we cannot distinguish the edges from the
non-edges of a random graph by a complete type (as opposed to witnessing the edges with
realizations of a single formula). We have the following connection between $2$-dependence and global $2$-dependence.

\begin{prop}
\begin{enumerate}
	\item If $T$ is globally $2$-dependent, then it is $2$-dependent.
	\item Let $T$ be a countable $2$-dependent theory and assume that there
exists a cardinal $\kappa$ such that $(\ded\kappa)^{\aleph_0}<2^{\kappa}$. Then
$T$ is globally $2$-dependent.
\end{enumerate}

\end{prop}

\begin{proof}
(1) If $\varphi\left(x,y_{1},y_{2}\right)$ is a formula witnessing failure
of $2$-dependence, then the proof of Proposition \ref{lem: pick asym rand graph} shows that $T$ is not globally $2$-dependent. 

(2) Fix $\kappa$, and let $\lambda$ be any regular cardinal $\geq2^{\kappa}$.
Let $\mathcal{G}_{\kappa,\lambda}$ be as given by Lemma \ref{lem: extension exists}.
Moreover, let $I,J$ and $c$ be as in Definition \ref{defn_globally 2dep}. By Proposition \ref{prop: few types on a tail},
there is some $\beta\in\lambda$ such that $\left|S_{J_{>\beta}}\left(I\right)\right|\leq\left(\ded\kappa\right)^{\aleph_{0}}$.
On the other hand, by definition of $\mathcal{G}_{\kappa,\lambda}$,
we still have $\left|S_{E,\left\{ \alpha\in\lambda:\alpha>\beta\right\} }\left(\kappa\right)\right|=2^{\kappa}>\left(\ded\kappa\right)^{\aleph_{0}}$
by assumption. Then we can find some $j,j'\in\lambda$ such that $\tp_{E}\left(j/\kappa\right)\neq\tp_{E}\left(j'/\kappa\right)$
but $\tp\left(b_{j}/Ic\right)=\tp\left(b_{j'}/Ic\right)$. In other words,
there is some $i\in\kappa$ such that $E\left(i,j\right)\leftrightarrow\neg E\left(i,j'\right)$
and still $b_{j}a_{i}c\equiv b_{j'}a_{i}c$, as wanted.
\end{proof}
%By a theorem of Mitchell, for any $\kappa$ with $\cof\left(\kappa\right)>\aleph_{0}$
%it is consistent that $\ded\kappa<2^{\kappa}$. Hence this criterion
%can always be used to determine $2$-dependence, in some model of
%ZFC (and then sometimes set-theoretic absoluteness can be applied).
\begin{problem}
Is there an analog of this type-counting criterion for $n$-dependence? Is it true that $n$-dependence implies global $n$-dependence (defined
analogously), in ZFC, or at least consistently for $n>2$?
\end{problem}
Concerning this problem, we remark that at least $n$-dependence implies global $n$-dependence when $T$ is $\aleph_0$-categorical (since every type in finitely many variables
is equivalent to a formula).

\subsection{The Composition Lemma}\label{sec: Comp Lemma}
All of the variables below are allowed to be tuples of arbitrary finite length.

\begin{theorem}[Composition Lemma]\label{thm: functions into stable are 2-dep }
Let $\mathcal{M}_0$ be a dependent structure in a language $\mathcal{L}_0$, and let $\mathcal{M}$ be an \emph{arbitrary} expansion of $\mathcal{M}$ in some language $\mathcal{L} \supseteq \mathcal{L}_0$. Let $\varphi(x_1, \ldots, x_{d})$ be an $\mathcal{L}_0$-formula. For each $i \in \{ 1, \ldots, d \}$, fix some $s_i < t_i \in \{1,2,3 \}$
and let $f_i: M_{y_{s_i}} \times M_{y_{t_i}}  \to M_{x_i}$ be an $\mathcal{L}$-definable binary function.
Then the $\mathcal{L}$-formula 
$$\psi(y_1; y_2, y_3) := \varphi \left (f_1(y_{s_1},y_{t_1}), \ldots, f_d(y_{s_{d}}, y_{t_{d}}) \right)$$ is $2$-dependent (with respect to $\Th_{\mathcal{L}}(\mathcal{M})$).
\end{theorem}
\begin{proof}
We work in a monster model $\mathbb{M}$ of $T := \Th_{\mathcal{L}}(\mathcal{M})$. Note that then the $\mathcal{L}_0$-reduct $\mathbb{M}_0$ of $\mathbb{M}$ is a monster model of $T_0 := \Th_{\mathcal{L}_0}(\mathcal{M}_0)$, and $T_0$ is dependent. Assume that the formula $\psi(y_1; y_2, y_3)$  as in the statement of the theorem is not $2$-dependent. Let $I=(a_\alpha : \alpha < \kappa)$ with $a_\alpha \in M_{y_2}$, $J=(b_\beta: \beta < \lambda)$ with $b_\beta \in M_{y_1}$ and $c \in M_{y_3}$ be as given by Proposition \ref{lem: pick asym rand graph} with $\lambda > 2^\kappa > |T|$, that is for every $\gamma < \lambda$ we have $|S_{\psi, J_{> \gamma}}(I\times\{c\})| = 2^\kappa$.

 Let $V_{1,2} := \{ i \leq d : (t_i, s_i) = (1,2) \}$, $V_{1,3} := \{ i \leq d : (t_i, s_i) = (1,3) \}$ and $V_{2,3} := \{ i \leq d : (t_i, s_i) = (2,3) \}$, then $V_{1,2}, V_{1,3}, V_{2,3}$ is a partition of $\{1, \ldots, d \}$. Let $f(b_\beta,a_\alpha) := (f_i (b_\beta, a_\alpha) : i \in V_{1,2})$, $f(\b_\beta,c) = (f_i (c, b_\beta) : i \in V_{1,3})$ and $f(a_\alpha,c) := (f_i (a_\alpha, b_\beta) : i \in V_{2,3})$. Then we have $$\models \psi(b_\beta;a_\alpha, c)  \iff \models \varphi'(f(b_\beta,a_\alpha), f( b_\beta, c), f(a_\alpha, c)),$$ where 
 $\varphi' \in \mathcal{L}_0$
  is obtained from $\varphi$ by regrouping the variables accordingly.

Let $A := \{ f(a_\alpha,c): \alpha < \kappa \}$. Consider the rectangular array $( f(b_\beta, a_\alpha) : \alpha<\kappa, \beta < \lambda)$. It is an indiscernible array by mutual indiscernibility of the sequences $(a_\alpha)$ and $(b_\beta)$. In particular, the sequence of rows 
 $((f(b_\beta, a_\alpha) : \alpha < \kappa) : \beta < \lambda)$ is $\emptyset$-indiscernible.
 As $T_0$ is dependent, $|A| \leq \kappa$ and $\lambda > (|T| + \kappa)$ is regular, there is some $\gamma < \lambda$ such that the sequence of columns $((f(b_\beta, a_\alpha) : \alpha < \kappa) :  \gamma < \beta < \lambda)$ is $\mathcal{L}_0$-indiscernible over $A$. 

%Fix $\gamma < \beta < \lambda$. As $T_K$ is stable, we have that $\tp_{\mathcal{L}_K}(f(c,b_j)/(f(a_i,b_j) b'_j : i < \kappa)A)$ is definable. In particular, there is some definition for $\varphi$, that is a formula $\theta_j(y, u_j, v_j) \in \mathcal{L}_K$ and some finite tuples $\bar{\alpha}_j$ given by $\alpha^j_1 < \ldots < \alpha^j_{k_j} < \kappa$ and $\bar{a}_j \in A^{k_j}$ with $k_j \in \mathbb{N}$ such that:  for every $i<\kappa$ we have
%$$\models \varphi(f(c,b_j), f(a_i,b_j), f(c,a_i), a'_i, b'_j) \iff \models \theta_j(f(a_i, b_j), f(c,a_i), a'_i, b'_j; (f(a_{\alpha_t}, b_{j}) a'_{\alpha_t} b'_j)_{1 \leq t \leq k_j}; \bar{a}_j)$$
%
%(noting that $f(c,a_i), a'_i \in A$ for all $i<\kappa$.) Observe that as $\beta <j<\lambda$ varies, there are at most $\kappa^{|T_K|}$-many choices for the tuple $(\theta_j, k_j, \bar{\alpha}_j, \bar{a}_j)$.
%

Fix $\gamma < \beta < \lambda$. For any tuple $e \in M_{(x_i : i \in V_{1,3})}$, let 
$$S_{e}^\beta := \{ \alpha < \kappa : \models \varphi'( f( b_\beta, a_\alpha), e,f(a_\alpha,c)) \} \subseteq \kappa,$$
and let $\mathcal{F}_\beta := \{ S^\beta_e :  e \in M_{(x_i : i \in V_{1,3})}\}$ be the collection of all such subsets of $\kappa$ that can be realized by some tuple. 

We then have that $\mathcal{F}_\beta = \mathcal{F}_{\beta'}$ for any $\gamma < \beta, \beta' < \lambda$. Indeed, by the $\mathcal{L}_0$-indiscernibility observed above, there is some $\sigma \in \Aut(\mathbb{M}_0/A)$  sending $( f( b_\beta, a_\alpha) : \alpha < \kappa)$ to $ (f( b_\beta', a_\alpha) : \alpha < \kappa)$. But then for any $e$ we have that $S^\beta_e = S^{\beta'}_{\sigma(e)}$ (recalling that $f(a_\alpha,c) \in A$ for all $\alpha<\kappa$), hence $\mathcal{F}_\beta \subseteq \mathcal{F}_{\beta'}$, and vice versa exchanging the roles of $\beta$ and $\beta'$. So let $\mathcal{F} := \mathcal{F}_\beta$ for some (equivalently, any) $\beta > \gamma$. Note that $S^\beta_e$ is determined by the $\mathcal{L}_0$-type $\tp_{\varphi'}(e/ ( f( b_\beta, a_\alpha): \alpha < \kappa) A)$. As $|f( b_\beta, a_\alpha): \alpha < \kappa) A| \leq \kappa$ and $\varphi'$ is dependent, we get that $|\mathcal{F}| \leq \ded(\kappa)$ by Fact \ref{fac: NIP by counting types}.

Now we estimate $|S_{\psi, J_{>\gamma}}(I\times\{c\})|$ (see Definition \ref{def: localized space of types}).
Given $\gamma < \beta < \lambda$, we have that $\tp_{\psi}(b_\beta / I\times\{c\})$ is determined by the set 
$$S^\beta_{f(b_\beta,c)} = \{ \alpha < \kappa :  \models \varphi'(f( b_\beta, a_\alpha),f( b_\beta, c),f(a_\alpha,c)) \} \in \mathcal{F}.$$

But as $|\mathcal{F}| \leq \ded(\kappa)$, there are only $\ded(\kappa)$ choices for this set, hence $|S_{\psi, J_{>\gamma}}(I\times\{c\})| \leq \ded(\kappa)$.

 This would imply a contradiction in a model of ZFC with $\ded (\kappa) < 2^{\kappa}$ (which exists by Fact \ref{fac: Mitchell}). But the property of a given formula $\psi$ being $2$-dependent is arithmetic, hence set-theoretically absolute, so we obtain the result in ZFC.
\end{proof}
\begin{expl}
Let $f: \mathbb{C}^2 \to \mathbb{C}$ be an arbitrary function, and let $p(x,y,z)$ be a polynomial over $\mathbb{C}$. Consider the relation $E\subseteq \mathbb{C}^3$ given by $$E(x,y,z) \iff p(f(x,y), f(x,z), f(y,z)) = 0.$$
 Then there is some finite $3$-partite $3$-hypergraph $H$ such that $E$ does not contain it as an induced tripartite hypergraph.
\end{expl}
\begin{remark} We will see in the proof of Theorem \ref{thm: Granger}(3) that we cannot relax the assumption $\mathcal{M}_0$ is dependent to just $2$-dependent. Generalizations of Theorem \ref{thm: functions into stable are 2-dep } for $n$-dependence and functions of arbitrary arity will be investigated in \cite{NdepGroups3}. \end{remark}

\section{2-dependence of bilinear forms over dependent fields}\label{sec: Granger}

In this section we consider certain theories of bilinear forms on  vector spaces, in a language with a separate sort for the field. Their basic model theory was studied by  Granger in \cite{granger1999stability}, and more recently in \cite{chernikov2016model, dobrowolski2020sets} from the point of view of generalized stability theory. Here we investigate $n$-dependence in these structures. As it was already mentioned in the introduction, all currently known algebraic examples of strictly $n$-dependent theories for $n \geq 2$ are closely related to bilinear forms over \emph{finite fields}. E.g., smoothly approximable structures studied in \cite{cherlin2003finite} are $2$-dependent and coordinatizable via bilinear forms over finite fields (see \cite[Example 2.2(4)]{chernikov2014n} for a discussion of their $2$-dependence); and the strictly $n$-dependent pure groups constructed in \cite{chernikov2019mekler} using Mekler's construction can be interpreted in alternating bilinear maps over finite fields as demonstrated in \cite{baudisch2002mekler}. Here we show that a more general situation is possible: every non-degenerate (symmetric or alternating) bilinear form on an infinite dimensional vector space over an \emph{arbitrary dependent field} is strictly $2$-dependent (and the assumption that the field is dependent is necessary, see Theorem \ref{thm: Granger}). Our proof of $2$-dependence relies crucially on the Composition Lemma (Theorem \ref{thm: functions into stable are 2-dep }) from the previous section. 
We view these examples as clarifying the scope of Conjecture \ref{conj: main conj} (namely, how much algebraic structure is required for the collapse of the $n$-dependence hierarchy) and guiding towards a correct formulation of its abstract counterpart (``every $n$-dependent theory is linear over its $1$-dependent part'').

We begin by recalling some definitions and results from \cite{granger1999stability}.
In this section we consider structures in the language $\mathcal{L}$ consisting of two sorts $V$ and $K$, the field language on $K$, the vector space language on $V$, scalar multiplication function $K \times V \to V$ and the bilinear form function $[x,y]: V \times V \to K$.
The language $\mathcal{L}_{\theta}$ is obtained from $\mathcal{L}$ by adding for each $n \in \omega$ a (definable) $n$-ary predicate $\theta_n(x_1, \ldots, x_n)$ which holds if and only if $x_1, \ldots, x_n \in V$ are linearly independent over $K$. Let $\mathcal{L}_\theta^K$ be a language expanding $\mathcal{L}_\theta$ by relations on $K^n, n \in \omega$ definable in the language of rings such that $K$ eliminates quantifiers in $\mathcal{L}_\theta^K $ (e.g.~we can always take Morleyzation of $K$).

\begin{defn}
	For $K$ a field,  $F \in \{A, S \}$ and an arbitrary $m \in \mathbb{N} \cup \{ \infty \}$, let $\prescript{}{F}{T}^K_m$ denote the $\mathcal{L}_{\theta}^K$-theory expressing that the sort corresponding to $K$ is a field which is moreover a model of $\Th(K)$; $V$ a $K$-vector space of dimension $m$; $[x,y]: V \times V \to K$ is a non-degenerate bilinear form of type $F$, where a form of type $S$ is a symmetric form, and a form of type $A$ is an alternating form; and the predicates $\theta_n$ define linear independent tuples of length $n$ over $K$.
\end{defn}

\begin{fact}\cite[Theorem 9.2.3]{granger1999stability} \label{fac: GrangerQE}
Let $F \in \{A, S \}$, and let $K$ be an arbitrary field  if $F = A$, or a field closed under square roots if $F=S$. Let $m \in \mathbb{N} \cup \{ \infty \}$ be arbitrary if $F = A$, or even if $m \in \mathbb{N}$ and $F=S$.  Then the theory $\prescript{}{F}{T}_{m}^{K}$ is consistent, complete and has elimination of quantifiers in the language $\mathcal{L}_{\theta}^K$.
\end{fact}

%\begin{fact}\cite[Corollary 5.4.3 (1)]{granger1999stability}\label{fac: basis exists}
%%\begin{enumerate}
%%\item 
%Any non-degenerate sesquilinear space of dimension $\leq \aleph_0$ which is not alternate has an orthogonal basis.
%%\end{enumerate}
%
%\end{fact}

We are ready to state the main result of this section.

\begin{theorem}\label{thm: Granger}
Let $T := \prescript{}{F}T^K_m$ be as in Fact \ref{fac: GrangerQE}.
\begin{enumerate}
\item If $m < \infty$ then $\Th(K)$ is $n$-dependent if and only if $T$ is $n$-dependent, for any $n \in \mathbb{N}_{\geq 1}$. 
	\item If $m = \infty$ and $\Th(K)$ is dependent then $T$ is  strictly $2$-dependent.
	\item If $n \in \mathbb{N}_{\geq 1}$, $\Th(K)$ is not $n$-dependent and $m = \infty$, then $T$ is not $2n$-dependent.
\end{enumerate}
\end{theorem}

\begin{cor}
\begin{enumerate}
\item The case of a finite field $K$ corresponding to extra-special $p$-groups was treated in \cite[Section 3]{hempel2016n}.
\item In \cite{baudisch2016neostability}, for each $n \in \mathbb{N}$ and $p$, Baudisch constructs a structure $D(n)$ in the language of groups with $n$ additional constant symbols, with $D(1)$ corresponding to extra-special $p$-groups. Since all these examples are interpretable in the bilinear form with additional constant symbols, they are all $2$-dependent.
\end{enumerate}
\end{cor}
The rest of the section constitutes a proof of this theorem.
\subsection{Proof of Theorem \ref{thm: Granger}(1)}
If  $m < \infty$, then $K$ is $n$-dependent if and only if $T$ is $n$-dependent, for any $n \geq 1$. This follows from the fact that any model of $\prescript{}{F}T^K_m$  can be interpreted in $K$ using the isomorphism $K^m\cong V$ for some $m \in \mathbb N$ as follows. Interpreting the vector space structure is obvious. Now, let $\mathcal B=(e_1, \dots, e_m)$ be the standard basis  of $K^m$. Then the bilinear form is completely determined by fixing $k_{i,j}=[e_i,e_j]$ for all $1 \leq i,j \leq m$. Let $\pi_i: K^m \rightarrow K$ be the projection map onto the $i$-th coordinate. Then for $v,w \in K^m$, we have that $[v,w]= \sum_{i,j=1}^m \pi_i(v)\pi_j(w)k_{i,j} $, which is definable over $\{ k_{i,j}:  1 \leq i, j \leq  m \}$.

\subsection{Proof of Theorem \ref{thm: Granger}(3)}\label{sec: pairing gives IP2}
Assume that $K$ is not $n$-dependent, then by Theorem \ref{thm: better reduction to singletons} it must be witnessed by some $\mathcal{L}_K$-formula $\varphi(\bar{x};y_1, \ldots, y_n)$ with each $y_i$ a single variable. Then by compactness for $1 \leq k \leq n$ we can find sequences $(c^k_{(i_k,j_k)} : (i_k,j_k) \in \omega \times \omega)$ with $\omega \times \omega$ ordered lexicographically and all $c^k_{i_k, j_k}$ pairwise distinct elements in $K$, such that for every $A\subseteq (\omega \times \omega)^n$ there is some $\bar{e}_A$ satisfying 
$$\models \varphi(\bar{e}_A; c^1_{(i_1, j_1)},\ldots, c^n_{(i_n, j_n)}) \iff ((i_1,j_1), \ldots, (i_n,j_n)) \in A.$$

As $m=\infty$, we can choose  $(a^k_i : 1 \leq k \leq n, i \in \omega)$ a tuple consisting of linearly independent elements in $V$. For each $1 \leq k \leq n$ and $j \in \omega$, let $f^k_{j} : V \to K$ be a linear function satisfying $f^k_j(a^k_i) = c^k_{(i,j)}$ for all $i \in \omega$. Since the bilinear form is non-degenerate, there exists some $b^k_j \in V$ such that $f^k_j(x) = [x, b^k_j]$ for all $x \in V$. But then, identifying $(\omega \times \omega)^n$ with $\omega^{2n}$, for any set $A \subseteq \omega^{2n}$, we have 
$$\models \varphi(\bar{e}_A, [a^1_{i_1}, b^1_{j_1}], \ldots, [a^n_{i_n}, b^n_{j_n}]) \iff (i_1,j_1,\ldots, i_n,j_n) \in A,$$ hence the formula $\psi(\bar{x};y_1, y_2, \ldots, y_{2n-1}, y_{2n}) = \varphi(\bar{x}, [y_1,y_2], \ldots, [y_{2n-1}, y_{2n}])$ is not $2n$-dependent, witnessed by the sequences $(a^1_{i_1}, b^1_{j_1}, \ldots, a^n_{i_n}, b^n_{j_n} : i_1,j_1, \ldots, i_n, j_n \in \omega)$.

\subsection{Proof of Theorem \ref{thm: Granger}(2)}
Let $\mathbb{M} \models T$ be a monster model. If $T$ is not $2$-dependent, by Proposition \ref{prop: indisc witness to IPn}  there exists tuples $\bar{a}_\alpha, \bar{b}_\beta$ and an $\mathcal{L}_{\theta}^K$-formula $\varphi(\bar{x}; \bar{y}, \bar{z})$ without parameters such that $(\bar{a}_\alpha, \bar{b}_\beta : \alpha,\beta \in \mathbb{Q})$ is $O_{2,p}$-indiscernible over $\emptyset$ and it is \emph{shattered} by $\varphi$. More precisely,  for every $A \subseteq \mathbb{Q} \times \mathbb{Q}$ there is some $\bar{c}_A$ such that 
$$ \models \varphi(\bar{c}_A; \bar{a}_\alpha, \bar{b}_\beta) \iff (\alpha,\beta) \in A.$$

We write $\bar{x} = \bar{x}^{K \frown}\bar{x}^V, \bar{y} = \bar{y}^{K \frown} \bar{y}^V, \bar{z} = \bar{z}^{K \frown} \bar{z}^V$ for the subtuples of the variables of the corresponding sorts, where $\bar{x}^K = (x^K_i : i \in X^K)$ and $\bar{x}^V = (x^V_i : i \in X^V)$ and $X^K \sqcup X^V$ is a partition of $\{1, \dots, |\bar{x}|\}$. We proceed similarly for $\bar{y}$ and $\bar{z}$. Let $ a_\alpha = \bar{a}_\alpha^{K \frown} \bar{a}_\alpha^V, \bar{b}_{\beta} = \bar{b}_{\beta}^{K \frown} \bar{b}_\beta^V$ for the corresponding subtuples in the $K$-sort and the $V$-sort, respectively. Let $\bar{a}_\alpha^V = (a^V_{\alpha,i} : i \in Y^V), \bar{b}_\beta^V = (b^V_{\beta,i} : i \in Z^V)$, etc.

As a first step towards obtaining a contradiction, we will show that all of the elements of the sort $V$ in the configuration witnessing the failure of $2$-dependence chosen above may be assumed linearly independent over $K$. This is achieved by modifying the initial configuration and the formula as demonstrated in the following four claims.

\begin{claim}\label{cla: lin indep seqs}
There exist a finite tuple $\bar{e}$ in $V$, a formula $\varphi(\bar{x}, \bar{w}, \bar{y}, \bar{z}) \in \mathcal{L}_{\theta}^K$ and sequences of tuples $(\bar{a}_\alpha, \bar{b}_\beta : \alpha, \beta \in \mathbb{Q})$ such that:
	\begin{enumerate}
		\item $(\bar{a}_\alpha, \bar{b}_\beta : \alpha, \beta \in \mathbb{Q})$ is $O_{2,p}$-indiscernible over $\bar{e}$;
		\item for any $\alpha^* \in \mathbb{Q}$ and $i^* \in Y^V$, we have that 
		$$a^V_{\alpha^*,i^*} \notin \Span\left((a^V_{\alpha^*,i} : i \in Y^V \setminus \{ i^*\} ), (\bar{a}^V_{\alpha} : \alpha \in \mathbb{Q} \setminus \{ \alpha^* \}),(\bar{b}^V_\beta : \beta \in \mathbb{Q}), \bar{e}\right);$$
		\item for any $\beta^* \in \mathbb{Q}$ and $j^* \in Z^V$, we have that 
		$$b^V_{\beta^*,j^*} \notin \Span\left((b^V_{\beta^*,j} : j \in Z^V \setminus \{ j^*\} ), (\bar{b}^V_{\beta} : \beta \in \mathbb{Q} \setminus \{\beta^*\}),(\bar{a}^V_\alpha : \alpha \in \mathbb{Q}), \bar{e}\right);$$

		\item $ \varphi(\bar{x}, \bar{e}; \bar{y}, \bar{z})$ shatters $(\bar{a}_\alpha, \bar{b}_\beta : \alpha, \beta \in \mathbb{Q})$, i.e. for every $A \subseteq \mathbb{Q}^2$, there is some $\bar{c}_A$ such that $\models \varphi(c_A, \bar{e}, \bar{a}_\alpha, \bar{b}_\beta) \iff (\alpha, \beta) \in A$. 
	\end{enumerate}
\end{claim}

\begin{proof}

Assume that $(\bar{a}_\alpha, \bar{b}_\beta)$ chosen in the beginning of the proof do not satisfy (2) with $\bar{e} = \emptyset$.
 Then there are some $\alpha^* \in \mathbb{Q}, i^* \in Y^V$ and finite sets $I, J \subseteq \mathbb{Q}, \alpha^* \notin I$ such that 
 
 $$a^V_{\alpha^*, i^*} \in \Span \left((a^V_{\alpha^*, i} : i \in Y^V \setminus \{i^* \} ),(\bar{a}^V_\alpha)_{\alpha \in I} ,(\bar{b}^V_\beta)_{\beta \in J} ,\bar{e} \right).$$
 
 Then there is an $\emptyset$-definable function $f$ and some finite tuple $\bar{k}$ in $K$ such that 
 $$ a^V_{\alpha^*, i^*} = f \left(\bar{k}, (a^V_{\alpha^*, i} : i \in Y^V \setminus \{i^* \} ), (\bar{a}^V_\alpha)_{\alpha \in I}, (\bar{b}^V_\beta)_{\beta \in J}, \bar{e} \right). $$
 
We let
 \begin{gather*}
 	\gamma^+ := \min \{ \alpha \in I : \alpha > \alpha^* \} \textrm{, } \gamma^- := \max \{ \alpha \in I : \alpha  < \alpha^*  \} \textrm{, }\delta := \max(J),\\
 	\bar{e}' := \left(\bar{a}^V_\alpha \right)_{\alpha \in I} \,^\frown \left(\bar{b}^V_\beta \right)_{\beta \in J} \,^\frown \bar{e},\\
 	\left(\bar{a}' \right)^V_\alpha := \left(a^V_{\alpha^*, i} : i \in Y^V \setminus \left\{i^* \right\} \right).
 \end{gather*}
Then 
$$\left(\bar{a}_\alpha^{K \frown}\left(\bar{a}' \right)^V_\alpha, \bar{b}_\beta : \alpha \in \left(\gamma^-, \gamma^+ \right), \beta \in \left(\delta, + \infty \right) \right)$$
 is $O_{2,p}$-indiscernible over $\bar{e}'$. As $\alpha^* \in \left( \gamma^-, \gamma^+ \right)$, it follows that for every $\alpha \in (\gamma^-, \gamma^+)$ there is some tuple $\bar{k}_\alpha$ in $K$ such that $a^V_{\alpha, i^*} = f \left( \bar{k}_\alpha, (\bar{a}')^V_\alpha, \bar{e}' \right)$. Let 
%$$ \sum_{\alpha \in I} \sum_{i \in Y^V} k_{\alpha,i}a^V_{\alpha,i} + \sum_{\beta \in J} \sum_{j\in Z^V} h_{\alpha,j}b^V_{\alpha,j} = 0$$
%for some $k_{\alpha,i}, h_{\alpha,j} \in K$ not all $0$. Let $\alpha^* := \max (I \cup J)$, and say $\alpha^* \in I$ (the case $\alpha^* \in J$ is similar). Let $I' := I \setminus \{\alpha^* \}$, and let $i^* \in Y^V$ be such that $k_{\alpha^*, i^*} \neq 0$. 
%Let 
%$$(\bar{a}')^V_\alpha := (a^V_{\alpha,i}: i \in Y^V \setminus\{i^*\})^{\frown}\bar{a}^{V \frown}_{\alpha \in I'}\bar{b}^V_{\beta \in J},$$
%$$\bar{k} := (k_{\alpha,i} : \alpha \in I, i \in Y^V)^{\frown}(h_{\beta,j} : \beta \in J, j \in Z^V).$$
%
%Then we have that $a_{\alpha^*, i^*}^V = f(\bar{k}, (\bar{a}')^V_{\alpha^*})$ for an $\emptyset$-definable function $f$. It follows by the choice of $\alpha^*$ and $O_{2,p}$-indiscernibility of $(\bar{a}_\alpha, \bar{b}_\beta : \alpha, \beta \in \mathbb{Q})$ that for every $\alpha \in (\alpha^*, \infty)$ there is some $\bar{k}_\alpha$, a tuple in $K$, so that $a_{\alpha, i^*}^V = f(\bar{k}_{\alpha}, (\bar{a}')^V_{\alpha})$.
$$ \varphi'(\bar{x}, \bar{w}, \bar{y}', \bar{z}) := \varphi\left(\bar{x}; \bar{y}^K,(y^V_{i})_{i \in Y^V, i < i^*}, f\left(\bar{y}^K_1, (y^V_{i})_{i \in Y^V \setminus \{ i^* \}}, \bar{w} \right), (y^V_{i})_{i \in Y^V, i > i^*}, \bar{z}\right),$$
where $\bar{y}' = (\bar{y}')^{K \frown}(\bar{y}')^V$, $(\bar{y}')^K = \bar{y}^{K  \frown}_1\bar{y}^K$ and $(\bar{y}')^V = (y^{V}_{i} : i \in Y^{V} \setminus \{ i^* \}) $. Let $\bar{a}'_\alpha := \bar{a}^{ K \frown}_\alpha(\bar{a}')^V_\alpha$ (so the tuple $k_\alpha^{\frown} \bar{a}'_\alpha$ corresponds to the variables $\bar{y}'$).

Restricting to the set $(\gamma^-, \gamma^+)\times(\delta, \infty)$ we may thus assume:

\begin{enumerate}

	\item[(a)] $(\bar{a}'_\alpha, \bar{b}_\beta : \alpha, \beta \in \mathbb{Q})$ is $O_{2,p}$-indiscernible over $\bar{e}'$ (follows by the choice of $\gamma^-, \gamma^+, \delta$, definition of $\bar{a}'_\alpha$ and $O_{2,p}$-indiscernibility of $(\bar{a}_\alpha, \bar{b}_\beta : \alpha, \beta \in \mathbb{Q})$),
	\item[(b)] $\varphi'(\bar{x}, \bar{e}'; \bar{y}' , \bar{z})$ shatters $(\bar{k}_\alpha^{\frown} \bar{a}'_\alpha, \bar{b}_\beta : \alpha, \beta \in \mathbb{Q})$ (as for any $\bar{c}$ and $\alpha, \beta \in \mathbb{Q}$ by the above we have $\models \varphi(\bar{c}, \bar{a}_\alpha, \bar{b}_\beta) \iff \models \varphi'(\bar{c}, \bar{e}', \bar{k}_\alpha^{\frown}\bar{a}'_\alpha, \bar{b}_\beta )$).
\end{enumerate}

By Fact \ref{fac: random hypergraph indiscernibles exist}, let $( \bar{h}_\alpha^\frown \bar{a}''_\alpha, \bar{b}'_\beta : \alpha,\beta \in \mathbb{Q})$ be an $O_{2,p}$-indiscernible over $\bar{e}'$, based on $( \bar{k}_\alpha ^\frown \bar{a}'_\alpha, \bar{b}_\beta : \alpha, \beta \in \mathbb{Q})$ over $\bar{e}'$. We still have that $ \varphi'(\bar{x}, \bar{e}', \bar{y}', \bar{z})$ shatters $( \bar{h}_\alpha^\frown \bar{a}''_\alpha, \bar{b}'_\beta : \alpha,\beta \in \mathbb{Q})$.
Replacing $\varphi$ by $\varphi'$, $\bar{e}$ by $\bar{e}'$ and the sequences $(\bar{a}_\alpha, \bar{b}_\beta : \alpha, \beta \in \mathbb{Q})$ by $(\bar{h}_\alpha^\frown\bar{a}''_\alpha, \bar{b}'_\beta : \alpha, \beta \in \mathbb{Q})$, we have thus reduced the length of the tuples $\bar{a}^V_\alpha$ (at the price of increasing the length of $\bar{a}^K_\alpha$). Repeating this argument for both $\bar{a}_\alpha$'s and $\bar{b}_\beta$'s at most finitely many times (as the variables $\bar y$ and $\bar z$ have finite length), we obtain the conclusion of the claim.
\end{proof}

\begin{claim}\label{cla: treating c} There exist finite tuples $\bar{e}$ in $V$ and $\bar{c}$, a formula $\varphi(\bar{x}, \bar{w}, \bar{y}, \bar{z})  \in \mathcal{L}_{\theta}^K$ and sequences of tuples $(\bar{a}_\alpha, \bar{b}_\beta : \alpha, \beta \in \mathbb{Q})$ satisfying (1)--(3) and
\begin{enumerate}
	\item[(4$'$)] $\models \varphi(\bar{c}, \bar{e}, \bar{a}_\alpha, \bar{b}_\beta) \iff G_{2,p} \models R_2(\alpha, \beta)$ for all $\alpha,\beta \in \mathbb{Q}$.
  \setcounter{enumi}{4}
	\item $c^V_i \notin \Span \left(  (c^V_j : j \neq i), \left( \bar{a}^V_{\alpha}\right)_{\alpha \in \mathbb{Q}}, \left( \bar{b}^V_{\beta} \right)_{\beta \in \mathbb{Q}}, \bar{e} \right)$ for any $i \in X^V$.
\end{enumerate}

\end{claim}
\begin{proof}

We start with $\bar{e}$, $\varphi$ and $(\bar{a}_\alpha, \bar{b}_\beta )_{\alpha, \beta \in \mathbb{Q}}$ satisfying (1)--(4) of Claim \ref{cla: lin indep seqs}.
By (4), there exists a tuple $\bar{c}$  such that 
$$ \models \varphi(\bar{c}, \bar{e}, \bar{a}_\alpha, \bar{b}_\beta) \iff G_{2,p} \models R_2(\alpha, \beta) \textrm{ for all } \alpha,\beta \in \mathbb{Q}.$$

 We write $\bar{c} = \bar{c}^{K \frown} \bar{c}^V$.
	Assume that there is some $i^* \in X^V$ such that $c^V_{i^*}$ is in the span of the tuple 
$$(c^V_i : i \in X^V \setminus \{ i^* \}) \,^{\frown} (\bar{a}^V_\alpha)_{\alpha \in I} \,^{\frown} (\bar{b}^V_\beta)_{\beta \in J} \,^{\frown}  \bar{e}$$
 for some finite sets $I,J \subseteq \mathbb{Q}$. Then 
$$c^V_{i^*} = f \left(\bar{c}^K_1, (c^V_i)_{i \in X^V \setminus \{ i^* \}},(\bar{a}^V_\alpha)_{\alpha \in I},(\bar{b}^V_\beta)_{\beta \in J}, \bar{e} \right) $$ for some $\emptyset$-definable function $f$ and some tuple $\bar{c}^K_1$ in $K$.
Let $\alpha^* := \max( I \cup J )$. We let $\bar{c}' := (\bar{c}')^{K \frown}(\bar{c}')^V$, where $(\bar{c}')^K := \bar{c}^{K \frown}_1 \bar{c}^{K}$ and $(\bar{c}')^V := (c^V_i : i \in X^V \setminus \{ i^* \})$. Let $\bar{e}' := (\bar{a}^V_\alpha)_{\alpha \in I} \,^\frown(\bar{b}^V_\beta)_{\beta \in J} \,^\frown \bar{e}$.
Restricting to a copy of $G_{2,p}$ contained in $(\alpha^*, \infty) \times (\alpha^*, \infty)$ (Remark \ref{rem : induced copy}), we thus have:
	$$G_{2,p}\models R_2(\alpha, \beta) \iff  \models \varphi \left(\bar{c}, \bar{e}, \bar{a}_\alpha, \bar{b}_\beta \right) \iff $$
	$$\models \varphi\left( \bar{c}^K, (\bar{c}^V_i)_{i<i^*}, f(\bar{c}^K_1, (\bar{c}')^V, \bar{e}'), (\bar{c}^V_i)_{i>i^*}, \bar{e}, \bar{a}_\alpha, \bar{b}_\beta \right)$$ 
	$$\iff \models \varphi'(\bar{c}', \bar{e}', \bar{a}_\alpha, \bar{b}_\beta)$$
	for an appropriate $\mathcal{L}_{\theta}^K$-formula $\varphi'$. Replacing $\varphi$ by $\varphi'$, $\bar{e}$ by $\bar{e}'$ and $\bar{c}$ by $\bar{c}'$, this shows that  (4$'$) is still satisfied. And   (1), (2) and (3) still hold as well (follows as $\bar{e}$ satisfies (1), (2), (3) and all the new elements in $\bar{e}'$ are from $(\bar{a}'_\alpha, \bar{b}'_\beta : \alpha, \beta < \alpha^*)$). We have thus reduced the length of the tuple $\bar{c}$ (at the price of increasing the length of $\bar{e}$). Repeating this argument finitely many times if necessary, we obtain the claim.
\end{proof}

\begin{claim}\label{cla: reducing e}
There exist finite tuples $\bar{e}$ in $V$ and $\bar{c}$, a formula $\varphi(\bar{x}, \bar{w}, \bar{y}, \bar{z})  \in \mathcal{L}_{\theta}^K$ and sequences of tuples $(\bar{a}_\alpha, \bar{b}_\beta : \alpha, \beta \in \mathbb{Q})$ satisfying (1)--(3), (4$'$), (5) and 
\begin{enumerate}
  \setcounter{enumi}{5}
	\item $e_i \notin \Span(e_j : j \neq i)$ for any $i \in |\bar{e}|$.
\end{enumerate}

\end{claim}
\begin{proof}
Start with $\bar{e}$, $\varphi$ and $(\bar{a}_\alpha, \bar{b}_\beta )_{\alpha, \beta \in \mathbb{Q}}$ given by Claim \ref{cla: treating c}.
As in the previous two claims, if $e_{i^*} \in \Span(\bar{e}_i : i \neq i^*)$, then $e_{i^*} = f(\bar{k}, (\bar{e}_i)_{ i \neq i^*})$ for some $\emptyset$-definable function $f$ and a tuple $\bar{k}$ in $K$. Replacing $\bar{e}$ by $(\bar{e}_i)_{ i \neq i^*}$, adding $\bar{k}$ to $\bar{c}^K$ and modifying the formula accordingly, the condition (4$'$) is still satisfied. And (1),(2),(3),(5),(6) still hold as we only pass to a subtuple of $\bar{e}$. Repeating this finitely many times we obtain the claim.
\end{proof}

Let $\bar{e}$, $\bar{c}$, $\varphi(\bar{x}, \bar{w}, \bar{y}, \bar{z})  \in \mathcal{L}_{\theta}^K$ and $(\bar{a}_\alpha, \bar{b}_\beta : \alpha, \beta \in \mathbb{Q})$ be as given by Claim \ref{cla: reducing e}. By (2), (3), (5), (6) and linear algebra we get that all elements in the tuple 
$$\bar{c}^V \,^{\frown} \left(\bar{a}^V_{\alpha} \right)_{\alpha \in \mathbb{Q}}\,^{\frown} \left(\bar{b}^V_{\beta} \right)_{\beta \in \mathbb{Q}}\,^{\frown} \bar{e}$$
 are linearly independent. Hence adjoining $\bar{e}^V$ to $\bar{c}^V$ and regrouping the variables of $\varphi$ accordingly, we get:

\begin{claim}\label{cla: lin indep counterex}
	There is an $\mathcal{L}_{\theta}^K$-formula $\varphi(\bar{x},\bar{y}, \bar{z})$ and tuples $\bar{c}, \bar{a}_\alpha, \bar{b}_\beta$ such that:
	\begin{enumerate}
		\item $(\bar{a}_\alpha, \bar{b}_\beta: \alpha, \beta \in \mathbb{Q})$ is $O_{2,p}$-indiscernible;
		\item all elements in the tuple $\bar{c}^V \,^{\frown}\left( \bar{a}^V_{\alpha} \right)_{\alpha \in \mathbb{Q}} \,^{\frown} \left(\bar{b}^V_{\beta} \right)_{\beta \in \mathbb{Q}}$ are linearly independent;
		\item $\models \varphi(\bar{c}, \bar{a}_\alpha, \bar{b}_\beta) \iff G_{2,p}\models R_2(\alpha, \beta)$.
	\end{enumerate}

\end{claim}

By quantifier elimination (Fact \ref{fac: GrangerQE}) every $\mathcal{L}_{\theta}^K$-formula is equivalent to a Boolean combination of atomic $\mathcal{L}_{\theta}^K$-formulas. Since $2$-dependence is preserved under Boolean combinations (Fact \ref{fac: props of n-dependent formulas}), it is sufficient to check $2$-dependence for atomic $\mathcal{L}_{\theta}^K$-formulas, which by the definition of the language $\mathcal{L}_{\theta}^K$ fall into one of  the following three cases. 
\noindent

\noindent\textbf{Case 1.} 
The formula $\varphi(\bar{x},\bar{y}, \bar{z})$ is of the form $\psi(t_1(\bar{x},\bar{y}, \bar{z}), \ldots, t_d(\bar{x},\bar{y}, \bar{z}))$ for some $\psi \in \mathcal{L}_K$ and some terms $t_l(\bar x,\bar{y}, \bar{z})$ taking values in $K$ and $1 \leq l \leq d$. 

In this case, for each $l$ we have the following possibilities:
\begin{itemize}
	\item the term $t_l(\bar{x},\bar{y},\bar{z})$ has height $1$, i.e.~it is one of the variables in $\bar{x}^{K \frown}\bar{y}^{K \frown} \bar{z}^{K}$;
	\item $t_l(\bar{x},\bar{y}, \bar{z}) = t^1_l(\bar{x},\bar{y}, \bar{z}) +_K t^2_l(x,\bar{y}, \bar{z})$ or $t_l(\bar x,\bar{y}, \bar{z}) = t^1_l(\bar x,\bar{y}, \bar{z}) \cdot_K t^2_l(\bar x,\bar{y}, \bar{z})$, for some terms $t^1_l, t^2_l$ of smaller height taking values in $K$;
	\item $t_l(\bar x,\bar{y}, \bar{z}) = [t^1_l(\bar x,\bar{y}, \bar{z}), t^2_l(\bar x,\bar{y}, \bar{z})]$ for some terms $t^1_l, t^2_l$ of smaller height taking values in $V$, but then:
	\begin{itemize}
	\item either $t^1_l$ is of height $1$, i.e.~it is one of the variables in $\bar x^{V \frown}\bar{y}^{V \frown }\bar{z}^V$;
	\item or $t^1_l(\bar x,\bar{y}, \bar{z}) = s_l^1(\bar x,\bar{y}, \bar{z}) \cdot_V s_l^2(\bar x,\bar{y}, \bar{z})$ for some terms $s_l^1, s_l^2$ of smaller height taking values in $K$ and $V$, respectively, in which case $t_l(\bar x,\bar{y}, \bar{z}) = s_l^1(\bar x,\bar{y}, \bar{z}) \cdot_K [s_l^2(\bar x,\bar{y}, \bar{z}), t^2_l(\bar x,\bar{y}, \bar{z})]$;
	\item or $t^1_l(\bar x,\bar{y}, \bar{z}) = s_l^1(\bar x,\bar{y}, \bar{z}) +_V s_l^2(\bar x,\bar{y}, \bar{z})$ for some terms $s_l^1, s_l^2$ of smaller height taking values in $V$, in which case $t_l(\bar x,\bar{y}, \bar{z}) = [s_l^1(\bar x,\bar{y}, \bar{z}), t^2_l(\bar x,\bar{y}, \bar{z})] +_K [s_l^2(\bar x,\bar{y}, \bar{z}), t^2_l(\bar x,\bar{y}, \bar{z})]$.

	\end{itemize}
	
	\noindent And similarly for $t^2_l$.
\end{itemize}
Applying this for each $l$ and iterating by recursion on the height of terms, we thus conclude that the formula $\psi(t_1(\bar x,\bar{y}, \bar{z}), \ldots, t_d(\bar x,\bar{y}, \bar{z}))$ is equivalent to
\begin{align*}
&(\ast) &\psi'\biggl(([x^V_i,y^V_j])_{ i \in X^V,j \in Y^V}, ([x^V_i,z^V_j])_{i \in X^V, j \in Y^V},([y^V_i,z^V_j])_{i \in Y^V, j \in Z^V},\\
& &([x^V_i,x^V_j])_{  i, j \in X^V},([y^V_i,y^V_j])_{  i,j \in Y^V}, ([z^V_i,z^V_j])_{ i,j \in Z^V},
  \bar{x}^K, \bar{y}^K, \bar{z}^K \biggr) 
\end{align*}
for some $\mathcal{L}_K$-formula $\psi'$.
As $\Th(K)$ is dependent, Theorem \ref{thm: functions into stable are 2-dep } implies that this formula is $2$-dependent. This concludes Case 1.

\noindent\textbf{Case 2.} The formula $\varphi(\bar{x};\bar{y},\bar{z})$ is given by 
$$\theta_d(t_1( \bar{x},\bar{y}, \bar{z}), \ldots, t_d(\bar{x},\bar{y}, \bar{z}))$$
for some $d \in \mathbb{N}$ and terms $t_l(\bar{x},\bar{y}, \bar{z})$ taking values in $V$.

In this case, by a simple recursion on the height of the terms, we see that for $1 \leq l \leq d$ the term $t_l$ must be of the form
$$\sum_{i \in X^V} t^X_{l,i}(\bar{x},\bar{y},\bar{z}) x_i^V + \sum_{j \in Y^V} t^{Y}_{l,j}(\bar{x},\bar{y}, \bar{z}) y^V_j + \sum_{k \in Z^V} t^{Z}_{l,k}(\bar{x},\bar{y},\bar{z}) z^V_k$$
for some terms $t^X_{l,i}, t^Y_{l,j}, t^Z_{l,k}$ taking values in $K$.

Using linear independence of the set of elements of the tuple $\bar{c}^{V \frown} \bar{a}^{V \frown}_\alpha \bar{b}^V_\beta$  for any $\alpha, \beta\in \mathbb Q$ established in Claim \ref{cla: lin indep counterex}, for any $\alpha, \beta \in \mathbb{Q}$ we have 
$$ G_{2,p} \models R_2(\alpha, \beta) \iff \models  \neg \theta_d(t_1(\bar{c},\bar{a}_\alpha, \bar{b}_\beta), \ldots, t_d(\bar{c},\bar{a}_\alpha, \bar{b}_\beta)) \iff$$
$$ \models (\exists e_1, \ldots, e_d \in K) (e_1, \ldots, e_d) \neq (0, \ldots, 0) \land \bigwedge_{i \in X^V} \left( \sum_{l=1}^d e_l \cdot t_{l,i}^X (\bar{c}, \bar{a}_\alpha, \bar{b}_\beta) = 0 \right) \land $$
$$  \bigwedge_{j \in Y^V}  \left( \sum_{l=1}^d e_l \cdot t^Y_{l,j}(\bar{c}, \bar{a}_\alpha, \bar{b}_\beta) = 0 \right) \land \bigwedge_{k\in Z^V} \left( \sum_{l=1}^d e_l \cdot t^Z_{l,k}(\bar{c}, \bar{a}_\alpha, \bar{b}_\beta) = 0 \right) \iff $$
$$ \models \psi \left( \left(t_{l,i}^X (\bar{c}, \bar{a}_\alpha, \bar{b}_\beta) \right)_{1 \leq l \leq d, i \in X^V}, \left(t^Y_{l,j}(\bar{c}, \bar{a}_\alpha, \bar{b}_\beta) \right)_{1\leq l \leq d, j\in Y^V}, \left(t^Z_{l,k}(\bar{c}, \bar{a}_\alpha, \bar{b}_\beta) \right)_{1 \leq l \leq d, k \in Z^V} \right)$$

for an appropriate formula $\psi \in \mathcal{L}_K$. But this is impossible by Case 1.

\noindent\textbf{Case 3.} The formula is of the form $t(\bar{x},\bar{y},\bar{z}) = 0$ for some term.

As in the previous case, then $t$ must be of the form  $$ \sum_{i \in X^V} t_{i}^X(\bar{x},\bar{y},\bar{z}) x_i^V + \sum_{j \in Y^V} t^{Y}_{j}(\bar{x},\bar{y}, \bar{z}) y^V_j + \sum_{k \in Z^V} t^{Z}_{k}(\bar{x},\bar{y},\bar{z}) z^V_k$$
for some terms $t^X_{i}, t^Y_{j}, t^Z_{k}$ taking values in $K$. Using again linear independence of $\bar{c}^{V \frown} \bar{a}^{V \frown}_\alpha \bar{b}^V_\beta$  for any $\alpha, \beta\in \mathbb Q$ established in Claim \ref{cla: lin indep counterex}, we have that 
$$G_{2,p} \models R_2(\alpha, \beta) \iff t(\bar{c}, \bar{a}_\alpha, \bar{b}_\beta) = 0 \iff$$
$$ \bigwedge_{i \in X^V}t^X_i(\bar{c}, \bar{a}_\alpha, \bar{b}_\beta)=0 \land \bigwedge_{j \in Y^V}t_j^Y(\bar{c}, \bar{a}_\alpha, \bar{b}_\beta) = 0 \land \bigwedge_{k \in Z^V} t_k^Z(\bar{c}, \bar{a}_\alpha, \bar{b}_\beta) = 0$$
$$ \iff \psi \left( (t^X_i(\bar{c},\bar{a}_\alpha, \bar{b}_\beta))_{i \in X^V},  (t^Y_j(\bar{c},\bar{a}_\alpha, \bar{b}_\beta))_{j \in Y^V}, (t^Z_k(\bar{c},\bar{a}_\alpha, \bar{b}_\beta))_{k \in Z^V}\right)$$
for an appropriate $\mathcal{L}_K$-formula $\psi$ and any $\alpha, \beta \in \mathbb{Q}$ --- which is impossible by Case 1.

This finishes the proof of Theorem \ref{thm: Granger} (2). \qed 

\section{Expansions by generic predicates and $n$-dependence}\label{sec: adding rand pred}

In this section we will show that an expansion of a geometric theory $T$ by a generic predicate, in the sense of \cite{chatzidakis1998generic}, is dependent if and only if it is $n$-dependent for some $n$, if and only if the algebraic closure in $T$ is disintegrated (Corollary \ref{cor: n-dep iff acl disint}). In fact, we prove that any expansion of an $n$-dependent geometric theory with disintegrated algebraic closure by generic relations of arity at most $n$ is $n$-dependent (Proposition \ref{prop: Tp is ndep}). While these results have no direct implication for Conjecture \ref{conj: main conj}, the authors view it as
providing some further evidence towards it. Namely, we view it as a ``toy example'' of a situation in which failure of $1$-dependence implies failure of $n$-dependence for all $n$, utilizing the geometric complexity (non-disintegration) of the algebraic closure, similar to the behavior expected in fields.

We begin by recalling some basics on geometric theories and expansions by generic predicates from \cite{chatzidakis1998generic}.
A theory $T$ is \emph{geometric} if it eliminates the $\exists^\infty$-quantifier and the algebraic closure operator $\acl$ satisfies exchange in every model of $T$. In this section, we denote by $\dim$ the $\acl$-dimension in a geometric theory, and by $\ind$ the corresponding algebraic independence relation (see e.g.~\cite[Section 2]{chatzidakis1998generic} for details).

\begin{defn}\label{defn: trivial element} Let $T$ be an arbitrary theory.
\begin{enumerate}
\item An element $a \in \mathbb{M}$ is \emph{non-trivial} if there exist elements $b,c \in \mathbb{M}$ and a small set $B \subseteq \mathbb{M}$ such that $a \in \acl(bcB) \setminus \left( \acl(bB) \cup \acl(c B) \right)$.
\item A theory $T$ is \emph{non-trivial} if there exists a non-trivial element in $\mathbb{M}$, otherwise we call $T$ \emph{trivial}. 
\end{enumerate} 
\end{defn}
\begin{remark}\label{rem: triv iff disint}
	It is immediate from the definitions that a geometric theory $T$ is trivial if and only if it has \emph{disintegrated algebraic closure}, i.e.~if $\acl(A) = \bigcup_{a \in A} \acl(a)$ for any set $A \subseteq \mathbb{M}$.
\end{remark}

Let now $T$ be an arbitrary $\mathcal{L}$-theory with quantifiers elimination (for the questions considered here, we may always assume it replacing $T$ by its Morleyzation) and also eliminating $\exists^{\infty}$. Let $S$ be a distinguished $\mathcal{L}(\emptyset)$-definable set in $T$. We denote by $T_{0,S}$ the theory in a language $\mathcal{L}_P := \mathcal{L} \cup \{P(x) \}$ given by $T \cup \{ P(x) \rightarrow S(x) \}$. When working in an $\mathcal{L}_{P}$-structure, we will write $\tp_{\mathcal{L}}$ and $\acl_{\mathcal{L}}$ to denote the type and the algebraic closure of a tuple in the $\mathcal{L}$-reduct (as opposed to its $\mathcal{L}_P$-type $\tp_{\mathcal{L}_P}$ and its algebraic closure $\acl_{\mathcal{L}_{P}}$ obtained using all $\mathcal{L}_P$-formulas).

\begin{fact}\label{fac: rand pred axioms}
\begin{enumerate}
	\item \cite[Theorem 2.4]{chatzidakis1998generic} The theory $T_{0,S}$ has a model companion $T_{P,S}$ with the following axiomatization:  $\mathcal{M} \models T_{P,S}$ if and only if
	\begin{enumerate}
	\item $\mathcal{M} \models T$;
	\item for every $\mathcal{L}$-formula $\theta(\bar{x},\bar{z})$ with $\bar{x} = (x_1, \ldots, x_n)$ and every $I\subseteq \{1, \ldots, n \}$, $\mathcal{M}$ satisfies
\begin{gather*}
	\forall \bar{z}  \Bigg( \bigg(  \Big( \exists \bar{x} \theta(\bar{x}, \bar{z}) \land \left( \bar{x} \cap \acl_{\mathcal{L}}(\bar{z}) = \emptyset \right) \Big) \land \bigwedge_{i=1}^n S(x_i) \land \bigwedge_{1 \leq i < j \leq n} x_i \neq x_j \bigg)\\
	\rightarrow  \exists \bar{x} \bigg( \theta(\bar{x}, \bar{z}) \land \bigwedge_{i \in I} (x_i \in P) \land \bigwedge_{i \notin I} (x_i \in S \setminus P) \bigg) \Bigg).
\end{gather*}
	\end{enumerate}
	\item Let $\mathcal{M} \models T_{P,S}$. Assume that $\bar{a}, \bar{b}$ are  tuples from $M$ and $A \subseteq M$ is a set of parameters. Then the following are equivalent:
	\begin{enumerate}
	\item $\tp_{\mathcal{L}_P}(\bar{a}/A) = \tp_{\mathcal{L}_P}(\bar{b}/A)$;
	\item there is an $A$-isomorphism of $\mathcal{L}_P$-structures from $\acl_{\mathcal{L}}(A,\bar{a})$ to $\acl_{\mathcal{L}}(A,\bar{b})$ which carries $\bar{a}$ to $\bar{b}$.
	\end{enumerate}
	\item If $\mathcal{M} \models T_{P,S}$, $a \in M$ and $A \subseteq M$, then $a \in \acl_{\mathcal{L}}(A) \iff a \in \acl_{\mathcal{L}_P}(A)$.
\end{enumerate}	
\end{fact}

One typically refers to $T_{P,S}$ as an expansion of $T$ by a \emph{generic predicate} on $S$.
If the predicate $S$ is equivalent to $x=x$ in $T$, we simply write $T_{P}$ instead of $T_{P,S}$.

Our first aim is to show that if $T$ is a non-trivial geometric theory, then its expansion by a generic predicate is not $n$-dependent for any $n$. In order to produce a definable relation witnessing failure of $n$-dependence for an arbitrary $n$, we consider certain algebraic configurations of ``higher arity''.

\begin{defn} 
\cite[Definition 2.6]{berenstein2016geometric} Let $T$ be a geometric theory and $B \subseteq \mathbb{M}$. We say that a tuple $\bar{a} = (a_1, \ldots, a_{n}) \in \mathbb{M}^n$ is an \emph{algebraic $n$-gon over $B$} if $\dim(\bar{a}/B) = n-1$, but any subset of $\{a_1, \ldots, a_{n} \}$ of size $n-1$ is independent over $B$.
\end{defn}

Note that any tuple obtained by permuting the elements of an algebraic $n$-gon over $B$ is still an algebraic $n$-gon over $B$ (by exchange of $\acl$). Due to the following fact, starting with a non-trivial element we can find an algebraic $n$-gon for any $n$.

\begin{fact} \cite[Lemma 2.7]{berenstein2016geometric} \label{fac: n-gon}
	Suppose that $T$ is a geometric theory and $a \in \mathbb{M}$ is non-trivial. Then for every $n \geq 3$ there exist some finite set $B \subseteq \mathbb{M}$ and an algebraic $n$-gon $(a_1, \ldots, a_{n})$ over $B$ such that $a_n = a$.
\end{fact}

\begin{prop}\label{non-disint gives IPn}
	Assume that $T$ is a geometric theory and there exists a non-trivial element in $S$. Then $T_{P,S}$ is not $n$-dependent for any $n \geq 1$.
	
	In particular, if the algebraic closure in $T$ is not disintegrated, then $T_P$ is not $n$-dependent for any $n \geq 1$.
\end{prop}
\begin{proof}
Fix $n \geq 1$, and let $\mathbb{M}$ be a monster model of $T_{P,S}$. Note that the $\mathcal{L}$-reduct of $\mathbb{M}$ is a monster model of $T$. By Fact \ref{fac: rand pred axioms}(3) we have $\acl_{\mathcal{L}_P} = \acl_{\mathcal{L}}$ in $\mathbb{M}$, so we will just write $\acl$ in the rest of the proof. By assumption, there exists a non-trivial $a \models S(x)$ in $\mathbb{M}$.

By Fact \ref{fac: n-gon}, let $a_1, \ldots, a_{n+1}, a_{n+2}$ and a finite set $B$ be such that $a_{n+2} = a$ and $(a_1, \ldots, a_{n+2})$ is an algebraic $(n+2)$-gon over $B$. So in particular $a_{n+2} \models S(x)$. Naming $B$ by constants, without loss of generality we may assume that $B = \emptyset$. Then $\{a_1, \ldots, a_{n+1} \}$ is an $\ind$-independent set. Using extension, symmetry and transitivity we can choose inductively sequences $\bar{a}_i = (a_i^j : j \in \omega)$ for $1 \leq i \leq n$ such that:
\begin{enumerate}
	\item $a_i^j \ind \bar{a}_{<i}a_i^{<j} a_{i+1} \ldots a_{n+1}$ for all $1 \leq i\leq n$ and all $j \in \omega$;
	\item $a_i^j \equiv_{ \bar{a}_1 \ldots \bar{a}_{i-1}a_{i+1} \ldots a_{n+1}} a_i$ for all $1 \leq i\leq n$ and all $j \in \omega$.
\end{enumerate}
In particular, by basic properties of $\ind$ and exchange (1) implies that 
\begin{enumerate}
	\item[(3)] $\{a_i^j : 1 \leq i \leq n, j \in \omega \} \cup \{ a_{n+1} \}$ is an $\ind$-independent set,
\end{enumerate}
and (2) implies that 
\begin{enumerate}
	\item[(4)] $a_1^{j_1} \ldots a_{n}^{j_{n}} a_{n+1}\equiv a_1 \ldots a_{n} a_{n+1}$ for all $j_1, \ldots, j_{n} \in \omega$.
\end{enumerate} 

By assumption we have $a_{n+2} \in S(\mathbb{M}) \cap \acl(a_1 \ldots a_{n+1})$. Then we can choose a formula $\varphi(x_1, \ldots, x_{n+2}) \in \mathcal{L}$ and $1 \leq k \in \omega$ such that:
\begin{enumerate}
\item[(5)] $\varphi(a'_1, \ldots, a'_{n+1}, x_{n+2}) \rightarrow S(x_{n+2})$ for any $a'_1, \ldots, a'_{n+1} \in \mathbb{M}$;
	\item[(6)]$\varphi(a_1, \ldots, a_{n+1}, x_{n+2})$ isolates $\tp_{\mathcal{L}}(a_{n+2} / a_1 \ldots a_{n+1})$;
	\item[(7)] $|\varphi(a_1, \ldots, a_{n+1}, \mathbb{M})| = k$ and $|\varphi(a'_1, \ldots, a'_{n+1},  \mathbb{M})| \in \{0, k\}$ for any $a'_1, \ldots, a'_{n+1} \in \mathbb{M}$.
%	\item[(8)] $\varphi(a_1, \ldots, a_n, x_{n+1}, a_{n+2})$ isolates $\tp_{\mathcal{L}}(a_{n+1} / a_1 \ldots a_{n} a_{n+2})$.
\end{enumerate}

\begin{claim}
The following holds:
\begin{enumerate}
%\item[(8)] $a_{n+1} \notin \acl(\bar{a}_1 \ldots \bar{a}_n)$;
%\item[(9)] for any $(j_1, \ldots, j_n) \in \omega^n$ we have
%$$a_{n+1} \notin \varphi(a_1^{j_1}, \ldots, a_n^{j_{n}}, a_{n+1}, \mathbb{M});$$
\item[(8)]for any $(j_1, \ldots, j_n) \in \omega^n$ we have
 $$|\varphi(a_1^{j_1}, \ldots, a_n^{j_{n}}, a_{n+1}, \mathbb{M})| = k \textrm{ and } \varphi(a_1^{j_1}, \ldots, a_{n}^{j_{n}}, a_{n+1}, \mathbb{M}) \cap \acl(\bar{a}_1 \ldots \bar{a}_n) = \emptyset;$$ 
	\item[(9)] for any $(j_1, \ldots, j_n) \neq (j'_1, \ldots, j'_n) \in  \omega^n$ 
	we have 
	$$\varphi(a_1^{j_1}, \ldots, a_n^{j_{n}}, a_{n+1}, \mathbb{M}) \cap \varphi(a_1^{j'_1}, \ldots, a_n^{j'_{n}}, a_{n+1}, \mathbb{M}) = \emptyset.$$
	 \end{enumerate}
	
\end{claim}
\proof
\begin{enumerate}
%\item[(8)] By (3) and forking calculus we have $a_{n+1} \ind \bar{a}_1 \ldots \bar{a}_n$, and $a_{n+1} \notin \acl(\emptyset)$ by assumption, hence $a_{n+1} \notin \acl(\bar{a}_1 \ldots \bar{a}_n)$.
%\item[(9)] Let $(j_1, \ldots, j_n) \in \omega^n$ be arbitrary. We have $a_{n+1} \notin \varphi(a_1^{j_1}, \ldots, a_n^{j_{n}}, a_{n+1}, \mathbb{M})$ since $a_1^{j_1} \ldots a_{n}^{j_{n}} a_{n+1}\equiv a_1 \ldots a_{n} a_{n+1}$ by (4).
	\item[(8)] 
		Let $(j_1, \ldots, j_n) \in \omega^n$ be arbitrary. By the first item in (7), let $\bar{b}$ be the tuple listing all $k$-realizations of $\varphi(a_1, \ldots, a_{n+1}, x_{n+2})$. 
		By (4) we can choose some tuple $\bar{b}'$ in $\mathbb{M}$ such that $a_1^{j_1} \ldots a_{n}^{j_{n}} a_{n+1} \bar{b}' \equiv a_1 \ldots a_{n} a_{n+1} \bar{b}$. Then each of the $k$ pairwise-distinct elements in $\bar{b}'$ is a realization of $\varphi(a_1^{j_1}, \ldots, a_n^{j_{n}}, a_{n+1}, x_{n+2})$, and these are all possible realizations by the second item in (7).
		
	Assume now that $b \in \varphi(a_1^{j_1}, \ldots, a_n^{j_{n}}, a_{n+1},\mathbb{M})$ is arbitrary. By (3) and basic properties of algebraic independence we have 
	$a_{n+1} \ind_{a_1^{j_1} \ldots a_n^{j_{n}}} a_1^{\neq j_1} \ldots a_n^{\neq j_n}$. Hence if $b \in \acl(\bar{a}_1 \ldots \bar{a}_n)$, then already $b \in \acl (a_1^{j_1} \ldots a_n^{j_{n}})$. But $a_1^{j_1} \ldots a_n^{j_n} a_{n+1}b \equiv a_1 \ldots a_n a_{n+1} a_{n+2}$ by (4) and (6), and $a_{n+2} \notin \acl(a_1 \ldots a_n)$ since $(a_1, \ldots, a_{n+2})$ is an algebraic $(n+2)$-gon --- a contradiction.

	\item[(9)] Let $(j_1, \dots, j_n) \neq (j_1', \dots, j_n')$ in $\omega^n$ be given, and let $J := \{1 \leq t \leq n : j_t = j'_t \}$. Thus $|J| < n$. By (3) and basic properties of algebraic independence, we have 
	$$(a_t^{j_t} : t \notin J) \ind_{a_{n+1} (a_t^{j_t} : t \in J)} (a_t^{j'_t} : t \notin J).$$
	Using this and the second item in (7), if 
	$$b\in \varphi(a_1^{j_1}, \ldots, a_n^{j_{n}}, a_{n+1}, \mathbb{M}) \cap \varphi(a_1^{j'_1}, \ldots, a_n^{j'_{n}}, a_{n+1}, \mathbb{M}),$$
	 then 
	$$b \in \acl \left((a_t^{j_t} : t\in J) a_{n+1}) \right).$$
	Additionally by (4) and (6) we have $a_1^{j_1} \ldots a_n^{j_n} a_{n+1}b \equiv a_1 \ldots a_n a_{n+1} a_{n+2}$, hence $a_{n+2} \in \acl((a_t : t \in J) a_{n+1})$. This is a contradiction since $(a_1, \ldots, a_{n+2})$ is an $(n+2)$-gon and $|J| < n$. \qed
 \end{enumerate}

Now, consider the formula
$$\psi(x_1, \ldots, x_{n+1}) := \exists x_{n+2} \in P \varphi(x_1, \ldots, x_{n+1}, x_{n+2}),$$
 we will show that $\psi$ is not $n$-dependent. For this, we show that for an arbitrary $m \in \omega$, $\psi$ shatters $(a^j_i : 1 \leq i \leq n, j \in m)$. Towards this,  let $I \subseteq m^n$ be fixed. Set $\bar{a} := (a_i^{j} : 1 \leq i \leq n, 1 \leq j \leq m)$, with $a_i^j$ as chosen in the beginning of the proof, and consider the $\mathcal{L}$-formula 
\begin{gather*}
	\theta(\bar{x}, \bar{a}) = \theta((x^t_{\bar{j}} : \bar{j} \in m^n, 1 \leq t \leq k), \bar{a}) := \\
	\exists x_{n+1} \bigwedge_{\bar{j} \in m^n} \bigwedge_{1 \leq t \leq k} \varphi \left(a_1^{j_1}, \ldots, a_n^{j_{n}}, x_{n+1}, x^t_{(j_1, \ldots, j_n)} \right) \land \rho \left((x^t_{\bar{j}} : \bar{j} \in m^n, 1 \leq t \leq k) \right),
\end{gather*}
where $\rho$ is a formula expressing that all of the elements of the tuple $(x^t_{\bar{j}} : \bar{j} \in m^n, 1 \leq t \leq k)$ are pairwise-distinct.
By the first item in (8), for each $\bar{j} \in m^n$ we let $\bar{b}_{\bar{j}} =(b_{\bar{j}}^t : 1 \leq t \leq k)$ be a tuple of length $k$ enumerating the set $\varphi(a_1^{j_1}, \ldots, a_n^{j_{n}}, a_{n+1}, \mathbb{M})$ in an arbitrary order, and let $\bar{b} := (b^t_{\bar{j}} : \bar{j} \in m^n, 1 \leq t \leq k)$. Then we have:
\begin{itemize}
\item $\bar{b} \subseteq S$, by (5);
\item all elements of the tuple $\bar{b}$ are pairwise-distinct, by (9);
\item $\bar{b} \cap \acl(\bar{a}) = \emptyset$, by the second item in (8);
\item $\models \theta(\bar{b}, \bar{a})$, with $\exists x_{n+1}$ realized by $a_{n+1}$ by the choice of $\bar{b}$.
\end{itemize}
Hence, applying Fact \ref{fac: rand pred axioms}, there exists some $\bar{c} = (c^t_{\bar{j}} : \bar{j} \in m^n, 1 \leq t \leq k) \subseteq S$ such that 
\begin{itemize}
	\item $\models \theta(\bar{c},\bar{a})$;
	\item $\bar{c}_{\bar{j}} := (c^t_{\bar{j}} : 1 \leq t \leq k) \subseteq P$ for every $\bar{j} \in I$;
	\item $\bar{c}_{\bar{j}}  \cap P = \emptyset $ for every $\bar{j} \in m^n \setminus I$.
\end{itemize}
As $\models \theta(\bar{c},\bar{a})$, all elements of $\bar{c}$ are pairwise distinct as it realizes $\rho$, and there exists some  $a^{I}_{n+1} \in \mathbb{M}$ such that 
$$\models \bigwedge_{\bar{j} \in m^n} \bigwedge_{1 \leq t \leq k} \varphi \left(a_1^{j_1}, \ldots, a_n^{j_{n}}, a^I_{n+1}, c^t_{(j_1, \ldots, j_n)} \right).$$
In particular, for every $\bar{j} = (j_1, \ldots, j_n) \in m^n$, every element of the tuple $\bar{c}_{\bar{j}}$ of length $k$ is in the set 
$\varphi \left(a_1^{j_1}, \ldots, a_n^{j_{n}}, a^I_{n+1}, \mathbb{M} \right)$. Hence, by the second item in (7), the tuple $\bar{c}_{\bar{j}}$ lists all of the elements of the set $\varphi \left(a_1^{j_1}, \ldots, a_n^{j_{n}}, a^I_{n+1}, \mathbb{M} \right)$. By the choice of $\bar{c}$ it follows that for every $\bar{j} \in m^n$,
$$\models \psi(a_1^{j_1}, \ldots, a_n^{j_n}, a^I_{n+1}) \iff i \in I.$$
As $m$ and $I$ were arbitrary we conclude that $\psi$ is not $n$-dependent by compactness.

Finally, the ``in particular'' part of the proposition is immediate by Remark \ref{rem: triv iff disint}.
\end{proof}

\begin{remark}\label{rem: counterex to ChPil}
	The case of $n=1$ of Proposition \ref{non-disint gives IPn} is claimed in \cite[Proposition 2.10]{chatzidakis1998generic} without the assumption that $T$ is geometric. However, their proof contains a gap and the claim is false as witnessed by the following example.
	Let $T$ be the theory of the infinite branching tree, i.e.~the theory of an infinite graph $(G,R)$ such that
	\begin{enumerate}
		\item for every vertex $a \in G$ there are infinitely many $b$ such that $aRb$,
		\item there are no cycles.
	\end{enumerate}
	It is not hard to see by back-and-forth that $T$ is complete and admits quantifier elimination after adding distance predicates (which are definable in the graph language using quantifiers). Then $T_P$ is stable, e.g.~since by \cite[Theorem 1.4]{ivanov1993structure} every expansion of a planar graph by unary predicates is stable. However, $\acl$ is not disintegrated (for any $a \in G$ and two elements $b,c$ connected to it, we have that $a \in \dcl(bc)$, but $a \notin \acl(b) \cup \acl(c)$. Note that $\acl$ doesn't satisfy exchange in this example since $b \notin \acl(ac)$.
	
%	The same argument as in the proof of Proposition \ref{non-disint gives IPn} applies to any (not necessarily geometric) theory eliminating $\exists^\infty$ and such that $S$ contains an infinite definable subgroup.
\end{remark}

\begin{remark}
In a recent erratum \cite{ChaPilErr} to \cite{chatzidakis1998generic}, it is observed that $T_{P,S}$ is not dependent assuming the following stronger variant of the failure of disintegration of $\acl$ in $T$: there exist a small model $M \prec \mathbb{M} \models T$ and tuples $a,b$ such that $\acl(M,a,b) \cap S \subsetneq \acl(M,a) \cup \acl(M,b)$,  and moreover $\tp(a/Mb)$ is finitely satisfiable in $M$.
\end{remark}

Next we will show a converse to Proposition \ref{non-disint gives IPn}: if the algebraic closure in a geometric theory is disintegrated, then $n$-dependence is preserved after adding a generic predicate. More generally, we consider expansions by ``generic'' relations of arity at most $n$. 

The following is an analog of Fact \ref{fac: rand pred axioms} in this more general setting. It is essentially from \cite{winkler1975model}, though we refer to \cite{kruckman2018interpolative} here. Namely, let $T$ be a complete theory in a language $\mathcal{L}$ with a distinguished $\mathcal{L}(\emptyset)$-definable set $S(x)$ and let $\mathcal{L}' \supseteq \mathcal{L}$ be a language such that $\mathcal{L'} \setminus \mathcal{L} = \{R_i : i \in I\}$ only contains relational symbols. Then we let $\mathcal{L}_0 := \mathcal{L} \cup \{ S^*(x)\}$, where $S^*(x)$ is a new unary predicate symbol, and let $T_0$ be the (complete) $\mathcal{L}_0$-theory axiomatized by $T \cup \{S(x) \leftrightarrow S^*(x)\}$. Let $T_{\cap}$ be the reduct of $T_0$ to the language $\mathcal{L}_{\cap} := \{S^*(x)\}$ (so a complete theory of a unary predicate). And for $i \in I$, we let $\mathcal{L}_i := \{ S^*(x)\} \cup \{ R_i \}$ and let $T_i$ be the model companion of the $\mathcal{L}_i$-theory $\{\forall x_1 \ldots \forall x_n R_i(x_1, \ldots, x_n) \rightarrow \bigwedge_{1 \leq j \leq n} S^*(x_j) \}$ (which exists by \cite[Theorem 5]{winkler1975model}, or by \cite[Remark 2.12.2]{chatzidakis1998generic}). In the next fact, the existence of a model companion $T'$ of the theory $T_{\cup} := \bigcup_{i \in I} T_i$ is given by \cite[Theorem 5.50]{kruckman2018interpolative}, the description of types in $T'$ by
\cite[Proposition 6.11]{kruckman2018interpolative} and the description of the algebraic closure $\acl_{T'} = \acl_{T}$ by \cite[Theorem 6.3]{kruckman2018interpolative}).
\begin{fact}\label{fac: QE interpolative fusion}
Let $T$ be a theory in the language $\mathcal{L}$ eliminating quantifiers and $\exists^\infty$, and fix an $\mathcal{L}(\emptyset)$-definable predicate $S$. Let $\mathcal{L}' \supseteq \mathcal{L}$ be a language such that $\mathcal{L}' \setminus \mathcal{L}$ only contains relational symbols, and consider the $\mathcal{L}'$-theory 
$$T'_0 := T \cup \left\{ R(x_1, \ldots, x_n) \rightarrow \bigwedge_{1 \leq i \leq n} S(x_i) : R \in \mathcal{L}' \setminus \mathcal{L}\right\}.$$
 Then $T'_0$ admits a model companion $T'$, and working in a monster model of $T'$ we have the following: for any tuples $\bar{a},\bar{b}$ and a subset $A$,

\begin{tabular}{  m{33em}  m{1cm} } 	
$\tp_{\mathcal{L'}}(\bar{a}/A) = \tp_{\mathcal{L}'}(\bar{b}/A)$ if and only if there is an $A$-isomorphism of $\mathcal{L}'$-structures from $\acl_{\mathcal{L}}(A,\bar{a})$ to 
$\acl_{\mathcal{L}}(A,\bar{b})$ which carries $\bar{a}$ to $\bar{b}.$ &\ \  $(\dagger)$ 
\end{tabular}
\end{fact}
We sometimes refer to $T'$ as the expansion of $T$ by \emph{generic relations} in $\mathcal{L}' \setminus \mathcal{L}$.

In \cite[Lemma 2.1]{hrushovski1991pseudo} Hrushovski observes that the random $n$-ary hypergraph is not a finite Boolean combination of relations of arity $n-1$. In order to demonstrate preservation of $n$-dependence in expansions of disintegrated theories by generic $n$-ary relations we will use the following infinitary generalization of this fact.
\begin{prop}\label{prop: inf Hrush}
For each $n \in \omega, n \geq 1$ and an infinite cardinal $\kappa$ there exists some cardinal $\lambda \geq \kappa$ satisfying the following: Let $G'_{n,p}$ be a $\lambda$-saturated model of $\Th(G_{n,p})$, let $\tilde{\mathcal{L}}$ be an arbitrary relational language with $|\tilde{\mathcal{L}}| \leq \kappa$ containing only relations of arity at most $n-1$, and let $\tilde{ O}'_{n,p}$ be an expansion of $O'_{n,p}$ obtained by adding arbitrary interpretations for all the relations in $\tilde{\mathcal{L}}$. Then the following holds: 

\begin{tabular}{  m{33em}  m{1cm} } 	 there are  $g_i, h_i \in P_i^{G'_{n,p}}, 1 \leq i \leq n$, such that  $\qftp_{\tilde{\mathcal{L}}}(g_1, \ldots, g_n) = \qftp_{\tilde{\mathcal{L}}}(h_1, \ldots, h_n)$ and  $G'_{n,p} \models R_n(g_1, \ldots, g_n) \wedge \neg R_n(h_1, \ldots, h_n).$ &  $(\ast)$
	\end{tabular}
\end{prop}
\begin{proof}
	We show this result by induction on $n$ (the base case $n=1$ obviously holds with $\lambda := \kappa$). 
	Now fix $n \geq 2$ and a cardinal $\kappa$. Let $\lambda=\lambda_{n-1}$ satisfy the proposition for $n-1$ and $\kappa$. We will show that $\lambda=\lambda_n := \beth_{n-2}\left(2^{\lambda_{n-1}}\right)^+$ satisfies the proposition for $n$.
	
	Towards a contradiction, assume that some $\lambda_n$-saturated $G'_{n,p} \equiv G_{n,p}$, some language $|\tilde{\mathcal{L}}| \leq \kappa$ and some expansion $\tilde{ O}'_{n,p}$ do not satisfy $(*)$, i.e.~for all $g_i, h_i \in P_i^{G'_{n,p}}, 1 \leq i \leq n$, 
	
	\begin{tabular}{  m{33em}  m{1cm} } if	 $\qftp_{\tilde{\mathcal{L}}}(g_1, \ldots, g_{n}) = \qftp_{\tilde{\mathcal{L}}}(h_1, \ldots, h_{n}) \land G'_{n,p} \models R_{n}(g_1, \ldots, g_n)$ then\\ $G'_{n,p} \models R_{n}(h_1, \ldots, h_{n}).$&  $(\dagger)$
	\end{tabular}

	By the choice of $\lambda_n$ and Erd\H{o}s-Rado we have $\lambda_n \rightarrow \left( \left(2^{\lambda_{n-1}} \right)^+ \right)^{n-1}_{2^{\lambda_{n-1}
	}}$, hence we can find some sets $A_i \subseteq P_i^{G'_{n,p}}, 1 \leq i \leq n-1$ such that $|A_i| \geq \left(2^{\lambda_{n-1}}\right)^{+}$ and $\qftp_{\tilde{\mathcal{L}}}(g_1, \ldots, g_{n-1}) = \qftp_{\tilde{\mathcal{L}}}(h_1, \ldots, h_{n-1})$ for all $g_i, h_i \in A_i, 1 \leq i \leq n-1$. Next, we can find a $\lambda_{n-1}$-saturated structure $G'_{n-1,p} \equiv G_{n-1,p}$ with $|G'_{n-1,p}| \leq 2^{\lambda_{n-1}}$ and such that $P_i^{G'_{n-1,p}} \subseteq A_i, 1 \leq i \leq n-1$.
	As $G'_{n,p}$ is $\lambda_n$-saturated and $\lambda_n >2^{\lambda_{n-1}}$, by the axioms of $\Th(G_{n,p})$ there exists some $c \in P_{n}^{G'_{n,p}}$ such that for all $g_i \in P_i^{G'_{n-1,p}}, 1 \leq i \leq n-1$ we have
	$$G'_{n-1,p} \models R_{n-1}(g_1, \ldots, g_{n-1}) \iff G'_{n,p} \models R_{n}(g_1, \ldots, g_{n-1}, c).$$
	Without loss of generality, we may assume that all relations in $\tilde {\mathcal L}$ are of arity exactly $n-1$. We consider the language $\tilde{\mathcal{L}}_{n-1}$ containing, for each $F \in \tilde{\mathcal{L}}$ and 
	$ \Upsilon \in \mathcal P( \{1, \dots, n-1\}) \setminus \emptyset$, an  $(n-1- |\Upsilon|)$-ary relational symbol $F_\Upsilon$. We define an expansion $\tilde{O}'_{n-1,p}$ of $O'_{n-1,p}$ in which we interpret each such $F_\Upsilon \in \tilde{\mathcal{L}}_{n-1}$ as 
	$F$ with each of the variables $x_i, i \in \Upsilon$ fixed by $c$ and restricted to the universe of $G'_{n-1,p}$ (e.g.~if $\Upsilon= \{2,n-1\}$, then $F_{\Upsilon}$ is interpreted as $F(x_1, c, x_3, \dots, x_{n-2}, c) \cap \prod_{i \in \{1, 2, \ldots, n-1\} \setminus \Upsilon} P_i^{G'_{n-1,p}}$). Hence $|\tilde{\mathcal{L}}_{n-1}| \leq \kappa$ and all relations in $\tilde{\mathcal{L}}_{n-1}$ have arity  at most $n-2$. Note that by the choice of $A_i, 1\leq i\leq n-1$ we automatically have that for any $F \in \tilde{\mathcal{L}}$ and any $g_i,h_i \in P_i^{G'_{n-1,p}}, 1 \leq i \leq n-1$,
	$$\tilde{ O}'_{n,p} \models F(g_1, \ldots, g_{n-1}) \iff \tilde{ O}'_{n,p} \models F(h_1, \ldots, h_{n-1}).$$

	By the choice of the $F_{\Upsilon}$'s, we have that for any $g_i, h_i \in P_i^{G'_{n-1,p}}, 1 \leq i \leq n-1$,
	$$\qftp_{\tilde{\mathcal{L}}_{n-1}}(g_1, \ldots, g_{n-1}) = \qftp_{\tilde{\mathcal{L}}_{n-1}}(h_1, \ldots, h_{n-1}) \implies $$
	$$\qftp_{\tilde{\mathcal{L}}}(g_1, \ldots, g_{n-1},c) = \qftp_{\tilde{\mathcal{L}}}(h_1, \ldots, h_{n-1}, c). $$
	
	Since $(\dagger)$ holds for $\tilde{ O}'_{n,p}$, by the choice of $c$ this implies 
	$$\qftp_{\tilde{\mathcal{L}}_{n-1}}(g_1, \ldots, g_{n-1}) = \qftp_{\tilde{\mathcal{L}}_{n-1}}(h_1, \ldots, h_{n-1}) \land G'_{n-1,p} \models R_{n-1}(g_1, \ldots, g_{n-1}) \implies $$
	$$G'_{n-1,p} \models R_{n-1}(h_1, \ldots, h_{n-1}).$$

That is, $\tilde{ O}'_{n-1,p}$ and $\tilde{\mathcal{L}}_{n-1}$ satisfy $(\dagger)$ for $n-1$ --- contradicting the induction hypothesis.
\end{proof}
Using Proposition \ref{prop: inf Hrush}, we can prove preservation of $n$-dependence.

\begin{prop}\label{prop: Tp is ndep}
Let $T$ be a theory in the language $\mathcal{L}$ eliminating $\exists^\infty$ (not necessarily geometric), and assume that all elements in $S$ are trivial. Fix $n \geq 1$, and let $T'$ be a generic expansion of $T$ in a language $\mathcal{L}'$ such that $\mathcal{L}' \setminus \mathcal{L}$ only contains relational symbols of arity at most $n$ living on $S$. Then $T'$ is $n$-dependent if and only if $T$ is $n$-dependent.
\end{prop}
\begin{proof}
Assume that $T$ is $n$-dependent, but that there is some $\mathcal{L}'$-formula $\varphi(x;y_1, \ldots,y_n)$ which is not $n$-dependent in $T'$. Let $T^{\Sk}$ be a Skolemization of $T'$ with a distinguished constant symbol $0$, in the language $\mathcal{L}^{\Sk} \supseteq \mathcal{L}', |\mathcal{L}^{\Sk}| = |\mathcal{L}'|$. Let $\kappa := |\mathcal{L}'|$, and let $\lambda$ be as given by Proposition \ref{prop: inf Hrush} for $n$ and $\kappa$.

Let $G'_{n,p}$ be a $\lambda$-saturated model of $\Th(G_{n,p})$. Working in a monster model $\mathbb{M}$ of $T^{\Sk}$ (which we may assume to be $|G'_{n,p}|^+$-saturated in particular), $\varphi(x;y_1, \ldots,y_n)$ is still not $n$-dependent, hence by Proposition \ref{prop: indisc witness to IPn}(3) there exist tuples $(a_g)_{g \in G'_{n,p}}$ and $b$ such that:
\begin{enumerate}
	\item $(a_g)_{g \in G'_{n,p}}$ is $O'_{n,p}$-indiscernible over $\emptyset$ and $G'_{n,p}$-indiscernible over $b$, both in the sense of $T^{\Sk}$;
	\item $\models \varphi(b;a_{g_1}, \ldots, a_{g_{n}}) \iff G'_{n,p}\models R_n(g_1, \ldots, g_n)$, for all $g_i 
	\in P_i$.
\end{enumerate}

First we would like to replace each $a_g$ and $b$ by their algebraic closures. In order to preserve indiscernibility, we have to enumerate these algebraic closures in a coherent manner, which can be done as follows.
Let $\left(f_{\alpha}(\bar{y}_{\alpha}) : \alpha \in \kappa  \right)$ be an arbitrary enumeration of all $\mathcal{L}^{\Sk}(\emptyset)$-definable functions. Given an arbitrary tuple $c = (c_j)_{ j \in I}$ in $\mathbb{M}$, we let $J_{c} := \kappa \times I^{< \omega}$ be ordered lexicographically with respect to the ordering on the ordinal $\kappa$ and the ordering on $I$. Then we consider the tuple $\Sk(c) = \left( \Sk_{j}(c) : j \in J_c \right)$, where for $j = (\alpha,\beta) \in \kappa \times I^{< \omega}$, $\Sk_j(c)$ is $f_{\alpha} \left(\left( c_t : t \in \beta \right) \right)$ (or $0$ if the sort of the parameter doesn't fit the sort of the variables of the function).

Now for each $1 \leq i \leq n$ we fix an arbitrary $g_i \in P_i^{G'_{n,p}}$. The set of elements appearing in the tuple $\Sk(a_{g_i})$ is an elementary $\mathcal{L}^{\Sk}$-submodel of $\mathbb{M}$ by Tarski-Vaught. In particular, we can choose some $J_i \subseteq J_{a_{g_i}}$ so that the tuple $a'_{g_i} := (\Sk_{j}(a_{g_i}) : j \in J_i)$ enumerates $\acl_{\mathcal{L}}(a_{g_i})$ without repeated elements. Now for an arbitrary $1 \leq i \leq n$ and $g \in P_i^{G'_{n,p}}$, we let $a'_g := (\Sk_{j}(a_{g}) : j \in J_i)$. Similarly, we choose some $J_0 \subseteq J_b$ so that the tuple $b' := (\Sk_{j}(b) : j \in J_0)$ lists $\acl_{\mathcal{L}}(b)$ without repeated elements. Since we have both indiscernibilities in (1) in the sense of $T^{\Sk}$, it follows that:
\begin{enumerate}
  \setcounter{enumi}{2}
  \item $b'$ enumerates $\acl_{\mathcal{L}}(b)$ (by definition);
  \item for any $g \in G'_{n,p}$, the tuple $a'_g$ enumerates $\acl_{\mathcal{L}}(a_g)$ without repetitions (as $g \in P_i^{G'_{n,p}} \implies a_g \equiv^{\mathcal{L}^{\Sk}} a_{g_i}$ by (1) $\implies a'_g \equiv^{\mathcal{L}} a'_{g_i}$ );
  \item $|b'|, |a'_g| \leq \kappa$;
  \item $(a'_g)_{g \in G'_{n,p}}$ is $O'_{n,p}$-indiscernible over $\emptyset$ and $G'_{n,p}$-indiscernible over $b'$, both in the sense of $T^{\Sk}$ (and hence in the sense of $T$ as well; again by (1) and definition of $a'_g$ and $b'$).
\end{enumerate}
 As $T$ is $n$-dependent, it follows from (6) by Proposition \ref{prop: char of NIP_k by preserving indisc}(3) that 
 \begin{enumerate}
   \setcounter{enumi}{6}
 	\item $(a'_g)_{g \in G'_{n,p}}$ is $O'_{n,p}$-indiscernible over $b'$ in the sense of $T$.
 \end{enumerate}
	
	Next we consider the $\mathcal{L}'$-isomorphism type of the structure induced on the tuple $b'a'_{g_1} \ldots a'_{g_n}$ as $\bar{g}$ varies, and demonstrate that it cannot reflect exactly the hyper-edge relation $R_n$ of $G'_{n,p}$. Without loss of generality we may assume that all relations in $\mathcal{L}' \setminus \mathcal{L}$ are of arity exactly $n$. By $O'_{n,p}$-indiscernibility in $T'$ in (6), for any $F \in \mathcal{L}'\setminus \mathcal{L}$ we have 
	\begin{enumerate}
	 \setcounter{enumi}{7}
		\item $\models F(a_{g_1}, \ldots, a_{g_n}) \iff \models F(a_{h_1}, \ldots, a_{h_n})$ for all $g_i,h_i \in P_i^{G'_{n,p}}, 1 \le i\leq n$.
	\end{enumerate}
	
	We consider an expansion $\tilde{G}'_{n,p}$ of $G'_{n,p}$ where for each $F(x_1, \ldots, x_{n}) \in \mathcal{L}'\setminus \mathcal{L}$, each $1 \leq i \leq n$ and each $j_t \in J_t, t \in \{0, \ldots, n \} \setminus \{i\}$ we add a new $(n-1)$-ary relation $R_{F,i,\bar{j}} \subseteq \prod_{t \in \{1, \ldots, n\}\setminus \{i\}} P_t^{G'_{n,p}}$ defined as follows: for any $(g_t)_{t \in \{1, \ldots, n \} \setminus \{i\}} \in \prod_{t \in \{1, \ldots, n\}\setminus \{i\}} P_t^{G'_{n,p}}$ we have
	$$(g_t)_{t \in \{1, \ldots, n \} \setminus \{i\}} \in R_{F,i,\bar{j}} ~ :\iff ~ \models F(a'_{g_1,j_1}, \ldots, a'_{g_{i-1},j_{i-1}}, b'_{j_0},a'_{g_{i+1},j_{i+1}}, \ldots, a'_{g_{n},j_{n}}),$$
	where $a'_{g,j}$ is the $j$th element of the tuple $a'_g$ (i.e.~$a'_{g,j} = \Sk_{j}(a_g)$).

	Note that taking 
	$$\tilde{\mathcal{L}} := \{R_{F,i,\bar{j}}: F \in \mathcal{L}' \setminus \mathcal{L}, 1 \leq i \leq n, j_t \in J_t  \textrm{ for } 0\leq t \leq n\} \cup \mathcal{L}$$
	 we have $|\tilde{\mathcal{L}}| \leq \kappa$ by (5). Then, by Proposition \ref{prop: inf Hrush} and the choice of $\lambda$, there exist some $g_i, h_i \in P_i^{G'_{n,p}}$ for $1 \leq i \leq n$ such that: 
	\begin{enumerate}
	 \setcounter{enumi}{8}
		\item $\qftp_{\tilde{\mathcal{L}}}(g_1, \ldots, g_n) = \qftp_{\tilde{\mathcal{L}}}(h_1, \ldots, h_n)$,
		\item $G'_{n,p} \models R_n(g_1, \ldots, g_n)$,
		\item $G'_{n,p} \models \neg R_n(h_1, \ldots, h_n)$.
	\end{enumerate}
	
	By (7) $a'_{g_1} \ldots a'_{g_n} \equiv^{\mathcal{L}}_{b'} a'_{h_1} \ldots a'_{h_n}$. Taking any $\mathcal{L}$-automorphism $\sigma$ of $\mathbb{M}$ sending the tuple on the left-hand side to the right-hand side over $b'$ we have:
	\begin{enumerate}
		 \setcounter{enumi}{11}
		\item $\sigma$ fixes $b'$ pointwise, and $\sigma(a'_{g_i}) = a'_{h_i}$ (preserving the ordering of the tuples) for all $1 \leq i \leq n$;
		\item $\sigma \left(\acl_{\mathcal{L}}(b'a'_{g_1} \ldots a'_{g_n}) \right) = \acl_{\mathcal{L}}(b'a'_{h_1} \ldots a'_{h_n})$ setwise;
		\item $\sigma \left(\acl_{\mathcal{L}}(b'a'_{g_1} \ldots a'_{g_n}) \cap S \right) = \acl_{\mathcal{L}}(b'a'_{h_1} \ldots a'_{h_n}) \cap S$ setwise. 
	\end{enumerate}
	
	Let now an  element  $e \in \acl_{\mathcal{L}}(b'a'_{g_1} \ldots a'_{g_n}) \cap S$ be arbitrary. As all elements of $S$ are trivial (in $T$) by assumption, we have $e \in \acl_{\mathcal{L}}(b') = b'$ or $e \in \acl_{\mathcal{L}}(a'_{g_i}) = a'_{g_i}$ for some $1 \leq i \leq n$. Hence $\acl_{\mathcal{L}}(b'a'_{g_1} \ldots a'_{g_n}) \cap S \subseteq b'a'_{g_1} \ldots a'_{g_n}$, and so also $\acl_{\mathcal{L}}(b'a'_{h_1} \ldots a'_{h_n}) \cap S \subseteq b'a'_{h_1} \ldots a'_{h_n}$. Given any tuple $c$ of elements from $\acl_{\mathcal{L}}(b'a'_{g_1} \ldots a'_{g_n})$, by definition of $T'$ it can satisfy a relation $F \in \mathcal{L}'\setminus \mathcal{L}$ only if $c$ is entirely contained in $S$, hence  only if $c$ is contained in the tuple $b'a'_{g_1} \ldots a'_{g_n}$.
But by (8) and (9) (unwinding the definition of $\tilde{\mathcal{L}}$ and using (12)), for any tuple $c$ in $b'a'_{g_1} \ldots a'_{g_n}$ and $F \in \mathcal{L}'\setminus \mathcal{L}$, $c$ satisfies $F$ if and only if $\sigma(c)$ satisfies $F$. And of course $c$ and $\sigma(c)$ agree on all $\mathcal{L}$-formulas. We thus conclude that $\sigma \restriction_{\acl_{\mathcal{L}}(b'a'_{g_1} \ldots a'_{g_n}) }$ is an automorphism of $\mathcal{L}'$-structures $\acl_{\mathcal{L}}(b'a'_{g_1} \ldots a'_{g_n})$ and $\acl_{\mathcal{L}}(b'a'_{h_1} \ldots a'_{h_n})$ sending $b'a'_{g_1}\ldots a'_{g_n}$ to $b'a'_{h_1}\ldots a'_{h_n}$.
		
 Hence by Fact \ref{fac: QE interpolative fusion}$(\dagger)$ we have 
 $$\tp_{\mathcal{L'}}(a_{g_1} \ldots a_{g_n}/b) = \tp_{\mathcal{L}'}(a_{h_1} \ldots a_{h_n}/b).$$
  But this contradicts the choice of $\varphi$ in view of (10) and (11).
	
	Finally, the converse implication of the proposition is obvious.
\end{proof}

Combining Propositions \ref{non-disint gives IPn} and \ref{prop: Tp is ndep} we thus have the following ``baby case'' of the relationship of the collapse of $n$-dependence to dependence and complicated geometry of algebraic closure that we expect to happen for fields.

\begin{cor}\label{cor: n-dep iff acl disint}
	Let $T$ be a geometric theory. The following are equivalent:
	\begin{enumerate}
		\item $T_P$ is dependent.
		\item $T_P$ is $n$-dependent for some $n\in \omega$.
		\item $T$ has disintegrated algebraic closure.
	\end{enumerate}
\end{cor}

\begin{problem}
	It should be possible to generalize these results on preservation of $n$-dependence to interpolative fusions of theories as studied in \cite{kruckman2018interpolative}.
\end{problem}

\appendix
\section{An explicit isomorphism in Kaplan-Scanlon-Wagner,\\ by Martin Bays}\label{sec: app Bays}

Let $k$ be a perfect field of characteristic $p>0$.
Let $\phi$ be the Frobenius automorphism, $\phi(x) := x^p$,
and let $\as$ be the Artin-Schreier map, $\as(x) := \phi(x) - x = x^p-x$.
Let $\a = (a_0,\ldots ,a_m) \in k^{m+1}$.
Let $G_{\a} := \{ \x \;|\; a_0\as(x_0) = \ldots  = a_m\as(x_m) \}$, considered as 
an algebraic subgroup over $k$ of the Cartesian power of the additive group 
$\G_a^{m+1}$.
A crucial step in the proof in \cite{kaplan2011artin} of Artin-Schreier closedness of dependent 
fields is to show that if $\a$ is an algebraically independent tuple, i.e.\ 
$\trd(\F_p(\a)/\F_p)=m+1$, then $G_{\a}$ is isomorphic over $k$ to the 
additive group, as algebraic groups. Hempel \cite{hempel2016n} improves this by 
showing that the same holds when the assumption is weakened to $\F_p$-linear 
independence of $(a_0^{-1},\ldots ,a_m^{-1})$. In both cases, the proof is rather 
indirect, going via showing that $G_{\a}$ is connected and then referring to 
some standard theorems characterising vector groups in positive 
characteristic.
The purpose of this appendix is to exhibit such an isomorphism.
Thanks to Mohammed Bardestani and Pierre Touchard for helpful discussion.

First we need the following fact about the Moore matrix, being the analogue of 
the Wronskian matrix with Frobenius in place of differentiation. For 
completeness, we include a proof.

\begin{fact} \label{f:moore}
    Let $k$ be a perfect field of characteristic $p>0$.
    Let $\c = (c_0,\ldots ,c_m) \in k^{m+1}$.
    Then the Moore matrix
    $M:=(\phi^i(c_j))_{0 \leq  i,j \leq  m}$ is singular if and only if $\c$ is
    $\F_p$-linearly dependent.
\end{fact}
\begin{proof}
    Suppose $\c$ is $\F_p$-linearly dependent, say $\sum_{j=0}^m \lambda_jc_j 
    = 0$ with $\lambdatup \in \F_p^{m+1} \setminus \{\tuple{0}\}$.
    Then $\sum_{j=0}^m \lambda_j\phi^i(c_j) = 0$ for all $i$, so $M$ is 
    singular.

    The converse is clear for $m=0$.
    So suppose $m\geq 1$,
    and $\c$ is $\F_p$-linearly independent
    but $M$ is singular,
    say $\bigwedge_{0\leq i\leq m}\sum_{j\geq 0} \mu_j \phi^i(c_j) = 0$
    with $\bar \mu  \in k \setminus \{\tuple{0}\}$.
    Since $(c_1,\ldots ,c_m)$ is $\F_p$-linearly independent, we may inductively 
    assume that the $m\times m$ matrix $M' := (\phi^i(c_j))_{0 < i,j \leq m}$ is 
    non-singular.
    It follows that $\mu_0 \neq  0$. Now, let $\alpha := \frac{\phi(\mu_0)}{\mu_0} \neq  0$. Then for $0< i \leq m$ we have  
    \begin{align*}
    	\phi(\mu_0) \phi^i(c_0) + \sum_{j\geq 1}\phi(\mu_j) \phi^i(c_j) &=\sum_{j\geq 0}\phi(\mu_j) \phi^i(c_j)\\
    	&=   \phi \left(\sum_{j\geq 0} \mu_j \phi^{i-1}(c_j)\right) \\
    	&= 0 \\
    	&= \alpha \cdot \sum_{j\geq 0} \mu_j \phi^i(c_j)\\
    	&=  \phi(\mu_0) \phi^i(c_0)+  \sum_{j\geq 1} \alpha \mu_j \phi^i(c_j).
    \end{align*}
Consequently $\sum_{j\geq 1}\phi(\mu_j) \phi^i(c_j)=  \sum_{j\geq 1} \alpha \mu_j \phi^i(c_j)$.
     By non-singularity of $M'$, we deduce that
    $\bigwedge_{0\leq j\leq m} \phi(\mu_j) = \alpha\mu_j$.
    Let $\beta$ (in an extension of $k$) be such that 
    $(\beta)^{p-1}=\alpha$ and set $\lambda_j := \frac{\mu_j}{\beta}$. Then either $\lambda_j=\mu_j=0$ or
    $$ \lambda_j^{p-1}= \frac{\mu_j^{p-1}}{\beta^{p-1}} = \frac{\phi(\mu_j)}{\alpha \mu_j} = 1. $$
   Hence $\lambdatup_j \in \F_p^{m+1} \setminus \{\tuple{0}\}$.
    But $\sum_{j\geq 0} \lambda_j c_j = \frac1\beta \sum_{j\geq 0} \mu_j c_j = 
    0$, contradicting $\F_p$-linear independence of $\c$.
\end{proof}

Now let $\a = (a_0,\ldots ,a_m) \in k^{m+1}$,
and suppose $\b := (a_0^{-1},\ldots ,a_m^{-1})$ is $\F_p$-linearly independent.

Write $\delta_{i,j}$ for the Kronecker delta.
By Fact~\ref{f:moore} applied to $\phi^{-m}(\b)$,
$(\phi^{-i}(b_j))_{0\leq i,j\leq m}$ is non-singular.
Since $k$ is perfect, $\phi^{-i}(b_j) \in k$.
So there exists $\alphatup = (\alpha_0,\ldots ,\alpha_m) \in k^{m+1} \setminus \{\bar 0\}$ such that
$\bigwedge_{0\leq i\leq m} \sum_{j\geq 0}\phi^{-i}(b_j)\alpha_j = \delta_{0,i}$.

\begin{claim} \label{c:alphaIndep}
    $\alphatup$ is $\F_p$-linearly independent.
\end{claim}
\begin{proof}
    Suppose not, so (permuting if necessary) we have $\alpha_0 =
    \sum_{j\geq 1}\lambda_j\alpha_j$ with $\lambda_j \in \F_p$.
    Then for $1\leq i\leq m$, we have
    $\phi^{-i}(b_0)\sum_{j\geq 1}\lambda_j\alpha_j + 
    \sum_{j\geq 1}\phi^{-i}(b_j)\alpha_j = \delta_{0,i} = 0$,
    and so $\sum_{j\geq 1} \alpha_j \phi^{-i}(b_j + \lambda_jb_0) = 0$.
    
    But $\alpha_j \neq  0$ for some $j\geq 1$, since $\alphatup \neq  \tuple{0}$.
    So by Fact~\ref{f:moore}, $(b_j+\lambda_jb_0)_{j\geq 1}$ is $\F_p$-linearly
    dependent and consequently so is $\b$, contrary to assumption.
\end{proof}

We proceed to define an algebraic isomorphism over $k$ of $G_{\a}$ with the 
additive group.
So let $K \geq  k$ be an arbitrary field extension, and let $\x \in G_{\a}(K)$, 
i.e.\ $a_0\as(x_0)=\ldots =a_m\as(x_m)$.
Set $$t := \sum_{j\geq 0} \alpha_j x_j.$$
\begin{claim} \label{c:tFrob}
    For $i\geq 0$, we have $\phi^i(t) = \sum_{j\geq 0} \phi^i(\alpha_j) x_j$.
\end{claim}
\begin{proof}
    This holds by definition for $i=0$.
    For $i>0$ we have
    $$\sum_{j\geq 0}\frac{\phi^i(\alpha_j)}{a_j} =
    \phi^{i}\left(\sum_{j\geq 0}\phi^{-i}(b_j)\alpha_j\right) = \phi^{i}(\delta_{0,i}) = 
    0.$$
    Using this, induction, and the equations of $G_{\a}$, we find
    \begin{align*} \phi^{i}(t)
    &= \phi\left(\sum_{j\geq 0} \phi^{i-1}(\alpha_j) x_j\right) \\
    &= \sum_{j\geq 0} \phi^{i}(\alpha_j) \phi(x_j) \\
    &= \sum_{j\geq 0} \phi^{i}(\alpha_j) \left(\as(x_j) + x_j\right) \\
    &= \sum_{j\geq 0} \frac{\phi^{i}(\alpha_j)}{a_j} a_j\as(x_j) +
    \sum_{j\geq 0} \phi^{i}(\alpha_j)x_j \\
    &= \sum_{j\geq 0} \phi^{i}(\alpha_j)x_j .\end{align*}
\end{proof}
Now by Claim~\ref{c:alphaIndep} and Fact~\ref{f:moore},
the matrix $(\phi^i(\alpha_j))_{0\leq i,j\leq m}$ is non-singular,
so say $(\beta_{ij})_{0\leq i,j\leq m}$ is the inverse, where $\beta_{ij} \in k$.

Then by Claim~\ref{c:tFrob}, $x_i = \sum_{j\geq 0} \beta_{ij}\phi^j(t)$.

So we have defined an isomorphism over $k$ of affine varieties
\[
\begin{array}{ccc} G_{\a} &\rightarrow & \G_a\\
  \x &\mapsto  &\sum_{j\geq 0} \alpha_j x_j \\
  (\sum_{j\geq 0} \beta_{ij}\phi^{j}(t))_i & \mapsfrom  &t
\end{array} \]
between $G_{\a}$ and the additive group $\G_a$; since the polynomials involved 
are additive polynomials, this is an isomorphism of algebraic groups.

\bibliographystyle{plain}
\bibliography{ref}

\end{document}